\documentclass[11pt,a4paper]{amsart}
\pdfoutput=1
\usepackage[centering,margin=1in]{geometry}
\usepackage{amsmath,amsfonts,amsthm,amssymb,enumitem,graphicx,mathrsfs,mathtools}
\usepackage[stretch=10]{microtype}
\usepackage[unicode,breaklinks=true]{hyperref}
\usepackage{xcolor}
\definecolor{dark-red}{rgb}{0.4,0.15,0.15}
\definecolor{dark-blue}{rgb}{0.15,0.15,0.4}
\definecolor{medium-blue}{rgb}{0,0,0.5}
\hypersetup{
	pdftitle={Archimedean Newform Theory for \texorpdfstring{$\GL_n$}{GL\9040\231}},
	pdfauthor={Peter Humphries},
	pdfnewwindow=true,
	colorlinks,
	linkcolor={dark-red},
	citecolor={dark-blue},
	urlcolor={medium-blue}
}
\makeatletter
\newcommand{\bigboxplus}{
\mathop{
	\vphantom{\bigoplus} 
	\mathchoice
	{\vcenter{\hbox{\resizebox{\widthof{$\displaystyle\bigoplus$}}{!}{$\boxplus$}}}}
	{\vcenter{\hbox{\resizebox{\widthof{$\bigoplus$}}{!}{$\boxplus$}}}}
	{\vcenter{\hbox{\resizebox{\widthof{$\scriptstyle\oplus$}}{!}{$\boxplus$}}}}
	{\vcenter{\hbox{\resizebox{\widthof{$\scriptscriptstyle\oplus$}}{!}{$\boxplus$}}}}
}\displaylimits 
}
\newcommand{\bigboxtimes}{
\mathop{
	\vphantom{\bigotimes} 
	\mathchoice
	{\vcenter{\hbox{\resizebox{\widthof{$\displaystyle\bigotimes$}}{!}{$\boxtimes$}}}}
	{\vcenter{\hbox{\resizebox{\widthof{$\bigotimes$}}{!}{$\boxtimes$}}}}
	{\vcenter{\hbox{\resizebox{\widthof{$\scriptstyle\otimes$}}{!}{$\boxtimes$}}}}
	{\vcenter{\hbox{\resizebox{\widthof{$\scriptscriptstyle\otimes$}}{!}{$\boxtimes$}}}}
}\displaylimits 
}
\makeatother
\allowdisplaybreaks
\newcommand{\A}{\mathbb{A}}
\newcommand{\Agp}{\mathrm{A}}
\newcommand{\AI}{\mathcal{AI}}
\newcommand{\C}{\mathbb{C}}
\newcommand{\dee}{\partial}
\newcommand{\df}{\mathfrak{d}}
\newcommand{\e}{\varepsilon}
\newcommand{\F}{\mathbb{F}}
\newcommand{\Hb}{\mathbb{H}}
\newcommand{\HH}{\mathcal{H}}
\newcommand{\Mgp}{\mathrm{M}}
\newcommand{\MM}{\mathcal{M}}
\newcommand{\N}{\mathbb{N}}
\newcommand{\Ngp}{\mathrm{N}}
\newcommand{\OO}{\mathcal{O}}
\newcommand{\Pgp}{\mathrm{P}}
\newcommand{\pp}{\mathfrak{p}}
\newcommand{\Q}{\mathbb{Q}}
\newcommand{\qq}{\mathfrak{q}}
\newcommand{\R}{\mathbb{R}}
\newcommand{\Ss}{\mathscr{S}}
\newcommand{\V}{\mathrm{V}}
\newcommand{\WW}{\mathcal{W}}
\newcommand{\Z}{\mathbb{Z}}
\newcommand{\Zgp}{\mathrm{Z}}
\DeclareMathOperator{\antidiag}{antidiag}
\DeclareMathOperator{\adj}{adj}
\DeclareMathOperator{\blockdiag}{blockdiag}
\DeclareMathOperator{\diag}{diag}
\DeclareMathOperator{\GL}{GL}
\DeclareMathOperator{\Hom}{Hom}
\DeclareMathOperator{\Ind}{Ind}
\DeclareMathOperator{\Mat}{Mat}
\DeclareMathOperator{\Norm}{N}
\DeclareMathOperator{\Ogp}{O}
\DeclareMathOperator{\sgn}{sgn}
\DeclareMathOperator{\SL}{SL}
\DeclareMathOperator{\SO}{SO}
\DeclareMathOperator{\Tr}{Tr}
\DeclareMathOperator*{\vol}{vol}
\DeclareMathOperator{\Ugp}{U}
\numberwithin{equation}{section}
\newtheorem{theorem}[equation]{Theorem}

\newtheorem{corollary}[equation]{Corollary}
\newtheorem{lemma}[equation]{Lemma}
\newtheorem{proposition}[equation]{Proposition}
\newtheorem*{problem}{Test Vector Problem}
\theoremstyle{remark}
\newtheorem{remark}[equation]{Remark}
\newtheorem{example}[equation]{Example}
\theoremstyle{definition}
\newtheorem{definition}[equation]{Definition}
\begin{document}

\title[Archimedean Newform Theory for $\GL_n$]{Archimedean Newform Theory for $\GL_n$}

\author{Peter Humphries}

\address{Department of Mathematics, University of Virginia, Charlottesville, VA 22904, USA}

\email{\href{mailto:pclhumphries@gmail.com}{pclhumphries@gmail.com}}

\urladdr{\href{https://sites.google.com/view/peterhumphries/}{https://sites.google.com/view/peterhumphries/}}



\thanks{Research partially supported by the European Research Council grant agreement 670239.}

\begin{abstract}
We introduce a new invariant, the conductor exponent, of a generic irreducible Casselman--Wallach representation of $\GL_n(F)$, where $F$ is an archimedean local field, that quantifies the extent to which this representation may be ramified. We also determine a distinguished vector, the newform, occurring with multiplicity one in this representation, with the complexity of this vector measured in a natural way by the conductor exponent. Finally, we show that the newform is a test vector for $\GL_n \times \GL_n$ and $\GL_n \times \GL_{n - 1}$ Rankin--Selberg integrals when the second representation is unramified. This theory parallels an analogous nonarchimedean theory due to Jacquet, Piatetski-Shapiro, and Shalika; combined, this completes a global theory of newforms for automorphic representations of $\GL_n$ over number fields. By-products of the proofs include new proofs of Stade's formul\ae{} and a new resolution of the test vector problem for archimedean Godement--Jacquet zeta integrals.
\end{abstract}

\maketitle

\section{Introduction}

Let $\MM_k(q,\chi)$ denote the finite-dimensional vector space of holomorphic modular forms of weight $k$, level $q$, and nebentypus $\chi$, where $\chi$ is a primitive Dirichlet character of conductor $q_{\chi} \mid q$. The classical theory of newforms due to Atkin and Lehner \cite{AL70} states that for each $q' \mid q$ with $q' \neq q$ and $q' \equiv 0 \pmod{q_{\chi}}$ and for each $\ell \mid \frac{q}{q'}$, the function $(\iota_{\ell} f)(z) \coloneqq f(\ell z)$ defines an element of $\MM_k(q,\chi)$ whenever $f \in \MM_k(q',\chi)$. We call $\iota_{\ell} f$ an oldform. Moreover, the orthogonal complement with respect to the Petersson inner product of the vector subspace of oldforms has an orthonormal basis consisting of newforms, which are eigenfunctions of the $n$-th Hecke operator not just for each positive integer $n$ for which $(n,q) = 1$ but for all $n \in \N$.

Casselman \cite{Cas73}, building on the seminal work of Jacquet and Langlands \cite{JL70}, gave an ad\`{e}lic reformulation of the Atkin--Lehner theory of newforms. Due to the fact that automorphic representations $\pi$ of $\GL_2(\A_{\Q})$ have a tensor product factorisation in terms of representations of $\GL_2(\R)$ and $\GL_2(\Q_p)$ for each prime $p$, this reformulation is purely local and is in terms of distinguished vectors in certain classes of representations of $\GL_2(\Q_p)$ determined in terms of congruence subgroups. Such a theory of newforms has been extended to the setting of generic irreducible admissible smooth representations of $\GL_n(F)$, where $F$ is a nonarchimedean local field \cite{JP-SS81}. This allows one to generalise Atkin--Lehner theory to automorphic representations of $\GL_n(\A_E)$ for any number field $E$. Furthermore, this theory has blossomed to include a well-developed theory of oldforms and conductor exponents associated to such representations, coupled with the local theory of test vectors for certain families of $\GL_n \times \GL_m$ Rankin--Selberg integrals \cite{Jac12,JP-SS81,Mat13,Ree91}.

While this nonarchimedean aspect of the ad\`{e}lic theory of newforms has been well understood since the seminal work of Jacquet, Piatetski-Shapiro, and Shalika \cite{JP-SS81}, there has been little development --- indeed, little even in the way of a conjectural formulation --- of the corresponding archimedean theory beyond the case $n = 2$. For $n = 2$, the archimedean aspect of the ad\`{e}lic theory of newforms, due to Popa \cite{Pop08}, concerns distinguished vectors in certain classes of representations of $\GL_2(\R)$ and $\GL_2(\C)$ determined in terms of the restrictions to the maximal compact subgroups $\Ogp(2)$ and $\Ugp(2)$. In the classical language of holomorphic modular forms and Maa\ss{} forms on the upper half-plane, this translates to prescribed automorphic forms determined in terms of Maa\ss{} raising and lowering operators, which raise and lower the weight of the automorphic form.

In this article, we put the archimedean setting on an equal footing with the nonarchimedean setting by developing such a theory for $\GL_n(\R)$ and $\GL_n(\C)$. We summarise our main results concerning generic irreducible Casselman--Wallach representations $\pi$ of $\GL_n$ as follows; precise statements are given in more detail in \hyperref[sect:archimedean]{Section \ref*{sect:archimedean}}.
\begin{itemize}
\item Among the $K_n$-types $\tau$ of $\pi$ whose restriction to $K_{n - 1}$ contains the trivial representation, there exists a unique $K_n$-type $\tau^{\circ}$ of lowest Howe degree, which occurs with multiplicity one in $\pi$; moreover, the subspace of $K_{n - 1}$-invariant $\tau^{\circ}$-isotypic vectors in $\pi$ is one-dimensional (\hyperref[thm:conductornewform]{Theorem \ref*{thm:conductornewform}}). We call the distinguished nonzero vector lying in this subspace, unique up to scalar multiplication, the newform of $\pi$, and we define the conductor exponent $c(\pi)$ of $\pi$ to be the Howe degree of $\tau^{\circ}$ (\hyperref[def:archcondnewfom]{Definition \ref*{def:archcondnewfom}}).
\item For each nonnegative integer $m \geq c(\pi)$, the dimension of the subspace of $K_{n - 1}$-invariant vectors in $\pi$ that are $\tau$-isotypic for some $K_n$-type $\tau$ of Howe degree $m$ is given explicitly as a certain binomial coefficient (\hyperref[thm:oldforms]{Theorem \ref*{thm:oldforms}}). We call vectors of this form oldforms.
\item The epsilon factor $\e(s,\pi,\psi)$ of $\pi$ is equal to $i^{-c(\pi)}$ (\hyperref[thm:epsilonconductor]{Theorem \ref*{thm:epsilonconductor}}); additionally, the conductor exponent $c(\pi)$ of $\pi$ is additive with respect to isobaric sums of representations (\hyperref[thm:additiveconductor]{Theorem \ref*{thm:additiveconductor}}) and is inductive (\hyperref[thm:inductiveconductor]{Theorem \ref*{thm:inductiveconductor}}).
\item When viewed in the Whittaker model and appropriately normalised, the newform is a test vector for the $\GL_n \times \GL_{n - 1}$ (\hyperref[thm:testvector]{Theorem \ref*{thm:testvector}}) and $\GL_n \times \GL_n$ (\hyperref[thm:GLnxGLn]{Theorem \ref*{thm:GLnxGLn}}) Rankin--Selberg integrals whenever the second representation is spherical.
\item The newform is a test vector for the Godement--Jacquet zeta integral (\hyperref[thm:GJram]{Theorem \ref*{thm:GJram}}).
\end{itemize}
As we explain in \hyperref[sect:nonarchimedean]{Sections \ref*{sect:nonarchimedean}} and \ref{sect:archimedean}, each of these results parallels an analogous result in the nonarchimedean setting.

Some aspects of this archimedean theory have long been expected. In particular, there has been a flurry of work in recent years studying the problem of finding test vectors for local $\GL_n \times \GL_m$ Rankin--Selberg integrals involving ramified representations over archimedean fields; see, for example, \cite{HIM12,HIM16,HIM22,IM22,Miy18,Pop08}. With the exception of \cite{Pop08}, previous work has invariably involved choosing test vectors that are associated in some way to the minimal $K$-type (in the sense of Vogan \cite{Vog81}), whereas we propose a theory of newforms without such a direct relation to minimal $K$-types. Apart from the recent work \cite{IM22}, these results have also been confined to low rank, namely $n \leq 3$ and $m \leq 2$.

A key part of this archimedean theory is the introduction of a new invariant, a nonnegative integer that we call the conductor exponent, associated to a generic irreducible Casselman--Wallach representation of $\GL_n(F)$ that in a certain sense quantifies the extent of ramification of such a representation. Somewhat surprisingly, there seems to have been no previous considerations in the literature of such a theory of the conductor exponent, in spite of the fact that it has many properties that mirror those of the nonarchimedean conductor exponent introduced by Jacquet, Piatetski-Shapiro, and Shalika \cite{JP-SS81}.

The structure of this article is the following. \hyperref[sect:prelim]{Section \ref*{sect:prelim}} contains a brief review of the theory of induced representations of Whittaker and Langlands types, $\GL_n \times \GL_m$ Rankin--Selberg integrals, Godement--Jacquet zeta integrals, and $L$-functions and epsilon factors. We survey the nonarchimedean theory of newforms, oldforms, conductor exponents, and test vectors for Rankin--Selberg integrals and Godement--Jacquet zeta integrals in \hyperref[sect:nonarchimedean]{Section \ref*{sect:nonarchimedean}}. This serves to motivate the results stated in \hyperref[sect:archimedean]{Section \ref*{sect:archimedean}}, where we present an analogous theory in the archimedean setting; these results are all essentially new, although several of the results for $n = 2$ are implicit in the work of Popa \cite{Pop08}. We discuss this theory further in \hyperref[sect:further]{Section \ref*{sect:further}}. The remaining sections are devoted to the proofs of the theorems stated in \hyperref[sect:archimedean]{Section \ref*{sect:archimedean}}.

\section{Preliminaries}
\label{sect:prelim}

\subsection{Groups and Haar Measures}
\label{sect:groupsHaar}

\subsubsection{Local Fields and Absolute Values}

Let $F$ be a local field, and denote by $|\cdot|_F$ the absolute value on $F$: for nonarchimedean $F$ with ring of integers $\OO$, maximal ideal $\pp$, and uniformiser $\varpi$, so that $\varpi \OO = \pp$ and $\OO / \pp \cong \F_q$ for some finite field of order $q$, this absolute value is normalised such that $|\varpi|_F = q^{-1}$, while for archimedean $F$, this is normalised such that
\[|x|_F = \begin{dcases*}
\max\{x,-x\} & if $F = \R$,	\\
x \overline{x} & if $F = \C$.
\end{dcases*}\]
When the local field is clear from context, we write $|\cdot|$ in place of $|\cdot|_F$. We also let $\|\cdot\| \coloneqq |\cdot|_{\C}^{1/2}$ denote the standard modulus on $\C$.

\subsubsection{Haar Measures on $F$ and $F^{\times}$}
\label{sect:HaarF}

We let $dx$ denote the Haar measure on $F$ normalised such that it is self-dual with respect to a fixed nontrivial additive character $\psi = \psi_F$ of $F$. For $F = \R$, we choose $\psi(x) \coloneqq \exp(2\pi i x)$, so that $dx$ is the Lebesgue measure; for $F = \C$, we choose $\psi(x) \coloneqq \exp(2\pi i(x + \overline{x}))$, so that $dx$ is twice the Lebesgue measure; finally, we choose $\psi$ to be unramified when $F$ is nonarchimedean, in which case $dx$ gives $\OO$ volume $1$. The multiplicative Haar measure $d^{\times} x$ for $F^{\times}$ is defined to be $\zeta_F(1) |x|^{-1} \, dx$, where
\[\zeta_F(s) \coloneqq \begin{dcases*}
\pi^{-\frac{s}{2}} \Gamma\left(\frac{s}{2}\right) & if $F = \R$,	\\
2(2\pi)^{-s} \Gamma(s) & if $F = \C$,	\\
\frac{1}{1 - q^{-s}} & if $F$ is nonarchimedean.
\end{dcases*}\]

\subsubsection{Subgroups of $\GL_n(F)$ and the Iwasawa Decomposition}

For each $r$-tuple of positive integers $(n_1,\ldots,n_r) \in \N^r$ for which $n_1 + \cdots + n_r = n$, let $\Pgp(F) = \Pgp_{(n_1,\ldots,n_r)}(F)$ denote the associated standard upper parabolic subgroup of $\GL_n(F)$ containing the standard Borel subgroup of upper triangular matrices. This has the Levi decomposition $\Pgp(F) = \Ngp_{\Pgp}(F) \Mgp_{\Pgp}(F)$, where the block-diagonal Levi subgroup $\Mgp_{\Pgp}(F)$ is isomorphic to $\GL_{n_1}(F) \times \cdots \times \GL_{n_r}(F)$, while the unipotent radical $\Ngp_{\Pgp}(F)$ of $\Pgp(F)$ consists of upper triangular matrices with block-diagonal entries $(1_{n_1},\ldots,1_{n_r})$; here we have written $1_n$ to denote the $n \times n$ identity matrix. When $\Pgp(F)$ is the standard Borel (and minimal parabolic) subgroup $\Pgp_{(1,\ldots,1)}(F)$, we write $\Ngp_{\Pgp}(F) \eqqcolon \Ngp_n(F) \cong F^{n(n - 1)/2}$, the subgroup of unipotent upper triangular matrices, and $\Mgp_{\Pgp}(F) \eqqcolon \Agp_n(F) \cong (F^{\times})^n$, the subgroup of diagonal matrices.

The maximal compact subgroup $K_n$ of $\GL_n(F)$, unique up to conjugacy, is
\[K_n = \begin{dcases*}
\Ogp(n) & if $F = \R$,	\\
\Ugp(n) & if $F = \C$,	\\
\GL_n(\OO) & if $F$ is nonarchimedean.
\end{dcases*}\]
When the context is clear, we write $K$ in place of $K_n$. Given a standard parabolic subgroup $\Pgp(F)$, we have the Iwasawa decomposition $\GL_n(F) = \Pgp(F) K_n = \Ngp_{\Pgp}(F) \Mgp_{\Pgp}(F) K_n$. Note that the Iwasawa decomposition is not unique since $\Mgp_{\Pgp}(F)$ intersects $K_n$ nontrivially.

\subsubsection{Haar Measures on $\GL_n(F)$ and Its Subgroups}

We normalise the Haar measure $dg$ on $\GL_n(F) \ni g$ via the Iwasawa decomposition $g = uak$ for the standard Borel subgroup, so that $dg = \delta_n^{-1}(a) \, du \, d^{\times} a \, dk$. Here $du = \prod_{j = 1}^{n - 1} \prod_{\ell = j + 1}^{n} du_{j,\ell}$ for $u \in \Ngp_n(F)$ with upper triangular entries $u_{j,\ell} \in F$, $d^{\times} a = \prod_{j = 1}^{n} d^{\times} a_j$ for $a \in \Agp_n(F)$ with diagonal entries $a_j \in F^{\times}$, $\delta_n(a) = \prod_{j = 1}^{n} |a_j|^{n - 2j + 1}$ denotes the modulus character of the Borel subgroup, and $dk$ is the Haar measure on the compact group $K_n \ni k$ normalised to give $K_n$ volume $1$ (so that when $F = \R$ and $n = 1$, in which case $K_1 = \Ogp(1) = \{\pm 1\} \cong \Z/2\Z$, this is just half the counting measure).

More generally, given a standard parabolic subgroup $\Pgp(F) = \Ngp_{\Pgp}(F) \Mgp_{\Pgp}(F)$ of $\GL_n(F)$, the Haar measure $dg$ on $\GL_n(F) \ni g$ is given by $dg = \delta_{\Pgp}^{-1}(m) \, du \, d^{\times} m \, dk$ with respect to the Iwasawa decomposition $g = umk$, where for $m = \blockdiag(m_1,\ldots,m_r)$, the modulus character is
\[\delta_{\Pgp}(m) = \prod_{j = 1}^{r} \left|\det m_j\right|^{n - 2(n_1 + \cdots + n_{j - 1}) - n_j}\]
and $d^{\times} m = \prod_{j = 1}^{r} dm_j$ with $dm_j$ the Haar measure on $\GL_{n_j}(F) \ni m_j$ normalised via the Iwasawa decomposition for the standard Borel subgroup of $\GL_{n_j}(F)$.

\subsection{Representations}

\subsubsection{Isobaric Sums}

Given representations $(\pi_1,V_{\pi_1}), \ldots, (\pi_r,V_{\pi_r})$ of $\GL_{n_1}(F), \ldots, \GL_{n_r}(F)$, where $F$ is a local field and $n_1 + \cdots + n_r = n$, we form the representation $\pi_1 \boxtimes \cdots \boxtimes \pi_r$ of $\Mgp_{\Pgp}(F)$, where $\boxtimes$ denotes the outer tensor product and $\Mgp_{\Pgp}(F)$ denotes the block-diagonal Levi subgroup of the standard (upper) parabolic subgroup $\Pgp(F) = \Pgp_{(n_1,\ldots,n_r)}(F)$ of $\GL_n(F)$. We then extend this representation trivially to a representation of $\Pgp(F)$. By normalised parabolic induction, we obtain an induced representation $(\pi,V_{\pi})$ of $\GL_n(F)$,
\[\pi \coloneqq \Ind_{\Pgp(F)}^{\GL_n(F)} \bigboxtimes_{j = 1}^{r} \pi_j,\]
where $V_{\pi}$ denotes the space of smooth functions $f : \GL_n(F) \to V_{\pi_1} \otimes \cdots \otimes V_{\pi_r}$ that satisfy
\[f(umg) = \delta_{\Pgp}^{1/2}(m) \pi_1(m_1) \otimes \cdots \otimes \pi_r(m_r) \cdot f(g),\]
for any $u \in \Ngp_{\Pgp}(F)$, $m = \blockdiag(m_1,\ldots,m_r) \in \Mgp_{\Pgp}(F)$, and $g \in \GL_n(F)$, and the action of $\pi$ on $V_{\pi}$ is by right translation, namely $(\pi(h) \cdot f)(g) \coloneqq f(gh)$. We call $\pi$ the isobaric sum of $\pi_1,\ldots,\pi_r$, which we denote by
\[\pi = \bigboxplus_{j = 1}^{r} \pi_j.\]

\subsubsection{Induced Representations of Whittaker and Langlands Types}

A representation $\pi$ of $\GL_n(F)$ is said to be an induced representation of Whittaker type if it is the isobaric sum of $\pi_1,\ldots,\pi_r$ and each $\pi_j$ is irreducible and essentially square-integrable. Such a representation is admissible and smooth; moreover, if $F$ is archimedean, then it is a Fr\'{e}chet representation of moderate growth and of finite length. The contragredient of an induced representation of Whittaker type $\pi = \pi_1 \boxplus \cdots \boxplus \pi_r$ is $\widetilde{\pi} = \widetilde{\pi_1} \boxplus \cdots \boxplus \widetilde{\pi_r}$, which is again an induced representation of Whittaker type. If each $\pi_j$ is additionally of the form $\sigma_j \otimes \left|\det\right|^{t_j}$, where $\sigma_j$ is irreducible, unitary, and square-integrable, and $\Re(t_1) \geq \cdots \geq \Re(t_r)$, then $\pi$ is said to be an induced representation of Langlands type.

\subsubsection{Whittaker Models}

A Whittaker functional $\Lambda : V_{\pi} \to \C$ of an admissible smooth representation $(\pi,V_{\pi})$ of $\GL_n(F)$ is a continuous linear functional that satisfies
\[\Lambda(\pi(u) \cdot v) = \psi_n(u) \Lambda(v)\]
for all $v \in V_{\pi}$ and $u \in \Ngp_n(F)$; here
\[\psi_n(u) \coloneqq \psi\left(\sum_{j = 1}^{n - 1} u_{j,j + 1}\right).\]
If $\pi$ is additionally irreducible, then the space of Whittaker functionals of $\pi$ is at most one-dimensional. If this space is indeed one-dimensional, so that there exists a unique such functional up to scalar multiplication, then $\pi$ is said to be generic. Every induced representation of Langlands type $(\pi,V_{\pi})$ admits a nontrivial Whittaker functional $\Lambda$; moreover, $\pi$ is isomorphic to its unique Whittaker model $\WW(\pi,\psi)$, which is the image of $V_{\pi}$ under the map $v \mapsto \Lambda(\pi(\cdot) \cdot v)$, so that $\WW(\pi,\psi)$ consists of Whittaker functions on $\GL_n(F)$ of the form $W(g) \coloneqq \Lambda(\pi(g) \cdot v)$. An induced representation of Whittaker type $\pi$ also has a one-dimensional space of Whittaker functionals, but the map $v \mapsto \Lambda(\pi(\cdot) \cdot v)$ need not be injective, so that the Whittaker model may only be a model of a quotient of $\pi$.

\subsubsection{Irreducible Representations}
\label{sect:irredrep}

For nonarchimedean $F$, every irreducible admissible smooth representation $\pi$ of $\GL_n(F)$ is isomorphic to the unique irreducible quotient of some induced representation of Langlands type. If $\pi$ is also generic, then it is isomorphic to some (necessarily irreducible) induced representation of Langlands type \cite{CS98}.

For archimedean $F$, we recall that a Casselman--Wallach representation of $\GL_n(F)$ is an admissible smooth Fr\'{e}chet representation of moderate growth and of finite length \cite{Wal92,BK14}; in particular, induced representations of Whittaker type are Casselman--Wallach representations. Every irreducible Casselman--Wallach representation $\pi$ of $\GL_n(F)$ is isomorphic to the unique irreducible quotient of some induced representation of Langlands type. Again, if $\pi$ is additionally generic, then it is isomorphic to some (necessarily irreducible) induced representation of Langlands type.

\subsubsection{Spherical Representations}
\label{sect:sphericalrep}

An induced representation of Whittaker type $\pi$ is said to be spherical if it has a $K$-fixed vector. Such a spherical representation $\pi$ must then be a principal series representation of the form $|\cdot|^{t_1} \boxplus \cdots \boxplus |\cdot|^{t_n}$; furthermore, the subspace of $K$-fixed vectors must be one-dimensional. This $K$-fixed vector, which is unique up to scalar multiplication, is called the spherical vector of $\pi$. For a spherical representation of Langlands type $\pi$, the spherical Whittaker function $W^{\circ}$ in the Whittaker model $\WW(\pi,\psi)$ is given by the Jacquet integral
\[W^{\circ}(g) \coloneqq \int_{\Ngp_n(F)} f^{\circ}(w_n u g) \overline{\psi_n}(u) \, du,\]
where $w_n \coloneqq \antidiag(1,\ldots,1)$ is the long Weyl element and $f^{\circ}$ is the canonically normalised spherical vector in the induced model: the unique smooth function $f^{\circ} : \GL_n(F) \to \C$ satisfying
\[f^{\circ}(1_n) = \prod_{j = 1}^{n - 1} \prod_{\ell = j + 1}^{n} \zeta_F(1 + t_j - t_{\ell}), \qquad f^{\circ}(uag) = f^{\circ}(g) \delta_n^{1/2}(a) \prod_{j = 1}^{n} |a_j|^{t_j}, \qquad f^{\circ}(gk) = f^{\circ}(g)\]
for all $u \in \Ngp_n(F)$, $a = \diag(a_1,\ldots,a_n) \in \Agp_n(F)$, $g \in \GL_n(F)$, and $k \in K$. The Jacquet integral of $f^{\circ}$ converges absolutely when $\Re(t_1) > \cdots > \Re(t_n)$ and extends holomorphically to all of $\C^n \ni (t_1,\ldots,t_n)$ (cf.\ \hyperref[sect:Jacquetint]{Section \ref*{sect:Jacquetint}}). For nonarchimedean $F$, the normalisation of $W^{\circ}$ is such that $W^{\circ}(1_n) = 1$.

\subsection{Integral Representations of \texorpdfstring{$L$}{L}-Functions}

\subsubsection{Rankin--Selberg Integrals}

We recall the definition of Rankin--Selberg integrals over a local field $F$; see \cite{Cog04} for further details. Given induced representations of Whittaker type $\pi$ of $\GL_n(F)$ and $\pi'$ of $\GL_m(F)$ with $m \leq n$, we take Whittaker functions $W \in \WW(\pi,\psi)$ and $W' \in \WW\left(\pi',\overline{\psi}\right)$ and form the local $\GL_n \times \GL_m$ Rankin--Selberg integral defined by
\begin{align}
\label{eq:RankinSelbergnm}
\Psi(s,W,W') & \coloneqq \int\limits_{\Ngp_m(F) \backslash \GL_m(F)} W \begin{pmatrix} g & 0 \\ 0 & 1_{n - m} \end{pmatrix} W'(g) \left|\det g\right|^{s - \frac{n - m}{2}} \, dg \quad \text{for $m < n$,}	\\
\label{eq:RankinSelbergnn}
\Psi(s,W,W',\Phi) & \coloneqq \int\limits_{\Ngp_n(F) \backslash \GL_n(F)} W(g) W'(g) \Phi(e_n g) \left|\det g\right|^s \, dg \quad \text{for $m = n$,}
\end{align}
where $\Phi \in \Ss(\Mat_{1 \times n}(F))$ is a Schwartz--Bruhat function and $e_n \coloneqq (0,\ldots,0,1) \in \Mat_{1 \times n}(F) = F^n$. These integrals converge absolutely for $\Re(s)$ sufficiently large and extend meromorphically to the entire complex plane via the local functional equation.

\subsubsection{Godement--Jacquet Zeta Integrals}

Following \cite{GJ72,Jac79}, we define Godement--Jacquet zeta integrals over a local field $F$. Given an induced representation of Whittaker type $(\pi,V_{\pi})$ of $\GL_n(F)$, we take $v_1 \in V_{\pi}$, $\widetilde{v_2} \in V_{\widetilde{\pi}}$, with associated matrix coefficient $\beta(g) \coloneqq \langle \pi(g) \cdot v_1, \widetilde{v_2}\rangle$ of $\pi$, and form the local Godement--Jacquet zeta integral defined by
\begin{equation}
\label{eq:GodementJacquet}
Z(s,\beta,\Phi) \coloneqq \int_{\GL_n(F)} \beta(g) \Phi(g) \left|\det g\right|^{s + \frac{n - 1}{2}} \, dg,
\end{equation}
where $\Phi \in \Ss(\Mat_{n \times n}(F))$ is a Schwartz--Bruhat function. This integral converges absolutely for $\Re(s)$ sufficiently large and extends meromorphically to the entire complex plane via the local functional equation.

\subsection{\texorpdfstring{$L$}{L}-Functions and Epsilon Factors}

\subsubsection{Rankin--Selberg $L$-Functions and Standard $L$-Functions}

For nonarchimedean $F$, the Rankin--Selberg $L$-function $L(s,\pi \times \pi')$ is the generator of the $\C[q^s,q^{-s}]$-fractional ideal of $\C(q^{-s})$ generated by the family of Rankin--Selberg integrals $\Psi(s,W,W')$ (or $\Psi(s,W,W',\Phi)$ if $m = n$) with $W \in \WW(\pi,\psi)$ and $W' \in \WW\left(\pi',\overline{\psi}\right)$ (and $\Phi \in \Ss(\Mat_{1 \times n}(F))$ if $m = n$) with $L(s,\pi \times \pi')$ normalised to be of the form $P(q^{-s})^{-1}$ for some $P(q^{-s}) \in \C[q^{-s}]$ whose constant term is $1$. For archimedean $F$, the Rankin--Selberg $L$-function $L(s,\pi \times \pi')$ is defined via the local Langlands correspondence as explicated in \cite{Kna94}.

Similarly, for nonarchimedean $F$, the standard $L$-function $L(s,\pi)$ is the generator of the $\C[q^s,q^{-s}]$-fractional ideal of $\C(q^{-s})$ generated by the family of Godement--Jacquet zeta integrals $Z(s,\beta,\Phi)$ with $v_1 \in V_{\pi}$, $\widetilde{v_2} \in V_{\widetilde{\pi}}$, and $\Phi \in \Ss(\Mat_{n \times n}(F))$, with $L(s,\pi)$ normalised to be of the form $P(q^{-s})^{-1}$ for some $P(q^{-s}) \in \C[q^{-s}]$ whose constant term is $1$. For archimedean $F$, the standard $L$-function $L(s,\pi)$ is defined via the local Langlands correspondence as explicated in \cite{Kna94}.

In both settings, upon decomposing $\pi$ and $\pi'$ as isobaric sums
\begin{equation}
\label{eq:pipi'isobaric}
\pi = \bigboxplus_{j = 1}^{r} \pi_j, \qquad \pi' = \bigboxplus_{j' = 1}^{r'} \pi_{j'}',
\end{equation}
we have the identities
\begin{equation}
\label{eq:isobaricRS}
L(s,\pi) = \prod_{j = 1}^{r} L(s,\pi_j), \qquad L(s,\pi \times \pi') = \prod_{j = 1}^{r} \prod_{j' = 1}^{r'} L(s,\pi_j \times \pi_{j'}').
\end{equation}
Moreover, Rankin--Selberg $L$-functions involving twists by a character are related to standard $L$-functions via the identity
\begin{equation}
\label{eq:twistedLfunction}
L(s,\pi \times |\cdot|^t) = L(s + t,\pi).
\end{equation}

\subsubsection{$L$-Functions for Representations of $\GL_n(\C)$}

Essentially square-integrable representations of $\GL_n(\C)$ exist only for $n = 1$, in which case the representation must be a character of the form $\pi(z) = \chi^{\kappa}(z) |z|_{\C}^t$ for some $\kappa \in \Z$ and $t \in \C$, where $z \in \GL_1(\C) = \C^{\times}$ and $\chi$ is the canonical character
\[\chi(z) \coloneqq e^{i \arg(z)} = \frac{z}{|z|_{\C}^{1/2}}.\]
The $L$-function of $\pi$ is
\begin{equation}
\label{eq:L(s,pi)C}
L(s,\pi) = \zeta_{\C}\left(s + t + \frac{\|\kappa\|}{2}\right),
\end{equation}
where
\begin{equation}
\label{eq:zetaCint}
\zeta_{\C}(s) \coloneqq 2(2\pi)^{-s} \Gamma(s) = \int_{\C^{\times}} |x|_{\C}^s \exp\left(-2\pi x \overline{x}\right) \, d^{\times} x.
\end{equation}
This integral representation of $\zeta_{\C}(s)$ converges absolutely for $\Re(s) > 0$ and extends meromorphically to the entire complex plane with poles at nonpositive integers. The contragredient of $\pi$ is $\widetilde{\pi} = \chi^{-\kappa} |\cdot|_{\C}^{-t}$, so that
\begin{equation}
\label{eq:L(s,tildepi)C}
L(s,\widetilde{\pi}) = \zeta_{\C}\left(s - t + \frac{\|\kappa\|}{2}\right).
\end{equation}

\subsubsection{$L$-Functions for Representations of $\GL_n(\R)$}

Essentially square-integrable representations of $\GL_n(\R)$ exist only for $n \in \{1,2\}$. An essentially square-integrable representation of $\GL_1(\R)$ must be a character of the form $\pi(x) = \chi^{\kappa}(x) |x|_{\R}^t$ for some $\kappa \in \{0,1\}$ and $t \in \C$, where $x \in \GL_1(\R) = \R^{\times}$ and $\chi$ is the canonical character
\[\chi(x) \coloneqq \sgn(x) = \frac{x}{|x|_{\R}}.\]
The $L$-function of $\pi$ is
\begin{equation}
\label{eq:L(s,pi)R1}
L(s,\pi) = \zeta_{\R}(s + t + \kappa),
\end{equation}
where
\begin{equation}
\label{eq:zetaRint}
\zeta_{\R}(s) \coloneqq \pi^{-\frac{s}{2}} \Gamma\left(\frac{s}{2}\right) = \int_{\R^{\times}} |x|_{\R}^s \exp\left(-\pi x^2\right) \, d^{\times} x.
\end{equation}
This integral representation of $\zeta_{\R}(s)$ converges absolutely for $\Re(s) > 0$ and extends meromorphically to the entire complex plane with poles at nonpositive even integers. The contragredient is $\widetilde{\pi} = \chi^{\kappa} |\cdot|_{\R}^{-t}$, so that
\begin{equation}
\label{eq:L(s,tildepi)R1}
L(s,\widetilde{\pi}) = \zeta_{\R}(s - t + \kappa).
\end{equation}

For $n = 2$, we note that
\[|\cdot|_{\R}^t \boxplus \chi |\cdot|_{\R}^t \cong \Ind_{\GL_1(\C)}^{\GL_2(\R)} |\cdot|_{\C}^t,\]
where $\GL_1(\C)$ is viewed as a subgroup of $\GL_2(\R)$ via the identification $a + ib \mapsto \begin{psmallmatrix} a & b \\ -b & a \end{psmallmatrix}$. For $\kappa \neq 0$, the essentially discrete series representation of weight $\|\kappa\| + 1$,
\[D_{\|\kappa\| + 1} \otimes \left|\det\right|_{\R}^t \coloneqq \Ind_{\GL_1(\C)}^{\GL_2(\R)} \chi^{\kappa} |\cdot|_{\C}^t \cong \Ind_{\GL_1(\C)}^{\GL_2(\R)} \chi^{-\kappa} |\cdot|_{\C}^t,\]
is essentially square-integrable, and every essentially square-integrable representation of $\GL_2(\R)$ is of the form $\pi = D_{\kappa} \otimes \left|\det\right|_{\R}^t$ for some integer $\kappa \geq 2$ and $t \in \C$.
The $L$-function of $\pi$ is
\begin{equation}
\label{eq:L(s,pi)R2}
L(s,\pi) = \zeta_{\C}\left(s + t + \frac{\kappa - 1}{2}\right) = \zeta_{\R}\left(s + t + \frac{\kappa - 1}{2}\right) \zeta_{\R}\left(s + t + \frac{\kappa + 1}{2}\right).
\end{equation}
The contragredient of $\pi$ is $\widetilde{\pi} = D_{\kappa} \otimes \left|\det\right|_{\R}^{-t}$, so that
\begin{equation}
\label{eq:L(s,tildepi)R2}
L(s,\widetilde{\pi}) = \zeta_{\C}\left(s - t + \frac{\kappa - 1}{2}\right) = \zeta_{\R}\left(s - t + \frac{\kappa - 1}{2}\right) \zeta_{\R}\left(s - t + \frac{\kappa + 1}{2}\right).
\end{equation}

\subsubsection{Epsilon Factors}

To any induced representations of Whittaker type $\pi$ of $\GL_n(F)$ and $\pi'$ of $\GL_m(F)$ and any nontrivial additive character $\psi$ of $F$, one can associate the epsilon factor $\e(s,\pi,\psi)$ of $\pi$, $\e(s,\pi',\psi)$ of $\pi'$, and $\e(s,\pi \times \pi',\psi)$ of $\pi \times \pi'$, which arise via the local functional equations for the Godement--Jacquet zeta integral and $\GL_n \times \GL_m$ Rankin--Selberg integral respectively. In particular, the local functional equation for the Godement--Jacquet zeta integral is
\begin{equation}
\label{eq:GJfunceq}
\frac{Z(1 - s,\widetilde{\beta},\widehat{\Phi})}{L(1 - s,\widetilde{\pi})} = \e(s,\pi,\psi) \frac{Z(s,\beta,\Phi)}{L(s,\pi)},
\end{equation}
where $\widetilde{\beta}(g) \coloneqq \beta\left(\prescript{t}{}{g}^{-1}\right)$ with $\prescript{t}{}{g}$ denoting the transpose of $g$, while the Fourier transform $\widehat{\Phi}$ is
\[\widehat{\Phi}(y) \coloneqq \int\limits_{\Mat_{n \times n}(F)} \Phi(x) \overline{\psi}(\Tr(x \prescript{t}{}{y})) \, dx.\]
Note that this normalisation of the Fourier transform differs from that in \cite{GJ72}, and that the epsilon factor is dependent on this normalisation. For the local functional equation for the $\GL_n \times \GL_m$ Rankin--Selberg integral, see, for example, \cite[Theorem 2.1]{Jac09}.

Upon decomposing $\pi$ and $\pi'$ as isobaric sums as in \eqref{eq:pipi'isobaric}, we have the identities
\begin{equation}
\label{eq:epsilonfactorise}
\e(s,\pi,\psi) = \prod_{j = 1}^{r} \e(s,\pi_j,\psi), \qquad \e(s,\pi \times \pi',\psi) = \prod_{j = 1}^{r} \prod_{j' = 1}^{r'} \e(s,\pi_j \times \pi_{j'}',\psi)
\end{equation}
by \cite[Theorem (3.1), Proposition (8.4), Proposition (9.4)]{JP-SS83} for nonarchimedean $F$ and via the local Langlands correspondence for archimedean $F$. When $\pi'$ is the trivial representation of $\GL_1(F)$, we have the equality $\e(s,\pi,\psi) = \e(s,\pi \times \pi',\psi)$ between the epsilon factor of $\pi$ defined via the Godement--Jacquet zeta integral and the epsilon factor of $\pi \times \pi'$ defined via the $\GL_n \times \GL_1$ Rankin--Selberg integral.

For nonarchimedean $F$, the epsilon factors $\e(s,\pi,\psi)$ and $\e(s,\pi \times \pi',\psi)$ are units in $\C[q^s,q^{-s}]$ of the form
\begin{equation}
\label{eq:epsilonfactorunit}
\e(s,\pi,\psi) = \e\left(\frac{1}{2},\pi,\psi\right) q^{-c(\pi) \left(s - \frac{1}{2}\right)}, \qquad \e(s,\pi \times \pi',\psi) = \e\left(\frac{1}{2},\pi \times \pi',\psi\right) q^{-c(\pi \times \pi') \left(s - \frac{1}{2}\right)}
\end{equation}
for some nonnegative integers $c(\pi)$ and $c(\pi \times \pi')$.

For archimedean $F$, we have that
\begin{equation}
\label{eq:epsilonfactor}
\e(s,\pi,\psi) = \begin{dcases*}
i^{-\kappa} & if $F = \R$ and $\pi = \chi^{\kappa} |\cdot|_{\R}^t$ for $\kappa \in \{0,1\}$,	\\
i^{-\kappa} & if $F = \R$ and $\pi = D_{\kappa} \otimes \left|\det\right|_{\R}^t$ for $\kappa \geq 2$,	\\
i^{-\|\kappa\|} & if $F = \C$ and $\pi = \chi^{\kappa} |\cdot|_{\C}^t$ for $\kappa \in \Z$,
\end{dcases*}
\end{equation}
which, via \eqref{eq:GJfunceq}, may be used to determine $\e(s,\pi,\psi)$ and $\e(s,\pi \times \pi',\psi)$ for \emph{any} induced representations of Whittaker type $\pi$ and $\pi'$, not just for essentially square-integrable representations.

\section{Nonarchimedean Newform Theory}
\label{sect:nonarchimedean}

We now detail the theory of the conductor exponent, newforms, and oldforms associated to generic irreducible admissible smooth representations of $\GL_n(F)$ with $F$ nonarchimedean (or more generally associated to induced representations of Whittaker or Langlands type; cf.\ \hyperref[sect:irredrep]{Section \ref*{sect:irredrep}}), as well as the relation between newforms and test vectors for $\GL_n \times \GL_m$ Rankin--Selberg integrals and test vectors for Godement--Jacquet zeta integrals. The results herein are for the most part well-known; we recall them as motivation for \hyperref[sect:archimedean]{Section \ref*{sect:archimedean}}, in which we discuss analogous yet new results for archimedean $F$.

\subsection{The Conductor Exponent, the Newform, and the Newform \texorpdfstring{$K$}{K}-Type}

Let $F$ be a nonarchimedean local field and let $K = \GL_n(\OO)$ be the maximal compact subgroup of $\GL_n(F)$, unique up to conjugation. For a nonnegative integer $m$, we define the following finite index subgroup of $K$:
\[K_1\left(\pp^m\right) \coloneqq \left\{k \in K : k_{n,1},\ldots,k_{n,n - 1}, k_{n,n} - 1 \in \pp^m \right\}.\]
(This is not to be confused with $K_1 \coloneqq \GL_1(\OO) = \OO^{\times}$, the maximal compact subgroup of $\GL_1(F) = F^{\times}$.) Given an induced representation of Langlands type $(\pi,V_{\pi})$ of $\GL_n(F)$, we define the vector subspace $V_{\pi}^{K_1(\pp^m)}$ of $V_{\pi}$ consisting of $K_1(\pp^m)$-fixed vectors:
\[V_{\pi}^{K_1(\pp^m)} \coloneqq \left\{v \in V_{\pi} : \pi(k) \cdot v = v \text{ for all } k \in K_1\left(\pp^m\right)\right\}.\]

The following theorem is due to Casselman \cite[Theorem 1]{Cas73} for $n = 2$ and Jacquet, Piatetski-Shapiro, and Shalika for arbitrary $n$ (though cf.\ \hyperref[rem:JacquetMatringe]{Remark \ref*{rem:JacquetMatringe}}).

\begin{theorem}[{Jacquet--Piatetski-Shapiro--Shalika \cite[Th\'{e}or\`{e}me (5)]{JP-SS81}}]
\label{thm:JPPSconductor}
Let $(\pi,V_{\pi})$ be an induced representation of Langlands type of $\GL_n(F)$. There exists a minimal nonnegative integer $m$ for which $V_{\pi}^{K_1(\pp^m)}$ is nontrivial. For this minimal value of $m$, $V_{\pi}^{K_1(\pp^m)}$ is one-dimensional.
\end{theorem}

\begin{definition}
\label{def:condnewform}
We define the conductor exponent of $\pi$ to be this minimal nonnegative integer $m$ and denote it by $c(\pi)$; we then call the ideal $\pp^{c(\pi)}$ the conductor of $\pi$. The newform of $\pi$ is defined to be the nonzero vector $v^{\circ} \in V_{\pi}^{K_1(\pp^{c(\pi)})}$, unique up to scalar multiplication.
\end{definition}

The uniqueness of the newform may be thought of as being a multiplicity-one theorem for newforms. The reason for naming this distinguished vector a newform is due to its relation to the classical theory of modular forms: as shown by Casselman \cite[Section 3]{Cas73}, an automorphic form on $\GL_2(\A_{\Q})$ whose associated Whittaker function is a pure tensor composed of newforms in the Whittaker model is the ad\`{e}lic lift of a classical newform in the sense of Atkin and Lehner \cite{AL70}.

\begin{remark}
There is no consensus on the name of this distinguished vector: Casselman \cite{Cas73} leaves it unnamed, Jacquet, Piatetski-Shapiro, and Shalika \cite{JP-SS81} name it the essential vector, whereas Reeder \cite{Ree91} calls it the new vector, and Schmidt \cite{Sch02}, regarding $\pi$ as being the local component of an automorphic representation, refers to it as the local newform. When viewed in the Whittaker model, this vector is referred to by Popa \cite{Pop08} as the Whittaker newform, whereas Matringe \cite{Mat13} calls it the essential Whittaker function. Similarly, some authors instead call $c(\pi)$ the conductor of $\pi$, while others yet refer to $q^{c(\pi)}$ as the conductor of $\pi$. Perhaps a more apt name for $q^{c(\pi)}$ is the absolute conductor of $\pi$, being the absolute norm of the ideal $\pp^{c(\pi)}$.
\end{remark}

\begin{remark}
Under the local Langlands correspondence, which gives a bijection between irreducible admissible smooth representations $\pi$ of $\GL_n(F)$ and $n$-dimensional Frobenius semisimple Weil--Deligne representations $\rho$ of $F$, the conductor exponent $c(\pi)$ is equal to the Artin exponent $a(\rho)$.
\end{remark}

If $c(\pi) = 0$, so that $K_1(\pp^{c(\pi)}) = K$, then $\pi$ must be a spherical representation and we say that $\pi$ is unramified. If $c(\pi) > 0$, then $\pi$ is said to be ramified. In this sense, the conductor exponent is a measure of the extent of ramification of $\pi$: it quantifies how ramified $\pi$ may be.

Since $\pi$ is admissible, $\Hom_K(\tau,\pi|_K)$ is finite-dimensional for each irreducible smooth representation $\tau$ of $K$. We say that such a representation $\tau$ is a \emph{$K$-type} of $\pi$ if $\Hom_K(\tau,\pi|_K)$ is nontrivial, and we call $\dim \Hom_K(\tau,\pi|_K)$ the \emph{multiplicity} of $\tau$ in $\pi$. The complexity of an irreducible smooth representation $\tau$ of $K$ can be measured by its \emph{level} $m$, which is the least nonnegative integer $m$ for which $\tau$ factors through the finite group $\GL_n(\OO/\pp^m)$. In \cite{Hum22}, the author proved the existence of a distinguished $K$-type of $\pi$ that occurs with multiplicity one and is closely associated to the newform and the conductor exponent.

\begin{theorem}[{\cite[Theorem 4.11]{Hum22}}]
\label{thm:nonarchimedeanKtype}
Let $(\pi,V_{\pi})$ be an induced representation of Langlands type of $\GL_n(F)$. Among the $K$-types of $\pi$ whose restriction to
\[K_{n - 1,1} \coloneqq \left\{\begin{pmatrix} a & b \\ 0 & 1 \end{pmatrix} \in K_n : a \in K_{n - 1}, \ b \in \Mat_{(n - 1) \times 1}(\OO)\right\}\]
contains the trivial representation, there exists a unique $K$-type $\tau^{\circ}$ of minimal level. Furthermore, $\tau^{\circ}$ occurs with multiplicity one in $\pi$, the level of $\tau^{\circ}$ is equal to the conductor exponent $c(\pi)$, and the subspace of $V_{\pi}$ of $\tau^{\circ}$-isotypic $K_{n - 1,1}$-invariant vectors is equal to the one-dimensional subspace $V_{\pi}^{K_1(\pp^{c(\pi)})}$ spanned by the newform $v^{\circ}$.
\end{theorem}

\begin{definition}
\label{def:nonarchimedeanKtype}
We call the distinguished $K$-type $\tau^{\circ}$ the \emph{newform $K$-type}.
\end{definition}

\subsection{Oldforms}

While Jacquet, Piatetski-Shapiro, and Shalika merely show that $V_{\pi}^{K_1(\pp^{c(\pi)})}$ is one-dimensional, one can also calculate the dimension of $V_{\pi}^{K_1(\pp^m)}$ for all $m \geq c(\pi)$ in terms of a binomial coefficient; for $n = 2$, this is due to Casselman \cite[Corollary to the Proof]{Cas73}, while Reeder has proven this result for arbitrary $n$.

\begin{theorem}[{Reeder \cite[Theorem 1]{Ree91}}]
\label{thm:Reeder}
Let $(\pi,V_{\pi})$ be an induced representation of Langlands type of $\GL_n(F)$ with $n \geq 2$. We have that
\[\dim V_{\pi}^{K_1(\pp^m)} = \begin{dcases*}
\binom{m - c(\pi) + n - 1}{n - 1} & if $m \geq c(\pi)$,	\\
0 & otherwise.
\end{dcases*}\]
\end{theorem}

Casselman and Reeder also give a basis for each of these spaces in terms of the action of certain Hecke operators on the newform. For $m > c(\pi)$, we call $V_{\pi}^{K_1(\pp^m)}$ the space of oldforms of exponent $m$. Once again, the reason for naming these distinguished vectors oldforms is due to their relation to the classical theory of modular forms: an automorphic form on $\GL_2(\A_{\Q})$ whose associated Whittaker function is a pure tensor composed of Whittaker newforms at all but finitely many places and of Whittaker oldforms at the remaining places corresponds to an oldform in the sense of Atkin and Lehner \cite{AL70}.

In \cite{Hum22}, the author showed that spaces of oldforms can be described in terms of distinguished $K$-types.

\begin{theorem}[{\cite[Theorem 4.11]{Hum22}}]
\label{thm:nonarchimedeanoldformKtype}
Let $(\pi,V_{\pi})$ be an induced representation of Langlands type of $\GL_n(F)$. For each $m \geq c(\pi)$, there exists a unique $K$-type $\tau_m$ of $\pi$ of level $m$ whose restriction to $K_{n - 1,1}$ contains the trivial representation. Furthermore, this $K$-type occurs with multiplicity
\[\binom{m - c(\pi) + n - 2}{n - 2},\]
and the direct sum indexed by nonnegative integers $\ell \in \{c(\pi),\ldots,m\}$ of the subspaces of $V_{\pi}$ of $\tau_{\ell}$-isotypic $K_{n - 1,1}$-invariant vectors is equal to $V_{\pi}^{K_1(\pp^m)}$, the space of oldforms of exponent $m$.
\end{theorem}

\subsection{Additivity and Inductivity of the Conductor Exponent}

Associated to any induced representation of Langlands type $\pi$ of $\GL_n(F)$ is an integer $c(\pi)$ defined via the epsilon factor as in \eqref{eq:epsilonfactorunit}.

\begin{theorem}[{Jacquet--Piateski-Shapiro--Shalika \cite[Section 5]{JP-SS83}}]
\label{thm:nonarchimedeanepsilonconductor}
Let $\pi$ be an induced representation of Langlands type of $\GL_n(F)$. The integer $c(\pi)$ appearing in the epsilon factor $\e(s,\pi,\psi)$ as in \eqref{eq:epsilonfactorunit} is equal to the conductor exponent of $\pi$.
\end{theorem}

From the multiplicativity of epsilon factors \eqref{eq:epsilonfactorise}, we have the following.

\begin{theorem}[{Jacquet--Piateski-Shapiro--Shalika \cite[Theorem (3.1)]{JP-SS83}}]
\label{thm:nonarchimedeanadditiveconductor}
For an induced representation of Langlands type $\pi = \pi_1 \boxplus \cdots \boxplus \pi_r$ of $\GL_n(F)$, we have that
\[c(\pi) = \sum_{j = 1}^{r} c(\pi_j).\]
\end{theorem}

Thus the conductor exponent $c(\pi)$ is additive with respect to isobaric sums; equivalently, the conductor $\pp^{c(\pi)}$ is multiplicative.

\begin{remark}
In the classical setting of automorphic forms on the upper half-plane, \hyperref[thm:nonarchimedeanadditiveconductor]{Theorem \ref*{thm:nonarchimedeanadditiveconductor}} manifests itself via the conductor of an Eisenstein newform: an Eisenstein newform, in the sense of \cite{You19}, is associated to a pair of primitive Dirichlet characters, and the conductor of such a newform is the product of the conductors of the two Dirichlet characters.
\end{remark}

Finally, the epsilon factor is inductive in degree zero: if $\pi$ and $\pi'$ are induced representations of Langlands type of $\GL_n(E)$, where $E$ is a finite cyclic extension of $F$ of degree $m \geq 2$, and $\AI_{E/F} \pi$ and $\AI_{E/F} \pi'$ denote the induced representations of Langlands type of $\GL_{mn}(F)$ obtained by induction \cite{HH95}, then
\[\frac{\e\left(s,\AI_{E/F} \pi, \psi\right)}{\e\left(s,\AI_{E/F} \pi',\psi\right)} = \frac{\e\left(s,\pi,\psi \circ \Tr_{E/F}\right)}{\e\left(s,\pi',\psi \circ \Tr_{E/F}\right)}.\]
Taking $\pi'$ to be the isobaric sum of $n$ copies of the trivial representation, we deduce the following.

\begin{theorem}
\label{thm:nonarchimedeaninductiveconductor}
For an induced representation of Langlands type $\pi$ of $\GL_n(E)$, where $E$ is a finite cyclic extension of $F$ of degree $m \geq 2$, we have that
\[c\left(\AI_{E/F} \pi\right) = f_{E/F} c(\pi) + d_{E/F} n,\]
where $f_{E/F}$ denotes the residual degree of $E/F$ and $d_{E/F}$ denotes the valuation of the discriminant of $E/F$.
\end{theorem}

\begin{remark}
This result also has a classical manifestation. Let $\psi$ be a Hecke Gr\"{o}\ss{}encharakter of a quadratic extension $E$ of $\Q$ with conductor $\qq \subset \OO_E$. By automorphic induction, one can associate to $\psi$ a classical newform $f_{\psi} : \Hb \to \C$ of conductor $\Norm(\qq) D_{E/\Q}$, where $\Norm(\qq) \coloneqq \# \OO_E / \qq$ is the absolute norm of the integral ideal $\qq$ and $D_{E/\Q}$ is the absolute discriminant of $E/\Q$.
\end{remark}

\subsection{Test Vectors for Rankin--Selberg Integrals}

Next, we discuss the relation between newforms and test vectors for Rankin--Selberg integrals. We first recall the test vector problem for $\GL_n \times \GL_m$ Rankin--Selberg integrals.

\begin{problem}
Given induced representations of Langlands type $\pi$ of $\GL_n(F)$ and $\pi'$ of $\GL_m(F)$ with $n \geq m$, determine the existence of right $K_n$- and $K_m$-finite Whittaker functions $W \in \WW(\pi,\psi)$ and $W' \in \WW(\pi',\overline{\psi})$, and additionally a right $K_n$-finite Schwartz--Bruhat function $\Phi \in \Ss(\Mat_{1 \times n}(F)) = \Ss(F^n)$ should $m$ be equal to $n$, such that
\[L(s,\pi \times \pi') = \begin{dcases*}
\Psi(s,W,W') & if $m < n$,	\\
\Psi(s,W,W',\Phi) & if $m = n$.
\end{dcases*}\]
\end{problem}

In full generality, this problem remains unresolved. For $m = n - 1$, $\pi'$ a spherical representation of Langlands type, and $W' = W^{\prime\circ} \in \WW(\pi',\overline{\psi})$ the spherical vector normalised as in \hyperref[sect:sphericalrep]{Section \ref*{sect:sphericalrep}}, this has been solved by Jacquet, Piatetski-Shapiro, and Shalika.

\begin{theorem}[{Jacquet--Piatetski-Shapiro--Shalika \cite[Th\'{e}or\`{e}me (4)]{JP-SS81}, Jacquet \cite{Jac12}, Matringe \cite[Corollary 3.3]{Mat13}}]
\label{thm:JPPSWhittaker}
For $n \geq 2$, let $\pi$ be an induced representation of Langlands type of $\GL_n(F)$. There exists a Whittaker function $W \in \WW(\pi,\psi)$ such that for any spherical representation of Langlands type $\pi'$ of $\GL_{n - 1}(F)$ with spherical Whittaker function $W^{\prime\circ} \in \WW(\pi',\overline{\psi})$,
\[\Psi(s,W,W^{\prime\circ}) = L(s,\pi \times \pi')\]
for $\Re(s)$ sufficiently large.

Moreover, there exists a unique such function $W^{\circ} \in \WW(\pi,\psi)$ that additionally satisfies
\[W^{\circ}\left(g\begin{pmatrix} k & 0 \\ 0 & 1 \end{pmatrix}\right) = W^{\circ}(g)\]
for all $k \in K_{n - 1}$. Up to multiplication by a scalar, this function is the newform $v^{\circ} \in V_{\pi}^{K_1(\pp^{c(\pi)})}$ viewed in the Whittaker model.
\end{theorem}

That is, $W^{\circ} \in \WW(\pi,\psi)$ is a right $K_{n - 1}$-invariant test vector for the $\GL_n \times \GL_{n - 1}$ Rankin--Selberg integral for all spherical representations $\pi'$ of $\GL_{n - 1}(F)$. We call $W^{\circ}$ the Whittaker newform of $\pi$.

\begin{remark}
\label{rem:JacquetMatringe}
The proof of \cite[Th\'{e}or\`{e}me (4)]{JP-SS81} is in fact incomplete, as was observed by Matringe; correct proofs have subsequently independently been given by Jacquet \cite[Theorem 1]{Jac12} and Matringe \cite[Corollary 3.3]{Mat13}.
\end{remark}

\begin{remark}
The normalisation of the Whittaker newform is such that $W^{\circ}(1_n) = 1$.
\end{remark}

The Whittaker newform is a test vector for more than just the $\GL_n \times \GL_{n - 1}$ Rankin--Selberg integral for all spherical representations $\pi'$ of $\GL_{n - 1}(F)$.

\begin{theorem}[{Kim \cite[Theorem 2.2.1]{Kim10}, Matringe \cite[Corollary 3.3]{Mat13}}]
\label{thm:KimMatringe}
Let $\pi$ be an induced representation of Langlands type of $\GL_n(F)$ with $n \geq 2$. For $m \in \{1, \ldots, n - 2\}$ and for every spherical representation of Langlands type $\pi'$ of $\GL_m(F)$ with spherical Whittaker function $W^{\prime\circ} \in \WW(\pi',\overline{\psi})$, the Whittaker newform $W^{\circ} \in \WW(\pi,\psi)$ of $\pi$ satisfies
\[\Psi(s,W^{\circ},W^{\prime\circ}) = L(s,\pi \times \pi')\]
for $\Re(s)$ sufficiently large.
\end{theorem}

Miyauchi also gave a proof when $m = 1$ \cite[Theorem 5.1]{Miy14} by a different method that generalises easily to prove the result for $m \in \{2, \ldots, n - 1\}$.

A similar theory also holds for the case $m = n$.

\begin{theorem}[{Kim \cite[Theorem 2.1.1]{Kim10}}]
\label{thm:Kim}
Let $\pi$ be an induced representation of Langlands type of $\GL_n(F)$. Then for every spherical representation of Langlands type $\pi'$ of $\GL_n(F)$ with spherical Whittaker function $W^{\prime\circ} \in \WW(\pi',\overline{\psi})$, the Whittaker newform $W^{\circ} \in \WW(\pi,\psi)$ of $\pi$ satisfies
\[\Psi(s,W^{\circ},W^{\prime\circ},\Phi^{\circ}) = L(s,\pi \times \pi')\]
for $\Re(s)$ sufficiently large, where $\Phi^{\circ} \in \Ss(\Mat_{1 \times n}(F))$ is given by
\[\Phi^{\circ}(x_1,\ldots,x_n) \coloneqq \begin{dcases*}
\frac{\omega_{\pi}^{-1}(x_n)}{\vol(K_0(\pp^{c(\pi)}))} & if $x_1,\ldots,x_{n - 1} \in \pp^{c(\pi)}$ and $x_n \in \OO^{\times}$,	\\
0 & otherwise,
\end{dcases*}\]
if $c(\pi) > 0$, where $\omega_{\pi}$ denotes the central character of $\pi$ and
\[K_0\left(\pp^m\right) \coloneqq \left\{k \in K : k_{n,1},\ldots,k_{n,n-1} \in \pp^m\right\},\]
while for $c(\pi) = 0$,
\[\Phi^{\circ}(x_1,\ldots,x_n) \coloneqq \begin{dcases*}
1 & if $x_1,\ldots,x_n \in \OO$,	\\
0 & otherwise.
\end{dcases*}\]
\end{theorem}

Again, this can also be proven via the method of Miyauchi \cite{Miy14}.

\begin{remark}
Little is known about test vectors for Rankin--Selberg integrals when $\pi'$ is ramified. Kim \cite[Proposition 2.2.2]{Kim10} has shown that Whittaker newforms are not test vectors when $\pi'$ is ramified: if $\pi'$ is a ramified representation of $\GL_m(F)$ with $m < n$ and $W^{\circ},W^{\prime\circ}$ are newforms of $\pi$ and $\pi'$ respectively, then $\Psi\left(s,W^{\circ},W^{\prime\circ}\right) = 0$ for all $s \in \C$. For $n = m = 2$, Kim has determined test vectors for certain pairs of representations $\pi,\pi'$ \cite{Kim10}. Recently, Kurinczuk and Matringe have explicitly determined test vectors for $n = m$ arbitrary, $\pi$ supercuspidal, and $\pi' = \widetilde{\pi} \otimes \chi$ for some unramified character $\chi$ of $F^{\times}$ \cite{KM19}. Booker, Krishnamurthy, and Lee have resolved a \emph{weakened} version of the test vector problem in \cite{BKL20} when $m < n$: they construct vectors $W \in \WW(\pi,\psi)$ and $W' \in \WW(\pi',\overline{\psi})$ for which $\Psi(s,W,W')$ is a multiple of $L(s,\pi \times \pi')$ by a nonzero polynomial in $q^{-s}$; an analogous result when $m = n$ is due to Jo \cite[Theorem 1.1]{Jo23}. Finally, Jacquet, Piatetski-Shapiro, and Shalika have shown that for each pair of representations $\pi,\pi'$ of $\GL_n(F),\GL_m(F)$ with $n > m$, there exist finite collections $\{W_j\} \subset \WW(\pi,\psi)$ and $\{W_j'\} \subset \WW(\pi',\overline{\psi})$ of right $K_n$- and $K_m$-finite Whittaker functions for which $\sum_j \Psi(s,W_j,W_j') = L(s,\pi \times \pi')$ \cite[Theorem (2.7)]{JP-SS83}; a similar result (additionally involving Schwartz--Bruhat functions $\{\Phi_j\} \subset \Ss(\Mat_{1 \times n}(F))$) also holds for $m = n$.
\end{remark}

\subsection{Test Vectors for Godement--Jacquet Zeta Integrals}

Finally, we mention the relation between newforms and test vectors for Godement--Jacquet zeta integrals. The test vector problem for the Godement--Jacquet zeta integral is the following.

\begin{problem}
Given an induced representation of Langlands type $(\pi,V_{\pi})$ of $\GL_n(F)$, determine the existence of $K$-finite vectors $v_1 \in V_{\pi}$, $\widetilde{v_2} \in V_{\widetilde{\pi}}$, and a bi-$K$-finite Schwartz--Bruhat function $\Phi \in \Ss(\Mat_{n \times n}(F))$ such that
\[Z(s,\beta,\Phi) = L(s,\pi).\]
\end{problem}

This has been solved for spherical representations by Godement and Jacquet and for nonspherical representations by the author.

\begin{theorem}[{Godement--Jacquet \cite[Lemma 6.10]{GJ72}, Humphries \cite[Theorem 1.2]{Hum21}}]
\label{thm:GJ}
Let $(\pi,V_{\pi})$ be an induced representation of Langlands type of $\GL_n(F)$. Let $\beta(g)$ denote the matrix coefficient $\langle \pi(g) \cdot v^{\circ}, \widetilde{v}^{\circ} \rangle$, where $v^{\circ} \in V_{\pi}$ denotes the newform and $\widetilde{v}^{\circ} \in V_{\widetilde{\pi}}$ is the corresponding newform normalised such that $\beta(1_n) = 1$. Then
\[Z(s,\beta,\Phi) = L(s,\pi)\]
for $\Re(s)$ sufficiently large, where $\Phi \in \Ss(\Mat_{n \times n}(F))$ is given by
\[\Phi(x) \coloneqq \begin{dcases*}
\frac{\omega_{\pi}^{-1}\left(x_{n,n}\right)}{\vol(K_0(\pp^{c(\pi)}))} & if $x \in \Mat_{n \times n}(\OO)$ with $x_{n,1},\ldots,x_{n,n - 1} \in \pp^{c(\pi)}$ and $x_{n,n} \in \OO^{\times}$,	\\
0 & otherwise
\end{dcases*}\]
if $c(\pi) > 0$, while for $c(\pi) = 0$,
\[\Phi(x) \coloneqq \begin{dcases*}
1 & if $x \in \Mat_{n \times n}(\OO)$,	\\
0 & otherwise.
\end{dcases*}\]
\end{theorem}

\section{Archimedean Newform Theory}
\label{sect:archimedean}

\subsection{Analogues of Nonarchimedean Results}

For an induced representation of Whittaker type $(\pi,V_{\pi})$ of $\GL_n(F)$ with $F$ archimedean, it is not so clear what the definition of the conductor exponent and the newform of $\pi$ ought to be. When $F$ is nonarchimedean, the conductor exponent may be defined either as the least nonnegative integer $m$ for which $V_{\pi}^{K_1(\pp^m)}$ is nonempty or as the exponent appearing in the epsilon factor $\e(s,\pi,\psi)$. Neither of these properties, however, can be easily imported to the archimedean setting. Indeed, there is no obvious analogue of the congruence subgroup $K_1(\pp^m)$ (though cf.\ \cite{JN19}), while the epsilon factor $\e(s,\pi,\psi)$ is an integral power of $i$ for all $s \in \C$; this integer is therefore only determined modulo $4$, and so cannot be directly used to define the conductor exponent.

Our starting observation that leads us to define the conductor exponent and newform in the archimedean setting is that the key property shared by the newform defined in terms of the congruence subgroup $K_1(\pp^{c(\pi)})$ in \hyperref[thm:JPPSconductor]{Theorem \ref*{thm:JPPSconductor}} and the Whittaker newform defined in terms of a test vector for the $\GL_n \times \GL_{n - 1}$ Rankin--Selberg integral in \hyperref[thm:JPPSWhittaker]{Theorem \ref*{thm:JPPSWhittaker}} is its invariance under the action of the subgroup $K_{n - 1} \ni k'$ embedded in $\GL_n(F)$ via the map $k' \mapsto \begin{psmallmatrix} k' & 0 \\ 0 & 1 \end{psmallmatrix}$. Moreover, \hyperref[thm:JPPSconductor]{Theorem \ref*{thm:JPPSconductor}} essentially states that the newform is the ``simplest'' such vector for which this is so, in the sense that $V_{\pi}^{K_1(\pp^m)}$ is trivial for $m < c(\pi)$. The following lemma exemplifies the necessity of $K_{n - 1}$-invariance for the Whittaker newform.

\begin{lemma}
Let $(\pi,V_{\pi})$ be an induced representation of Langlands type of $\GL_n(F)$ with $n \geq 2$, and let $W \in \WW(\pi,\psi)$ be a right $K_n$-finite Whittaker function, so that the action of $\pi \begin{psmallmatrix} k' & 0 \\ 0 & 1 \end{psmallmatrix}$ on $W$ for $k' \in K_{n - 1}$ generates a finite-dimensional representation $\tau'$ of $K_{n - 1}$. If $\Hom_{K_{n - 1}}(1,\tau')$ is trivial, then the $\GL_n \times \GL_{n - 1}$ Rankin--Selberg integral $\Psi\left(s,W,W^{\prime\circ}\right)$ is identically equal to zero for every spherical representation of Langlands type $\pi'$ of $\GL_{n - 1}(F)$ with spherical Whittaker function $W^{\prime\circ} \in \WW(\pi',\overline{\psi})$.
\end{lemma}

\begin{proof}
We may write $W(g) = \Lambda(\pi(g) \cdot v)$ for some $v \in V_{\pi}$. Via the Iwasawa decomposition with respect to the standard Borel subgroup and the fact that $W^{\prime\circ}$ is the spherical Whittaker function, the Rankin--Selberg integral $\Psi\left(s,W,W^{\prime\circ}\right)$, defined in \eqref{eq:RankinSelbergnm}, is equal to
\[\int_{\Agp_{n - 1}(F)} W^{\prime\circ}(a') \left|\det a'\right|^{s - \frac{1}{2}} \delta_{n - 1}^{-1}(a') \Lambda\left(\pi\begin{pmatrix} a' & 0 \\ 0 & 1 \end{pmatrix} \cdot \int_{K_{n - 1}} \pi\begin{pmatrix} k' & 0 \\ 0 & 1 \end{pmatrix} \cdot v \, dk'\right) \, d^{\times} a',\]
and the integral over $K_{n - 1}$ vanishes whenever $\Hom_{K_{n - 1}}(1,\tau')$ is trivial.
\end{proof}

This leads us to search for vectors $v \in V_{\pi}$ that are invariant under the action of $K_{n - 1}$ and to define the newform of $\pi$ to be the ``simplest'' such vector. A natural way to interpret ``simplest'' is to search for vectors in $K$-types of $\pi$ that are ``small'' in a sense that we make precise. Once we have located the newform of $\pi$, we show that, when viewed in the Whittaker model, it is a Whittaker newform in the sense of \hyperref[thm:JPPSWhittaker]{Theorem \ref*{thm:JPPSWhittaker}}. In this way, we work in the reverse direction of the nonarchimedean setting, where Jacquet, Piatetski-Shapiro, and Shalika \cite{JP-SS81} first prove \hyperref[thm:JPPSWhittaker]{Theorem \ref*{thm:JPPSWhittaker}} and then use this to deduce \hyperref[thm:JPPSconductor]{Theorem \ref*{thm:JPPSconductor}}.

\subsection{\texorpdfstring{$K$}{K}-Types}

The representation $\pi$ being admissible means that $\Hom_K \left(\tau, \pi|_K\right)$ is finite-dimensional for any $\tau$ in $\widehat{K}$, the set of equivalence classes of irreducible representations of $K$. We call $\tau$ a $K$-type of $\pi$ should the vector space $\Hom_K \left(\tau, \pi|_K\right)$ be nontrivial. To each $\tau \in \widehat{K}$, one can associate a nonnegative integer $\deg \tau$ called the Howe degree of $\tau$ \cite{How89}; this is the archimedean analogue of the level of an irreducible smooth representation of $\GL_n(\OO)$. In order to define this, we first recall the theory of highest weights for the two groups $\Ugp(n)$ and $\Ogp(n)$, then explain how the Howe degree of an irreducible representation is defined in terms of the highest weight.

\subsubsection{Highest Weight Theory for \texorpdfstring{$\Ugp(n)$}{U(n)}}

The equivalence classes of finite-dimensional irreducible representations of the unitary group
\[\Ugp(n) \coloneqq \left\{k \in \Mat_{n \times n}(\C) : k \prescript{t}{}{\overline{k}} = 1_n\right\}\]
are parametrised by the set of highest weights, which we may identify with $n$-tuples of integers that are nonincreasing:
\[\Lambda_n \coloneqq \left\{\mu = \left(\mu_1,\ldots,\mu_n\right) \in \Z^n : \mu_1 \geq \cdots \geq \mu_n\right\}.\]
We denote by $\tau_{\mu}$ an irreducible representation of $\Ugp(n)$ with highest weight $\mu \in \Lambda_n$. The Howe degree of $\tau_{\mu}$ is given by
\begin{equation}
\label{eq:HowedegreeC}
\deg \tau_{\mu} = \sum_{j = 1}^{n} \|\mu_j\|.
\end{equation}

\begin{remark}
There is another natural invariant that one associate to an irreducible representation $\tau_{\mu}$, namely its Vogan norm $\|\tau_{\mu}\|_{\V}$ \cite{Vog81}, which is
\[\|\tau_{\mu}\|_{\V}^2 = \sum_{j = 1}^{n} \left(\mu_j + n + 1 - 2j\right)^2.\]
\end{remark}

\subsubsection{Highest Weight Theory for \texorpdfstring{$\Ogp(n)$}{O(n)}}

The equivalence classes of finite-dimensional irreducible representations $\tau_{\mu}$ of the orthogonal group
\[\Ogp(n) \coloneqq \left\{k \in \Mat_{n \times n}(\R) : k \prescript{t}{}{k} = 1_n\right\}\]
are parametrised by an $n$-tuple of integers $\mu$, which we again call highest weights (though unlike $\Ugp(n)$, the compact Lie group $\Ogp(n)$ is not connected). These highest weights $\mu$ are precisely those for which the highest weight vector of the irreducible representation of $\Ugp(n)$ with highest weight $\mu$ generates $\tau_{\mu}$ when restricted to $\Ogp(n)$, and are of the form
\[\mu = (\mu_1,\ldots,\mu_m,\underbrace{\eta,\ldots,\eta}_{n - 2m \text{ times}},\underbrace{0,\ldots,0}_{m \text{ times}}) \in \N_0^n,\]
where $m \in \{0,\ldots,\lfloor \frac{n}{2}\rfloor\}$, $\mu_1 \geq \cdots \geq \mu_m \geq 1$, $\eta \in \{0,1\}$, and $\N_0 \coloneqq \N \cup \{0\}$. We again let $\Lambda_n$ denote the set of highest weights; note in particular that $\Lambda_1 = \{0,1\}$. We denote by $\tau_{\mu}$ an irreducible representation of $\Ogp(n)$ with highest weight $\mu \in \Lambda_n$. The Howe degree of $\tau_{\mu}$ is
\begin{equation}
\label{eq:HowedegreeR}
\deg \tau_{\mu} = \sum_{j = 1}^{n} \mu_j.
\end{equation}

\begin{remark}
The Vogan norm $\|\tau_{\mu}\|_{\V}$ of $\tau_{\mu}$ is
\[\|\tau_{\mu}\|_{\V}^2 = \sum_{j = 1}^{m} \left(\mu_j + n - 2j\right)^2 + \sum_{j = m + 1}^{\left\lfloor \frac{n}{2} \right\rfloor} (n - 2j)^2.\]
\end{remark}

\subsection{The Conductor Exponent, the Newform, and the Newform \texorpdfstring{$K$}{K}-Type}
\label{sect:condexp}

Let $(\pi,V_{\pi})$ be an induced representation of Whittaker type of $\GL_n(F)$. We define the projection map $\Pi^{K_{n - 1}} : V_{\pi} \to V_{\pi}$ given by
\[\Pi^{K_{n - 1}}(v) \coloneqq \int_{K_{n - 1}} \pi\begin{pmatrix} k' & 0 \\ 0 & 1 \end{pmatrix} \cdot v \, dk',\]
whose image is the subspace of $K_{n - 1}$-invariant vectors; for the sake of consistency and completeness, we define $\Pi^{K_0}$ to be the identity when $n = 1$. We also define the projection map
\[\Pi^{\tau}(v) \coloneqq \int_{K_n} \xi^{\tau}(k) \pi(k) \cdot v \, dk\]
for each irreducible representation $\tau \in \widehat{K_n}$, where
\[\xi^{\tau}(k) \coloneqq (\dim \tau) \Tr \tau(k^{-1})\]
is the elementary idempotent associated to $\tau$. The image of $V_{\pi}$ under $\Pi^{\tau}$ is the $\tau$-isotypic subspace $V_{\pi}^{\tau}$ of $V_{\pi}$, which is finite-dimensional since $\pi$ is admissible and is trivial unless $\tau$ is a $K$-type of $\pi$. The composition of these two projections is the projection
\[\left(\Pi^{\tau,K_{n - 1}}\right)(v) \coloneqq \left(\Pi^{K_{n - 1}} \circ \Pi^{\tau}\right)(v) = \left(\Pi^{\tau} \circ \Pi^{K_{n - 1}}\right)(v) = \int_{K_n} \xi^{\tau,K_{n - 1}}(k) \pi(k) \cdot v \, dk\]
onto the subspace of $K_{n - 1}$-invariant $\tau$-isotypic vectors
\[V_{\pi}^{\tau,K_{n - 1}} \coloneqq \left(\Pi^{\tau,K_{n - 1}}\right)(V_{\pi}) = \left\{v \in V_{\pi}^{\tau} : \pi\begin{pmatrix} k' & 0 \\ 0 & 1 \end{pmatrix} \cdot v = v \text{ for all $k' \in K_{n - 1}$}\right\}.\]
Here
\begin{equation}
\label{eq:xitauKn-1}
\xi^{\tau,K_{n - 1}}(k) \coloneqq \int_{K_{n - 1}} \xi^{\tau}\left(\begin{pmatrix} k' & 0 \\ 0 & 1 \end{pmatrix} k\right) \, dk' = \int_{K_{n - 1}} \xi^{\tau}\left(k \begin{pmatrix} k' & 0 \\ 0 & 1 \end{pmatrix}\right) \, dk'.
\end{equation}
Note that $\xi^{\tau,K_{n - 1}}$ is identically equal to zero if and only if $\Hom_{K_{n - 1}}(1,\tau|_{K_{n - 1}})$ is trivial. Finally, for any nonnegative integer $m$, we set
\[V_{\pi}(m)^{K_{n - 1}} \coloneqq \bigoplus_{\substack{\tau \in \widehat{K_n} \\ \deg \tau = m}} V_{\pi}^{\tau,K_{n - 1}},\]
the subspace of $K_{n - 1}$-invariant vectors that are $\tau$-isotypic for some $\tau \in \widehat{K_n}$ of degree $m$. We prove the following.

\begin{theorem}[{Cf.\ \hyperref[thm:JPPSconductor]{Theorems \ref*{thm:JPPSconductor}} and \ref{thm:nonarchimedeanKtype}}]
\label{thm:conductornewform}
Let $(\pi,V_{\pi})$ be an induced representation of Whittaker type of $\GL_n(F)$. There exists a minimal nonnegative integer $m$ for which $V_{\pi}(m)^{K_{n - 1}}$ is nontrivial. For this minimal value of $m$, $V_{\pi}(m)^{K_{n - 1}}$ is one-dimensional.
\end{theorem}

That is, we show that among the $K$-types $\tau \in \widehat{K_n}$ of $\pi$ for which $V_{\pi}^{\tau,K_{n - 1}}$ is nontrivial, there exists a unique such $K$-type $\tau^{\circ}$ minimising $\deg \tau$; moreover, $V_{\pi}^{\tau^{\circ},K_{n - 1}}$ is one-dimensional.

\begin{definition}[{Cf.\ \hyperref[def:condnewform]{Definitions \ref*{def:condnewform}} and \ref{def:nonarchimedeanKtype}}]
\label{def:archcondnewfom}
We call $\deg \tau^{\circ}$ the conductor exponent of $\pi$ and denote it by $c(\pi)$, and the nonzero vector $v^{\circ} \in V_{\pi}(c(\pi))^{K_{n - 1}} = V_{\pi}^{\tau^{\circ},K_{n - 1}}$, unique up to scalar multiplication, is called the newform of $\pi$. The $K$-type $\tau^{\circ}$ containing the newform is called the newform $K$-type of $\pi$.
\end{definition}

Analogously to the nonarchimedean theory of the conductor exponent, we regard $c(\pi)$ as quantifying the extent to which $\pi$ is ramified; once again, the conductor exponent $c(\pi)$ is zero if and only if $\pi$ is unramified, so that $\pi$ is a spherical representation. Moreover, the fact that $v^{\circ}$ is unique up to scalar multiplication may be thought of as a multiplicity-one theorem for newforms.

\begin{remark}
In general, the newform $K$-type $\tau^{\circ}$ is not equal to the minimal $K$-type (in the sense of Vogan \cite{Vog81}) of $\pi$, and $\deg \tau^{\circ}$ is not equal to the Vogan norm of the minimal $K$-type. On the other hand, we shall see that the Howe degree of the minimal $K$-type \emph{is} equal to $\deg \tau^{\circ} = c(\pi)$. This gives another way to define the conductor exponent, albeit with the downside that the connection between the minimal $K$-type and the newform is unclear.
\end{remark}

\begin{remark}
\label{rem:permutation}
The conductor exponent and newform $K$-type of an induced representation of Whittaker type $\pi = \pi_1 \boxplus \cdots \boxplus \pi_r$ remain unchanged when $\pi$ is replaced by $\pi_{\sigma(1)} \boxplus \cdots \boxplus \pi_{\sigma(r)}$ for any permutation $\sigma \in S_r$.
\end{remark}

\subsection{Oldforms}

We also define oldforms in the archimedean setting.

\begin{definition}
The space of oldforms of exponent $m > c(\pi)$ is $V_{\pi}(m)^{K_{n - 1}}$.
\end{definition}

\begin{theorem}[{Cf.\ \hyperref[thm:Reeder]{Theorems \ref*{thm:Reeder}} and \ref{thm:nonarchimedeanoldformKtype}}]
\label{thm:oldforms}
Let $(\pi,V_{\pi})$ be an induced representation of Whittaker type of $\GL_n(F)$ with $n \geq 2$. We have that
\[\dim V_{\pi}(m)^{K_{n - 1}} = \begin{dcases*}
\binom{\frac{m - c(\pi)}{2} + n - 2}{n - 2} & if $m \geq c(\pi)$ and $m \equiv c(\pi) \pmod{2}$,	\\
0 & otherwise.
\end{dcases*}\]
\end{theorem}

More precisely, we prove that when $m \geq c(\pi)$ with $m \equiv c(\pi) \pmod{2}$, either $\Hom_{K_n}(\tau, \pi|_{K_n})$ or $\Hom_{K_{n - 1}}(1,\tau|_{K_{n - 1}})$ is trivial for all but one $K$-type $\tau$ with $\deg \tau = m$, and for this particular $K$-type,
\begin{align*}
\dim \Hom_{K_n}\left(\tau, \pi|_{K_n}\right) & = \binom{\frac{m - c(\pi)}{2} + n - 2}{n - 2},	\\
\dim \Hom_{K_{n - 1}}\left(1,\tau|_{K_{n - 1}}\right) & = 1.
\end{align*}
The fact that $V_{\pi}(m)^{K_{n - 1}}$ is trivial when $m \equiv c(\pi) + 1 \pmod{2}$ follows more generally from a result of Fan \cite{Fan18}, namely that for $\tau \in \widehat{K_n}$, $V_{\pi}^{\tau}$ is trivial whenever $\deg \tau \equiv c(\pi) + 1 \pmod{2}$.

We may think of $\bigoplus_{j = 0}^{m} V_{\pi}(j)^{K_{n - 1}}$ as being the archimedean analogue of $V_{\pi}^{K_1(\pp^m)}$. This space has dimension $\binom{\lfloor \frac{m - c(\pi)}{2} \rfloor + n - 1}{n - 1}$ if $m \geq c(\pi)$ and dimension zero otherwise. This aligns closely with the dimension of $V_{\pi}^{K_1(\pp^m)}$ as in \hyperref[thm:Reeder]{Theorem \ref*{thm:Reeder}}, namely $\binom{m - c(\pi) + n - 1}{n - 1}$ if $m \geq c(\pi)$ and zero otherwise.

\begin{remark}
For nonarchimedean $F$, Reeder \cite{Ree91} explicitly describes the space of oldforms as the image of the space of newforms under the action of certain elements of the Hecke algebra. It would be of interest to generalise this to the archimedean setting. In particular, when viewed in the Whittaker model, it would be of interest to describe oldforms in terms of the action of certain differential operators on the newform. For $\GL_2(\R)$, these are simply the Maa\ss{} raising and lowering operators, while the theory for $\GL_3(\R)$ is explored in \cite{BM19}.
\end{remark}

\subsection{Additivity and Inductivity of the Conductor Exponent}

Associated to any induced representation of Whittaker type $\pi$ of $\GL_n(F)$ is the epsilon factor $\e(s,\pi,\psi)$; with $\psi$ chosen as in \hyperref[sect:HaarF]{Section \ref*{sect:HaarF}}, the epsilon factor is an integral power of $i$ via \eqref{eq:epsilonfactorise} and \eqref{eq:epsilonfactor}. This integer is only determined modulo $4$, in contrast to the nonarchimedean setting. In this regard, the following theorem is not so instructive as \hyperref[thm:nonarchimedeanepsilonconductor]{Theorem \ref*{thm:nonarchimedeanepsilonconductor}}.

\begin{theorem}[{Cf.\ \hyperref[thm:nonarchimedeanepsilonconductor]{Theorem \ref*{thm:nonarchimedeanepsilonconductor}}}]
\label{thm:epsilonconductor}
Let $\pi$ be an induced representation of Whittaker type of $\GL_n(F)$. The epsilon factor $\e(s,\pi,\psi)$ is equal to $i^{-c(\pi)}$, where $c(\pi)$ denotes the conductor exponent of $\pi$.
\end{theorem}

Upon decomposing $\pi$ as the isobaric sums $\pi = \pi_1 \boxplus \cdots \boxplus \pi_r$, we have the identity
\[i^{-c(\pi)} = \e(s,\pi,\psi) = \prod_{j = 1}^{r} \e(s,\pi_j,\psi) = i^{-c(\pi_1) - \cdots - c(\pi_r)}\]
via \eqref{eq:epsilonfactor}. This is only enough to conclude that $c(\pi)$ is congruent to $c(\pi_1) + \cdots + c(\pi_r)$ modulo $4$, rather than an equality of conductor exponents. Nonetheless, we still prove by different means that the conductor exponent is additive with respect to isobaric sums.

\begin{theorem}[{Cf.\ \hyperref[thm:nonarchimedeanadditiveconductor]{Theorem \ref*{thm:nonarchimedeanadditiveconductor}}}]
\label{thm:additiveconductor}
For an induced representation of Whittaker type $\pi = \pi_1 \boxplus \cdots \boxplus \pi_r$ of $\GL_n(F)$, we have that
\[c(\pi) = \sum_{j = 1}^{r} c(\pi_j).\]
\end{theorem}

The epsilon factor is again inductive in degree zero, so that if $\pi$ and $\pi'$ are induced representations of Whittaker type of $\GL_n(\C)$ and $\AI_{\C/\R} \pi$ and $\AI_{\C/\R} \pi'$ denote the induced representation of Whittaker type of $\GL_{2n}(\R)$ obtained by induction \cite{Hen10}, then
\[\frac{\e\left(s,\AI_{\C/\R} \pi, \psi\right)}{\e\left(s,\AI_{\C/\R} \pi',\psi\right)} = \frac{\e\left(s,\pi,\psi \circ \Tr_{\C/\R}\right)}{\e\left(s,\pi',\psi \circ \Tr_{\C/\R}\right)}.\]
Taking $\pi'$ to be the isobaric sum of $n$ copies of the trivial representation, we find that
\[c\left(\AI_{\C/\R} \pi\right) \equiv c(\pi) + n \pmod{4}.\]
We prove by different means that this congruence can be replaced by an equality.

\begin{theorem}[{Cf.\ \hyperref[thm:nonarchimedeaninductiveconductor]{Theorem \ref*{thm:nonarchimedeaninductiveconductor}}}]
\label{thm:inductiveconductor}
For an induced representation of Whittaker type $\pi$ of $\GL_n(\C)$, we have that
\[c\left(\AI_{\C/\R} \pi\right) = c(\pi) + n.\]
\end{theorem}

\subsection{Test Vectors for Rankin--Selberg Integrals}

One can once again ask about the existence of test vectors for $\GL_n \times \GL_m$ Rankin--Selberg integrals. This is known for spherical representations $\pi$ and $\pi'$ with $m \in \{n - 1,n\}$, due to Stade \cite[Theorem 3.4]{Sta01}, \cite[Theorem 1.1]{Sta02} (see also \cite[Lemma 4.2]{GLO08} and \cite[Theorem 3.1]{IsSt13}), while recent work of Ishii and Miyazaki deals with the case of pairs of principal series representations of certain specific forms \cite[Theorem 2.5]{IM22}. We prove the existence of a test vector for nonspherical representations $\pi$ of $\GL_n(F)$ provided that $\pi'$ is a spherical representation of Langlands type of $\GL_{n - 1}(F)$ and $W' = W^{\prime\circ} \in \WW\left(\pi',\overline{\psi}\right)$ is the spherical vector normalised as in \hyperref[sect:sphericalrep]{Section \ref*{sect:sphericalrep}}.

\begin{theorem}[{Cf.\ \hyperref[thm:JPPSWhittaker]{Theorem \ref*{thm:JPPSWhittaker}}}]
\label{thm:testvector}
Let $\pi$ be an induced representation of Langlands type of $\GL_n(F)$ with $n \geq 2$. There exists a Whittaker function $W \in \WW(\pi,\psi)$ such that for any spherical representation of Langlands type $\pi'$ of $\GL_{n - 1}(F)$ with spherical Whittaker function $W^{\prime\circ} \in \WW(\pi',\overline{\psi})$,
\[\Psi(s,W,W^{\prime\circ}) = L(s,\pi \times \pi')\]
for $\Re(s)$ sufficiently large.

Moreover, there exists a unique such function $W^{\circ} \in \WW(\pi,\psi)$ that additionally satisfies
\[W^{\circ}\left(g\begin{pmatrix} k & 0 \\ 0 & 1 \end{pmatrix}\right) = W^{\circ}(g)\]
for all $k \in K_{n - 1}$. Up to multiplication by a scalar, this function is the newform $v^{\circ} \in V_{\pi}(c(\pi))^{K_{n - 1}}$ viewed in the Whittaker model.
\end{theorem}

We again call $W^{\circ}$ the Whittaker newform. The normalisation of the Whittaker newform is specified in \hyperref[def:cannorm]{Definition \ref*{def:cannorm}}.

Stade \cite[Theorem 1.1]{Sta02} has shown that for all spherical representations $\pi,\pi'$ of $\GL_n(\R)$,
\[\Psi(s,W^{\circ},W^{\prime\circ},\Phi^{\circ}) = L(s,\pi \times \pi'),\]
where $\Phi^{\circ} \in \Ss(\Mat_{1 \times n}(\R))$ is given by $\Phi^{\circ}(x) \coloneqq \exp\left(-\pi x \prescript{t}{}{x}\right)$; see also \cite[Lemma 4.1]{GLO08} and \cite[Theorem 3.1]{IsSt13}. We extend this to nonspherical representations $\pi$.

\begin{theorem}[{Cf.\ \hyperref[thm:Kim]{Theorem \ref*{thm:Kim}}}]
\label{thm:GLnxGLn}
Let $\pi$ be an induced representation of Langlands type of $\GL_n(F)$. For every spherical representation of Langlands type $\pi'$ of $\GL_n(F)$ with spherical Whittaker function $W^{\prime\circ} \in \WW(\pi',\overline{\psi})$, the Whittaker newform $W^{\circ} \in \WW(\pi,\psi)$ of $\pi$ satisfies
\[\Psi(s,W^{\circ},W^{\prime\circ},\Phi^{\circ}) = L(s,\pi \times \pi')\]
for $\Re(s)$ sufficiently large, where $\Phi^{\circ} \in \Ss(\Mat_{1 \times n}(F))$ is given by
\begin{equation}
\label{eq:Phicirc}
\left(\dim \tau^{\circ}\right) \overline{P^{\circ}}(x) \exp\left(-d_F \pi x \prescript{t}{}{\overline{x}}\right), \qquad d_F \coloneqq [F:\R] = \begin{dcases*}
1 & if $F = \R$,	\\
2 & if $F = \C$,
\end{dcases*}
\end{equation}
with $P^{\circ}$ a homogeneous harmonic polynomial that depends only on the newform $K$-type $\tau^{\circ}$ of $\pi$.
\end{theorem}

We discuss this distinguished polynomial $P^{\circ}$ in further detail in \hyperref[sect:hhp]{Section \ref*{sect:hhp}}; it is the unique homogeneous harmonic polynomial that is right $K_{n - 1}$-invariant and satisfies $P^{\circ}(e_n) = 1$ in the vector space of such polynomials that forms a model of $\tau^{\circ}$.

In the nonarchimedean setting, it is not only the case that the $\GL_n \times \GL_{n - 1}$ Rankin--Selberg integral $\Psi(s,W^{\circ},W^{\prime\circ})$ is equal to $L(s,\pi \times \pi')$ for all spherical representations $\pi'$ of $\GL_{n - 1}(F)$, but also for all spherical representations $\pi'$ of $\GL_m(F)$ and all $m \in \{1, \ldots, n - 1\}$; see \hyperref[thm:KimMatringe]{Theorem \ref*{thm:KimMatringe}}. In the archimedean setting, on the other hand, it is widely believed (see, for example, \cite[Section 2.6]{Bum89}) that $\Psi(s,W,W')$ is \emph{never} equal to $L(s,\pi \times \pi')$ when $\pi'$ is an induced representation of Whittaker type of $\GL_m(F)$ with $m \in \{1, \ldots, n - 2\}$ and the Whittaker functions $W \in \WW(\pi,\psi)$ and $W' \in \WW(\pi',\overline{\psi})$ are right $K_n$- and $K_m$-finite respectively.

Notably, Ishii and Stade \cite[Theorem 3.2]{IsSt13}, furthering the work of Hoffstein and Murty \cite{HM89}, have shown that for all spherical representations $\pi$ of $\GL_n(\R)$ and $\pi'$ of $\GL_{n - 2}(\R)$,
\begin{equation}
\label{eq:GLnxGLn-2}
\Psi(s,W^{\circ},W^{\prime\circ}) = L(s,\pi \times \pi') \frac{1}{2\pi i} \int_{\sigma - i\infty}^{\sigma + i\infty} \frac{L\left(w,\widetilde{\pi}\right)}{2 L(s + w,\pi')} \, dw
\end{equation}
for some sufficiently large $\sigma$. Though we do not include a proof, our methods may also be used to show the identity \eqref{eq:GLnxGLn-2} when $\pi$ is a ramified induced representation of Langlands type and $W^{\circ}$ is the Whittaker newform.

\begin{remark}
Just as in the nonarchimedean setting, little is known about test vectors for archimedean Rankin--Selberg integrals when $\pi'$ is ramified. The nonarchimedean method of Kim \cite[Proposition 2.2.2]{Kim10} remains valid in the archimedean setting in showing that Whittaker newforms are not test vectors when $\pi'$ is ramified: if $\pi'$ is a ramified representation of $\GL_m(F)$ with $m < n$ and $W^{\circ},W^{\prime\circ}$ are newforms of $\pi$ and $\pi'$ respectively, then $\Psi\left(s,W^{\circ},W^{\prime\circ}\right) = 0$ for all $s \in \C$. The test vector problem has been resolved for $m = n = 2$ by Miyazaki \cite[Theorem 6.1]{Miy18} and Hirano, Ishii, and Miyazaki \cite[Appendix A]{HIM22} and for $m = 2$ and $n = 3$ by Hirano, Ishii, and Miyazaki \cite{HIM16,HIM22}. Ishii and Miyazaki have also resolved this problem for $n$ arbitrary and $m \in \{n,n - 1\}$ when both $\pi$ and $\pi'$ are principal series representations \cite[Theorems 2.5 and 2.9]{IM22}. Akin to the nonarchimedean work of Jo \cite{Jo23}, the author and Jo have resolved a \emph{weakened} version of the test vector problem in \cite{HJ24} when $m = n$, where it is shown that there exists an explicit right $K_n$-finite Schwartz--Bruhat function $\Phi \in \Ss(\Mat_{1 \times n}(F))$ for which $\Psi(s,W^{\circ},W^{\prime\circ},\Phi)$ is a multiple of $L(s,\pi \times \pi')$ by a nonzero polynomial in $s$.
\end{remark}

\subsection{Test Vectors for Godement--Jacquet Zeta Integrals}

Test vectors for archimedean Godement--Jacquet zeta integrals are known to exist; see \cite{Lin18} for $F = \R$ and \cite{Ish19} for the case $F = \C$. We give a new resolution of the test vector problem via newforms.

\begin{remark}
We were unable to verify certain aspects of \cite{Lin18}; cf.\ \cite[Footnote \textdagger]{Hum21}.
\end{remark}

\begin{theorem}[{Cf.\ \hyperref[thm:GJ]{Theorem \ref*{thm:GJ}}}]
\label{thm:GJram}
Let $(\pi,V_{\pi})$ be an induced representation of Whittaker type of $\GL_n(F)$. Let $\beta(g)$ denote the matrix coefficient $\langle \pi(g) \cdot v^{\circ}, \widetilde{v}^{\circ} \rangle$, where $v^{\circ} \in V_{\pi}$ denotes the newform and $\widetilde{v}^{\circ} \in V_{\widetilde{\pi}}$ is the corresponding newform normalised such that $\beta(1_n) = 1$. Then
\[Z(s,\beta,\Phi) = L(s,\pi)\]
for $\Re(s)$ sufficiently large, where $\Phi \in \Ss(\Mat_{n \times n}(F))$ is given by
\begin{equation}
\label{eq:Phi}
\left(\dim \tau^{\circ}\right) \overline{P^{\circ}}(e_n x) \exp\left(-d_F \pi \Tr\left(x \prescript{t}{}{\overline{x}}\right)\right),
\end{equation}
with $P^{\circ}$ a homogeneous harmonic polynomial that depends only on the newform $K$-type $\tau^{\circ}$ of $\pi$.
\end{theorem}

The homogeneous harmonic polynomial $P^{\circ}$ appearing in \hyperref[thm:GJram]{Theorem \ref*{thm:GJram}} is the same as that appearing in \hyperref[thm:GLnxGLn]{Theorem \ref*{thm:GLnxGLn}}.

\section{Further Discussion}
\label{sect:further}

\subsection{Previous Results}

Previous cases of test vectors for $\GL_n \times \GL_n$ and $\GL_n \times \GL_{n - 1}$ Rankin--Selberg integrals over archimedean fields with $n$ arbitrary were only known when either both representations are unramified --- see \cite[Theorem 3.4]{Sta01}, \cite[Lemma 4.2]{GLO08}, and \cite[Theorem 3.2]{IsSt13} --- or when both representations are principal series representations of certain specific forms, due to recent work of Ishii and Miyazaki \cite[Theorems 2.5 and 2.9]{IM22}. Jacquet \cite[Theorem 2.7]{Jac09} has shown that for each pair $\pi,\pi'$ of generic irreducible Casselman--Wallach representations of $\GL_n(F),\GL_{n - 1}(F)$, there exists a finite collection $\{W_j\} \subset \WW(\pi,\psi)$ and $\{W_j'\} \subset \WW(\pi',\overline{\psi})$ of right $K_n$- and $K_{n - 1}$-finite Whittaker functions for which $\sum_j \Psi(s,W_j,W_j') = L(s,\pi \times \pi')$; a similar result (additionally involving Schwartz functions $\{\Phi_j\} \subset \Ss(\Mat_{1 \times n}(F))$) also holds for pairs of representations of $\GL_n(F)$.

Explicit descriptions of specific Whittaker functions for ramified principal series representations of $\GL_n(\R)$ are given in \cite{IO14} and of $\GL_3(\C)$ in \cite{HO09}, while recursive formul\ae{} for specific Whittaker functions for principal series representations of $\GL_n(\R)$ and $\GL_n(\C)$ are given in \cite{IM22}. In all three cases, these are Whittaker functions in the minimal $K$-type; as observed in \cite[Theorem 6.1]{HIM12}, such a Whittaker function $W$ is generally not a test vector for Rankin--Selberg integrals $\Psi(s,W,W^{\prime\circ})$ when $W^{\prime\circ}$ is the spherical vector of a spherical representation.

For $n \leq 3$ and $m \leq 2$, the state of affairs of test vectors for $\GL_n \times \GL_m$ Rankin--Selberg integrals is in much better shape. In particular, the existence of the Whittaker newform for $\GL_2(F)$ has previously been proven by Popa \cite{Pop08}, though in a slightly different formulation that we now describe.

Let $\pi$ be a generic irreducible Casselman--Wallach representation of $\GL_2(F)$ with central character $\omega_{\pi}$, let $T$ denote the diagonal torus embedded in $\GL_2(F)$, and let $\chi_T$ be a character of $T$ whose restriction to the centre of $\GL_2(F)$ is $\omega_{\pi}^{-1}$. Given $\tau \in \widehat{K_2}$, we define $\WW^{\tau,T}$ to be the subspace of the $\tau$-isotypic subspace $\WW^{\tau}$ of $\WW(\pi,\psi)$ for which $\pi |_{K_2 \cap T}$ acts by $\chi_T^{-1}$.

\begin{theorem}[{Popa \cite[Proposition 1, Theorem 1]{Pop08}}]
The space $\WW^{\tau,T}$ is at most one-dimensional. Furthermore, if $\chi_T(\diag(a_1,a_2)) = |a_1|^{s - 1/2} |a_2|^{1/2 - s} \omega_{\pi}^{-1}(a_2)$ and $\tau = \tau^{\circ}$ is the $K$-type of lowest degree for which $\WW^{\tau,T}$ is nontrivial, then there exists $W^{\circ} \in \WW^{\tau^{\circ},T}$ for which
\[\Psi(s,W^{\circ},1) = L(s,\pi).\]
\end{theorem}

While this superficially appears to be different than the definition of the Whittaker newform, it is readily checked that
\[\WW^{\tau,T} = \left\{W \in \WW^{\tau} : W\left(g \begin{pmatrix} k & 0 \\ 0 & 1 \end{pmatrix}\right) = W(g) \text{ for all $g \in \GL_2(F)$ and $k \in K_1$}\right\},\]
and so the Whittaker function $W^{\circ}$ is indeed the Whittaker newform and $\WW^{\tau^{\circ},T}$ is equal to the space $V_{\pi}^{K_1(c(\pi))} = V_{\pi}^{\tau^{\circ}|_{K_1}}$ when viewed in the Whittaker model.

Jacquet \cite[Theorem 17.2]{Jac72} has previously determined, in several cases, explicit right $K_2$-finite test vectors $W \in \WW(\pi,\psi)$ and $W' \in \WW(\pi',\overline{\psi})$ and Schwartz functions $\Phi \in \Ss(\Mat_{1 \times 2}(F))$ for the $\GL_2 \times \GL_2$ Rankin--Selberg integral $\Psi(s,W,W',\Phi)$ when $F = \R$; moreover, this includes cases involving generic irreducible Casselman--Wallach representations $\pi,\pi'$ of $\GL_2(F)$ for which $\pi'$ need not be spherical. Further results along these lines have been proven by Zhang \cite[Proposition 2.5.2]{Zha01} and extended to $F = \C$ by Miyazaki \cite[Theorem 6.1]{Miy18}; see also \cite[Appendix A]{HIM22}.

For $n = 3$, Hirano, Ishii, and Miyazaki \cite{HIM16,HIM22} have given an explicit description of right $K_3$- and $K_2$-finite Whittaker functions $W \in \WW(\pi,\psi)$ and $W' \in \WW(\pi',\overline{\psi})$ for which $\Psi(s,W,W') = L(s,\pi \times \pi')$, where $\pi$ and $\pi'$ are \emph{any} generic irreducible Casselman--Wallach representations of $\GL_3(F)$ and $\GL_2(F)$ respectively; in particular, $\pi'$ need not be spherical.

\subsection{Examples and Applications}

\subsubsection{The Conductor Exponent for Essentially Square-Integrable Representations}

For $\pi = \chi^{\kappa} |\cdot|^t$, where $\kappa \in \{0,1\}$ for $F = \R$ and $\kappa \in \Z$ for $F = \C$, the conductor exponent is $c(\pi) = \|\kappa\|$. For $\pi = D_{\kappa} \otimes \left|\det\right|^t$, where $\kappa \geq 2$, the conductor exponent is $c(\pi) = \kappa$. Via \hyperref[thm:additiveconductor]{Theorem \ref*{thm:additiveconductor}}, this allows one to calculate the conductor exponent of any induced representation of Whittaker type.

\subsubsection{$\GL_2$ Examples}

Let us consider the classical setting of automorphic forms on the upper half-plane.

\begin{example}
Let $f$ be a holomorphic cuspidal newform of weight $\kappa \geq 2$, level $q \in \N$, and nebentypus $\psi$ modulo $q$, where $\psi(-1) = (-1)^{\kappa}$. Then the underlying automorphic representation of $\GL_2(\A_{\Q})$, as described in \cite[Section 3]{Cas73}, has as its archimedean component a discrete series of weight $\kappa$, $D_{\kappa}$. In particular, the conductor exponent of this archimedean component is simply the weight $\kappa$.
\end{example}

\begin{example}
\label{ex:Maass}
Let $f$ be a Maa\ss{} cuspidal newform of weight $\kappa \in \{0,1\}$, level $q \in \N$, and nebentypus $\psi$, where $\psi(-1) = (-1)^{\kappa}$. The archimedean component of the underlying automorphic representation of $\GL_2(\A_{\Q})$ is one of three possible representations:
\begin{itemize}
\item If $\kappa = 1$, then the archimedean component is a principal series representation of the form $|\cdot|^{it_f} \boxplus \chi |\cdot|^{-it_f}$ or $\chi |\cdot|^{it_f} \boxplus |\cdot|^{-it_f}$, where $t_f \in \R$ is the spectral parameter of $f$. The conductor exponent is $1$.
\item If $\kappa = 0$ and $f$ is even, so that $f(-\overline{z}) = f(z)$, then the archimedean component is of the form $|\cdot|^{it_f} \boxplus |\cdot|^{-it_f}$, where $t_f \in \R \cup i(-1/2,1/2)$. The conductor exponent is $0$.
\item If $\kappa = 0$ and $f$ is odd, so that $f(-\overline{z}) = -f(z)$, then the archimedean component is of the form $\chi |\cdot|^{it_f} \boxplus \chi |\cdot|^{-it_f}$, where again $t_f \in \R \cup i(-1/2,1/2)$. The conductor exponent is $2$.
\end{itemize}
\end{example}

\begin{remark}
The newform $K$-type of the archimedean component of the automorphic representation associated to an odd Maa\ss{} form of weight zero is \emph{not} its minimal $K$-type: the newform $K$-type is $\tau_{(2,0)}$, whereas the minimal $K$-type is the determinant representation $\tau_{(1,1)}$. Classically, this manifests itself via the fact that an odd Maa\ss{} form $f$ of weight zero is \emph{not} a test vector for the Rankin--Selberg integral
\[\int_{0}^{\infty} f(iy) y^{s - \frac{1}{2}} \, \frac{dy}{y}.\]
Instead, one must use Maa\ss{} raising and lowering operators on such a Maa\ss{} form to obtain a test vector.
\end{remark}

\subsubsection{A Non-Application: the Analytic Conductor and Analytic Newvector}

We have introduced in \hyperref[sect:condexp]{Section \ref*{sect:condexp}} the notion of the conductor exponent of a generic irreducible Casselman--Wallach representation $\pi$ as a measure of the extent of ramification. This can also be thought of as quantifying the (representation-theoretic) complexity of $\pi$.

There is another well-known quantification of the complexity of $\pi$: the analytic conductor. First introduced by Iwaniec and Sarnak \cite[(31)]{IwSa00}, this is defined on essentially square-integrable representations $\pi$ by
\[\qq(\pi) \coloneqq \begin{dcases*}
1 + \left\|t + \kappa\right\| & if $F = \R$ and $\pi = \chi^{\kappa} |\cdot|_{\R}^t$,	\\
\left(1 + \left\|t + \frac{\kappa - 1}{2}\right\|\right)\left(1 + \left\|t + \frac{\kappa + 1}{2}\right\|\right) & if $F = \R$ and $\pi = D_{\kappa} \otimes \left|\det\right|_{\R}^t$,	\\
\left(1 + \left\|t + \frac{\|\kappa\|}{2}\right\|\right)^2 & if $F = \C$ and $\pi = \chi^{\kappa} |\cdot|_{\C}^t$,
\end{dcases*}\]
and extended to arbitrary induced representations of Whittaker type $\pi = \pi_1 \boxplus \cdots \boxplus \pi_r$ via multiplicativity:
\[\qq(\pi) \coloneqq \prod_{j = 1}^{r} \qq(\pi_j).\]
One can relate this to the asymptotic behaviour of the local $\gamma$-factor $\gamma(s,\pi,\psi)$ as $s$ tends to $1/2$; see in particular \cite[Lemma 3.1 (1)]{JN19} and also \cite[Section 3.1]{CFKRS05} (the latter also alludes to how this similarly encompasses the conductor at the nonarchimedean places).

These two quantifications are entirely distinct. This can be seen most clearly for spherical representations $\pi = |\cdot|^{it_f} \boxplus |\cdot|^{-it_f}$ of $\GL_2(\R)$ occurring as the archimedean component of an automorphic representation associated to an even Maa\ss{} form $f$ of weight zero. As discussed in \hyperref[ex:Maass]{Example \ref*{ex:Maass}}, the conductor exponent of such a representation $\pi$ is $c(\pi) = 0$; on the other hand, the analytic conductor is $\qq(\pi) = (1 + \|t_f\|)^2$. This distinction boils down to the fact that the analytic conductor $\qq(\pi)$ is better thought of as being a measure of the size of the $L$-function $L(s,\pi)$ rather than a measure of the complexity of the representation itself.

Similarly, the newform is unrelated to the analytic newvectors introduced by Jana and Nelson in \cite{JN19}. Analytic newvectors are a different archimedean analogue of newforms; they are a \emph{family} of vectors in a generic irreducible unitary Casselman--Wallach representation that are \emph{almost} invariant under certain subsets (but not subgroups) of $\GL_n(F)$ that are archimedean analogues of the congruence subgroups $K_1(\pp^m)$. Notably, analytic newvectors are not necessarily $K_n$-finite, though they are $K_{n - 1}$-invariant; in particular, they lie in the subspace spanned by the newform and all oldforms. Note that Jana and Nelson call the newform an ``algebraic newvector''.

\subsubsection{Global Eulerian Integrals}

We mention a global application of the resolution of the test vector problem for Rankin--Selberg integrals. In order to do so, we require some discussion of global automorphic forms.

Let $\psi_{\A_{\Q}}$ denote the standard additive character of $\A_{\Q}$ that is unramified at every place of $\Q$. Given a global number field $F$, define the additive character $\psi_{\A_F} \coloneqq \psi_{\A_{\Q}} \circ \Tr_{\A_F/\A_{\Q}}$ of $\A_F$. The conductor of $\psi_{\A_F}$ is the inverse different $\df^{-1}$ of $F$; furthermore, there is a finite id\`{e}le $d \in \A_F^{\times}$ representing $\df$ such that $\psi_{\A_F} = \bigotimes_v \psi_v^{d_v}$, where the additive character $\psi_v^{d_v}$ of $F_v$ of conductor $\df_v^{-1}$ is defined by $\psi_v^{d_v}(x) \coloneqq \psi_v(d_v x)$ with $\psi_v$ an unramified additive character of $F_v$ as in \hyperref[sect:HaarF]{Section \ref*{sect:HaarF}}.

Let $(\pi,V_{\pi})$ be a cuspidal automorphic representation of $\GL_n(\A_F)$ with $n \geq 2$, where $V_{\pi}$ is a space of automorphic forms on $\GL_n(\A_F)$. Let $W_{\varphi} \in \WW(\pi,\psi_{\A_F})$ denote the Whittaker function of $\varphi \in V_{\pi}$, so that
\[W_{\varphi}(g) \coloneqq \int\limits_{\Ngp_n(F) \backslash \Ngp_n(\A_F)} \varphi(ug) \overline{\psi_{\A_F,n}}(u) \, du,\]
where $du$ denotes the Tamagawa measure. If $W_{\varphi}$ is a pure tensor, we may write $W_{\varphi} = \prod_v W_{\varphi,v}$ with $W_{\varphi,v} \in \WW(\pi_v,\psi_v^{d_v})$, where the generic irreducible admissible smooth representations $\pi_v$ are the local components of the automorphic representation $\pi = \bigotimes_v \pi_v$. We define the global newform $\varphi^{\circ} \in V_{\pi}$ to be such that at each place $v$ of $F$, we have that $W_{\varphi^{\circ},v}(g_v) \coloneqq W_v(\diag(d_v^{n - 1},\ldots,d_v,1) g_v) \in \WW(\pi_v,\psi_v^{d_v})$ with $W_v \in \WW(\pi_v,\psi_v)$ the (local) Whittaker newform as in \hyperref[thm:JPPSWhittaker]{Theorems \ref*{thm:JPPSWhittaker}} and \ref{thm:testvector}.

\begin{proposition}
Let $F$ be a global number field of absolute discriminant $D_{F/\Q}$, and let $(\pi,V_{\pi})$ be a cuspidal automorphic representation of $\GL_n(\A_F)$ with $n \geq 2$. Then the global newform $\varphi^{\circ} \in V_{\pi}$ is such that for all cuspidal automorphic representations $(\pi',V_{\pi'})$ of $\GL_{n - 1}(\A_F)$ that are everywhere unramified with global newforms $\varphi^{\prime\circ} \in V_{\pi'}$ and for $\Re(s)$ sufficiently large, the global $\GL_n \times \GL_{n - 1}$ Rankin--Selberg integral
\[I(s,\varphi^{\circ},\varphi^{\prime\circ}) \coloneqq \int\limits_{\GL_{n - 1}(F) \backslash \GL_{n - 1}(\A_F)} \varphi^{\circ} \begin{pmatrix} g & 0 \\ 0 & 1 \end{pmatrix} \varphi^{\prime\circ}(g) \left|\det g\right|_{\A_F}^{s - \frac{1}{2}} \, dg\]
with $dg$ the Tamagawa measure is equal, up to multiplication by a constant dependent only on $F$, to the product of $\omega_{\pi'}^{-1}(d) D_{F/\Q}^{\frac{n(n - 1)s}{2}}$, where $\omega_{\pi'}$ denotes the central character of $\pi'$, and of the global completed Rankin--Selberg $L$-function $\Lambda(s, \pi \times \pi') \coloneqq \prod_v L(s,\pi_v \times \pi_v')$.
\end{proposition}

\begin{remark}
Via a regularisation process due to Ichino and Yamana \cite[Theorem 1.1]{IY15}, a suitably modified version of this result still holds even if either $\pi$ or $\pi'$ (or both) is not cuspidal.
\end{remark}

\begin{proof}
By unfolding,
\begin{align*}
I(s,\varphi^{\circ},\varphi^{\prime\circ}) & = \int\limits_{\Ngp_{n - 1}(\A_F) \backslash \GL_{n - 1}(\A_F)} W_{\varphi^{\circ}}\begin{pmatrix} g & 0 \\ 0 & 1 \end{pmatrix} W_{\varphi^{\prime\circ}}(g) \left|\det g\right|_{\A_F}^{s - \frac{1}{2}} \, dg	\\
& = c_{F/\Q} \prod_v \Psi(s,W_{\varphi^{\circ},v},W_{\varphi^{\prime\circ},v}),
\end{align*}
where the local Rankin--Selberg integrals $\Psi(s,W_{\varphi^{\circ},v},W_{\varphi^{\prime\circ},v})$ are as in \eqref{eq:RankinSelbergnm}, and $c_{F/\Q} > 0$ is an absolute constant dependent only on $F$ that arises from the compatibility of the global Tamagawa measure on $\Ngp_{n - 1}(\A_F) \backslash \GL_{n - 1}(\A_F)$ compared to the local Haar measure on $\Ngp_{n - 1}(F_v) \backslash \GL_{n - 1}(F_v)$ as normalised in \hyperref[sect:groupsHaar]{Section \ref*{sect:groupsHaar}}. Upon making the change of variables $g_v' \mapsto \diag(d_v^{1 - n},\ldots,d_v^{-1}) g_v'$ and using the fact that $W_v'(\diag(d_v^{-1},\ldots,d_v^{-1}) g_v') = \omega_{\pi_v'}^{-1}(d_v) W_v'(g_v')$, we see that
\[\Psi(s,W_{\varphi^{\circ},v},W_{\varphi^{\prime\circ},v}) = \omega_{\pi_v'}^{-1}(d_v) |d_v|_v^{-\frac{n(n - 1) s}{2} + \frac{n(n - 1)(2n - 1)}{12}} \Psi(s,W_v,W_v').\]
The result now follows from \hyperref[thm:JPPSWhittaker]{Theorems \ref*{thm:JPPSWhittaker}} and \ref{thm:testvector}.
\end{proof}

A similar result holds for global $\GL_n \times \GL_n$ Rankin--Selberg integrals via \hyperref[thm:Kim]{Theorems \ref*{thm:Kim}} and \ref{thm:GLnxGLn}. These involve an Eisenstein series
\[E(g,s;\Phi,\eta) = \left|\det g\right|_{\A_F}^s \int_{F^{\times} \backslash \A_F^{\times}} \Theta_{\Phi}'(a,g) \eta(a) |a|_{\A_F}^{ns} \, d^{\times} a,\]
where $\eta : F^{\times} \backslash \A_F^{\times} \to \C^{\times}$ is a Hecke character, $\Phi \in \Ss(\Mat_{1 \times n}(\A_F))$ is a Schwartz--Bruhat function, and
\[\Theta_{\Phi}(a,g) \coloneqq \sum_{\xi \in \Mat_{1 \times n}(F)} \Phi(a \xi g), \qquad \Theta_{\Phi}'(a,g) \coloneqq \Theta_{\Phi}(a,g) - \Phi(0).\]

\begin{proposition}
Let $F$ be a number field of absolute discriminant $D_{F/\Q}$, and let $(\pi,V_{\pi})$ and $(\pi',V_{\pi'})$ be automorphic representations of $\GL_n(\A_F)$ with central characters $\omega_{\pi}$ and $\omega_{\pi'}$. Suppose that at least one of $\pi$ and $\pi'$ is cuspidal and that $\pi$ and $\pi'$ have disjoint ramification. Then there exists a right $K$-finite Schwartz--Bruhat function $\Phi \in \Ss(\Mat_{1 \times n}(\A_F))$ such that for $\Re(s)$ sufficiently large, the global $\GL_n \times \GL_n$ Rankin--Selberg integral
\[I(s,\varphi^{\circ},\varphi^{\prime\circ},\Phi) \coloneqq \int\limits_{\Zgp(\A_F) \GL_n(F) \backslash \GL_n(\A_F)} \varphi^{\circ}(g) \varphi^{\prime\circ}(g) E(g,s; \Phi, \omega_{\pi} \omega_{\pi'}) \, dg\]
with $dg$ the Tamagawa measure is equal, up to multiplication by a constant dependent only on $F$, to the product of $D_{F/\Q}^{\frac{n(n - 1)s}{2}}$ and of the global completed Rankin--Selberg $L$-function $\Lambda(s, \pi \times \pi') \coloneqq \prod_v L(s,\pi_v \times \pi_v')$.
\end{proposition}

\begin{remark}
When both $\pi$ and $\pi'$ are noncuspidal, it ought to be possible to prove a suitably modified version of this result by extending to $\GL_n$ a regularisation process due to Zagier \cite[Theorem]{Zag82} for $\GL_2$.
\end{remark}

We also may show the existence of a test vector for the global Godement--Jacquet zeta integral via \hyperref[thm:GJ]{Theorems \ref*{thm:GJ}} and \ref{thm:GJram}.

\begin{proposition}
Let $F$ be a number field, and let $(\pi,V_{\pi})$ be an automorphic representation of $\GL_n(\A_F)$. Then there exists a matrix coefficient $\beta(g) \coloneqq \langle \pi(g) \cdot \varphi^{\circ}, \widetilde{\varphi}^{\circ}\rangle$ and a bi-$K$-finite Schwartz--Bruhat function $\Phi \in \Ss(\Mat_{n \times n}(\A_F))$ such that for $\Re(s)$ sufficiently large, the global Godement--Jacquet zeta integral
\[Z(s,\beta,\Phi) \coloneqq \int_{\GL_n(\A_F)} \beta(g) \Phi(g) \left|\det g\right|_{\A_F}^{s + \frac{n - 1}{2}} \, dg\]
with $dg$ the Tamagawa measure is equal, up to multiplication by a constant dependent only on $F$, to the global completed standard $L$-function $\Lambda(s,\pi) \coloneqq \prod_v L(s,\pi_v)$.
\end{proposition}

A similar result also holds for the Piatetski-Shapiro--Rallis integral \cite{P-SR87}, since this doubling integral is equal to the Godement--Jacquet zeta integral \cite[Proposition 3.2]{P-SR87}.

\begin{remark}
In fact, we show something slightly stronger, namely that
\[\int_{\GL_n(\A_F)} \varphi^{\circ}(hg) \Phi(g) \left|\det g\right|_{\A_F}^{s + \frac{n - 1}{2}} \, dg\]
is equal, up to multiplication by a constant dependent only on $F$, to $\Lambda(s,\pi) \varphi^{\circ}(h)$ for all $h \in \GL_n(\A_F)$.
\end{remark}

\subsection{Strategy of the Proofs}

\subsubsection{Nonarchimedean Strategies}

As discussed in \hyperref[sect:nonarchimedean]{Section \ref*{sect:nonarchimedean}}, there are several approaches towards developing nonarchimedean newform theory for $\GL_n$. We briefly examine the challenges in transporting each of these methods to the archimedean setting.

Matringe \cite{Mat13} uses the nonarchimedean theory of Bernstein--Zelevinsky derivatives to explicitly construct the Whittaker newform. The archimedean theory of Bernstein--Zelevinsky derivatives is less well-developed (though see \cite{AGS15} and \cite{Cha15} for two different approaches), and it does not seem straightforward to transport Matringe's proof to this setting.

The method of Jacquet \cite{Jac12} does not use Bernstein--Zelevinsky derivatives; the proof, however, is nonconstructive and only shows the existence of a $K_{n - 1}$-invariant test vector for $\GL_n \times \GL_{n - 1}$ Rankin--Selberg integrals. In the archimedean setting, it seems difficult to describe the newform $K$-type when one only knows of the existence of such a test vector.

Finally, the method of Miyauchi \cite{Miy14} assumes the existence of the newform in $V_{\pi}^{K_1(\pp^{c(\pi)})}$ and uses the action of certain Hecke operators to derive a recursive relation between certain values of the newform in the Whittaker model; using this, one can then show that the newform in the Whittaker model is a test vector for Rankin--Selberg integrals. The archimedean analogue of this is to assume the existence of the newform in the newform $K$-type and use the action of certain differential operators (arising from elements of the centre of the universal enveloping algebra) to derive systems of partial differential equations satisfied by the newform in the Whittaker model. This is essentially the approach undertaken in \cite{HO09} and \cite{IO14} (where the Whittaker function studied lies in the minimal $K$-type, not the newform $K$-type); already for $\GL_3(\C)$, however, this leads to immense combinatorial difficulties in solving these systems of partial differential equations.

\subsubsection{The Archimedean Strategy}

We take a different path. To prove the existence of the newform and newform $K$-type, we use Frobenius reciprocity to reduce the problem to branching rules on the associated maximal compact subgroups. Here we benefit from the fact that, unlike in the nonarchimedean setting, irreducible representations of $K$ and explicit branching rules are well-understood, and the induced representations of Whittaker type are particularly easy to describe, since essentially square-integrable representations do not exist for $\GL_n(F)$ with $n \geq 3$. This approach also determines the dimension of spaces of oldforms and yields the additivity and inductivity of the conductor exponent. All of this is proven in \hyperref[sect:newformKtype]{Section \ref*{sect:newformKtype}}.

In order to study the newform in more detail, with an eye towards formul\ae{} for the Whittaker newform that are beneficial for evaluating $\GL_n \times \GL_{n - 1}$ and $\GL_n \times \GL_n$ Rankin--Selberg integrals, we require additional knowledge, given in \hyperref[sect:hhp]{Section \ref*{sect:hhp}}, of a particular model of the newform $K$-type, namely a space of homogeneous harmonic polynomials. Using this, we explicitly construct the newform in the induced model of $\pi$ in \hyperref[sect:newforminduced]{Section \ref*{sect:newforminduced}}.

We give three different constructions of the newform in the induced model: via the Iwasawa decomposition, via convolution sections, and via Godement sections. Each construction has its advantages and disadvantages. The construction via the Iwasawa decomposition is straightforward but lacks a direct relation to Whittaker functions. The construction via convolution sections, following work of Jacquet \cite{Jac04}, gives a recursive formula for the newform in terms of a convolution of the newform itself with an explicit standard Schwartz function; this gives an immediate resolution of the test vector problem for archimedean Godement--Jacquet integrals. Finally, following Jacquet \cite{Jac09}, the newform is shown to be given as an element of a Godement section, which gives a recursive formula for the newform in terms of an integral of a distinguished newform of a representation of $\GL_{n - 1}(F)$ against a particular standard Schwartz function; this is limited to certain induced representations of Whittaker type, but is invaluable as an inductive step.

With these formul\ae{} in hand, we then express the newform in the Whittaker model via the Jacquet integral in \hyperref[sect:newformWhittaker]{Section \ref*{sect:newformWhittaker}}. The usage of convolution sections and Godement sections gives us recursive formul\ae{} for the Whittaker newform. The latter in particular gives what we call a propagation formula: this is a recursive formula for $\GL_n(F)$ Whittaker functions in terms of $\GL_{n - 1}(F)$ Whittaker functions. 

Our expression for the newform via convolution sections gives a quick resolution of the test vector problem for the Godement--Jacquet zeta integral. Our strategy for resolving the test vector problems for $\GL_n \times \GL_n$ and $\GL_n \times \GL_{n - 1}$ Rankin--Selberg integrals follows an approach pioneered by Jacquet \cite{Jac09} (which is also followed in \cite{IM22}). We employ a double induction argument presented in \hyperref[sect:RankinSelberg]{Section \ref*{sect:RankinSelberg}}. This type of argument is due to Jacquet \cite{Jac09} (who, in turn, attributes this strategy to Shalika): it expresses the $\GL_n \times \GL_n$ Rankin--Selberg integral as the product of a $\GL_n \times \GL_{n - 1}$ Rankin--Selberg integral and a $\GL_n$ Godement--Jacquet zeta integral, and similarly expresses the $\GL_n \times \GL_{n - 1}$ Rankin--Selberg integral as a product of a $\GL_{n - 1} \times \GL_{n - 1}$ Rankin--Selberg integral and a $\GL_{n - 1}$ Godement--Jacquet zeta integral. (In fact, we find a slightly more direct approach via convolution sections that masks the presence of Godement--Jacquet zeta integrals.)

\subsubsection{Additional Remarks on the Proofs}

We emphasise that the proofs of \hyperref[thm:testvector]{Theorems \ref*{thm:testvector}} and \ref{thm:GLnxGLn}, given in \hyperref[sect:RankinSelberg]{Section \ref*{sect:RankinSelberg}}, are independent of the proofs in \cite{GLO08,IsSt13,Sta01,Sta02} of the unramified cases. Although our proofs are somewhat involved when ramification is present, they are particularly simple for spherical representations. In particular, these gives proofs of Stade's formul\ae{}, namely the unramified cases of \hyperref[thm:testvector]{Theorems \ref*{thm:testvector}} and \ref{thm:GLnxGLn}. These reproofs of Stade's formul\ae{} also follow from the proofs of \cite[Theorems 2.5 and 2.9]{IM22}; they are essentially implicit in the work of Jacquet \cite{Jac09}.

Notably, we do not explicitly make use of the action of the universal enveloping algebra of the complexified Lie algebra of $\GL_n(F)$ as differential operators on Whittaker functions, nor do we require any calculations involving Mellin transforms. In this regard, our construction of Whittaker functions is entirely distinct to that of much of previous work on archimedean Whittaker functions \cite{HIM12,HIM16,HO09,IO14,IsSt13,Pop08,Sta90,Sta95,Sta01,Sta02}. In particular, our proofs of \hyperref[thm:testvector]{Theorems \ref*{thm:testvector}} and \ref{thm:GLnxGLn} demonstrate that the Jacquet integral \emph{is} adequate for the direct computation of archimedean Rankin--Selberg integrals, contrary to an assertion of Ishii and Oda \cite[p.\ 1288]{IO14}, provided one couples this with the usage of convolution sections and Godement sections.

\section{The Newform \texorpdfstring{$K$}{K}-Type}
\label{sect:newformKtype}

To study the newform $K$-type of an induced representation of Whittaker type $\pi$ of $\GL_n(F)$, as well as determine the dimension of spaces of oldforms, we must determine the dimension of $\Hom_{K_n}(\tau, \pi|_{K_n})$ for each $\tau \in \widehat{K_n}$ for which $\Hom_{K_{n - 1}}(1, \tau|_{K_{n - 1}})$ is nontrivial. This is achieved via branching rules.

\subsection{Branching from \texorpdfstring{$\GL_n(\C)$}{GL\9040\231(C)} to \texorpdfstring{$\Ugp(n)$}{U(n)}}

Let $\pi = \pi_1 \boxplus \cdots \boxplus \pi_n$ be an induced representation of Whittaker type of $\GL_n(\C)$, so that for each $j \in \{1,\ldots,n\}$, $\pi_j = \chi^{\kappa_j} |\cdot|^{t_j}$ for some $\kappa_j \in \Z$ and $t_j \in \C$.

\begin{lemma}
\label{lem:HomUn}
For $\tau \in \widehat{\Ugp(n)}$, we have that
\begin{equation}
\label{eq:HomUn}
\Hom_{\Ugp(n)} \left(\tau, \pi|_{\Ugp(n)}\right) \cong \Hom_{\Ugp(1)^n}\left(\tau |_{\Ugp(1)^n}, \bigboxtimes_{j = 1}^{n} \tau_{\kappa_j}\right).
\end{equation}
\end{lemma}

Here we view $\Ugp(1)^n$ as the subgroup of diagonal matrices in $\Ugp(n)$; it is the maximal compact subgroup of the Levi subgroup $\Mgp_{(1,\ldots,1)}(\C) \cong \GL_1(\C)^n$ of the standard parabolic subgroup $\Pgp(\C) = \Pgp_{(1,\ldots,1)}(\C)$ of $\GL_n(\C)$ from which $\pi$ is induced.

\begin{proof}
Mackey's restriction-induction formula implies that
\[\pi|_{\Ugp(n)} \cong \Ind_{\Ugp(1)^n}^{\Ugp(n)} \pi|_{\Ugp(1)^n},\]
and so by the Frobenius reciprocity theorem,
\[\Hom_{\Ugp(n)} \left(\tau, \pi|_{\Ugp(n)}\right) \cong \Hom_{\Ugp(1)^n}\left(\tau|_{\Ugp(1)^n}, \pi|_{\Ugp(1)^n}\right).\]
It remains to note that
\[\pi|_{\Ugp(1)^n} \cong \bigboxtimes_{j = 1}^{n} \tau_{\kappa_j}.\qedhere\]
\end{proof}

The right-hand side of \eqref{eq:HomUn} is a branching from $\Ugp(n)$ to $\Ugp(1)^n$. This can be understood via iterating the following branching rule from $\Ugp(n)$ to $\Ugp(n - 1) \times \Ugp(1)$.

\begin{lemma}[{\cite[Proposition 10.1]{Pro94}}]
\label{lem:U_{n - 1}xU_1branching}
For $\tau_{\mu} \in \widehat{\Ugp(n)}$ of highest weight $\mu = (\mu_1,\ldots,\mu_n) \in \Lambda_n$,
\[\tau_{\mu}|_{\Ugp(n - 1) \times \Ugp(1)} \cong \bigoplus_{\lambda \in \Lambda_1} \bigoplus_{\substack{\nu = (\nu_1,\ldots,\nu_{n - 1}) \in \Lambda_{n - 1} \\ \mu_1 \geq \nu_1 \geq \mu_2 \geq \cdots \geq \nu_{n - 1} \geq \mu_n \\ \sum_{j = 1}^{n - 1} \nu_j = \sum_{j = 1}^{n} \mu_j - \lambda}} \tau_{\nu} \boxtimes \tau_{\lambda}.\]
In particular,
\[\tau_{\mu}|_{\Ugp(n - 1)} \cong \bigoplus_{\substack{\nu = (\nu_1,\ldots,\nu_{n - 1}) \in \Lambda_{n - 1} \\ \mu_1 \geq \nu_1 \geq \mu_2 \geq \cdots \geq \nu_{n - 1} \geq \mu_n}} \tau_{\nu}.\]
\end{lemma}

\begin{corollary}
\label{cor:Un-1triv}
The restriction to $\Ugp(n - 1)$ of the irreducible representation $\tau_{\mu} \in \widehat{\Ugp(n)}$ of highest weight $\mu = (\mu_1, \ldots, \mu_n) \in \Lambda_n$ contains the trivial representation if and only if $\mu_1 \geq 0$, $\mu_2 = \cdots = \mu_{n - 1} = 0$, and $\mu_n \leq 0$, in which case the trivial representation occurs with multiplicity one.
\end{corollary}

We now iterate the branching rule in \hyperref[lem:U_{n - 1}xU_1branching]{Lemma \ref*{lem:U_{n - 1}xU_1branching}} to determine the multiplicity of a representation $\tau_{\lambda_1} \boxtimes \cdots \boxtimes \tau_{\lambda_n}$ of $\Ugp(1)^n$ in a given representation $\tau_{\mu}$ of $\Ugp(n)$.

\begin{lemma}
\label{lem:U_1^{n}branching}
For $\tau_{\mu} \in \widehat{\Ugp(n)}$ of highest weight $\mu = (\mu_1,0,\ldots,0,\mu_n) \in \Lambda_n$ and for any $\lambda_1,\ldots,\lambda_n \in \Lambda_1$,
\begin{multline*}
\dim \Hom_{\Ugp(1)^n}\left(\tau_{\mu}|_{\Ugp(1)^n}, \bigboxtimes_{j = 1}^{n} \tau_{\lambda_j}\right)	\\
= \begin{dcases*}
\binom{\ell + n - 2}{n - 2} & \parbox{.6\textwidth}{if $\displaystyle \mu_1 = \sum_{j = 1}^{n} \max\{\lambda_j,0\} + \ell$ and $\displaystyle \mu_n = \sum_{j = 1}^{n} \min\{\lambda_j,0\} - \ell$ for some $\ell \in \N_0$,}	\\
0 & otherwise.
\end{dcases*}
\end{multline*}
\end{lemma}

\begin{proof}
We take $\mu = (\mu_1,0,\ldots,0,\mu_n)$ in \hyperref[lem:U_{n - 1}xU_1branching]{Lemma \ref*{lem:U_{n - 1}xU_1branching}} and then iterate this branching rule in order to find that
\[\tau_{\mu}|_{\Ugp(1)^n} \cong \bigoplus_{j = 1}^{n - 1} \bigoplus_{\lambda_j \in \Lambda_1} \bigoplus_{\substack{(\nu_{j,1},0,\ldots,0, \nu_{j,n - j}) \in \Lambda_{n - j} \\ \nu_{j,1} \leq \nu_{j - 1,1}, \ \nu_{j,n - j} \geq \nu_{j - 1,n - j + 1} \\ \nu_{j,1} + \nu_{j,n - j} = \nu_{j - 1,1} + \nu_{j - 1,n - j + 1} - \lambda_j}} \bigboxtimes_{j = 1}^{n} \tau_{\lambda_j},\]
where we define $\nu_{0,1} \coloneqq \mu_1$, $\nu_{0,n} \coloneqq \mu_n$, and $\lambda_n \coloneqq \nu_{n - 1,1}$. By induction, the condition $\nu_{j,1} + \nu_{j,n - j} = \nu_{j - 1,1} + \nu_{j - 1,n - j + 1} - \lambda_j$ implies that
\[\mu_1 + \mu_n - \sum_{k = 1}^{j} \lambda_k = \begin{dcases*}
\nu_{j,1} + \nu_{j,n - j} & for $j \in \{1, \ldots, n - 2\}$,	\\
\nu_{n - 1,1} & for $j = n - 1$,	\\
0 & for $j = n$.
\end{dcases*}\]
The multiplicity of $\tau_{\lambda_1} \boxtimes \cdots \boxtimes \tau_{\lambda_n}$ in $\tau_{\mu}$ is thereby equal to
\[\# \left\{(\nu_{1,1},\ldots,\nu_{n - 2,1}) \in \N_0^{n - 2} : \nu_{j - 1,1} \geq \nu_{j,1} + \max\{\lambda_j,0\} \text{ for all $j \in \{1, \ldots, n - 1\}$}\right\},\]
namely the number of $(n - 2)$-tuples $(\nu_{1,1},\ldots,\nu_{n - 2,1})$ for which the system of inequalities
\[\mu_1 \geq \nu_{1,1} + \max\{\lambda_1,0\} \geq \cdots \geq \nu_{n - 2,1} + \sum_{j = 1}^{n - 2} \max\{\lambda_j,0\} \geq \sum_{j = 1}^{n} \max\{\lambda_j,0\}\]
holds. This is zero unless there exists some $\ell \in \N_0$ such that
\[\mu_1 = \sum_{j = 1}^{n} \max\{\lambda_j,0\} + \ell, \qquad \mu_n = \sum_{j = 1}^{n} \min\{\lambda_j,0\} - \ell,\]
in which case the multiplicity is precisely the number of ordered $(n - 2)$-tuples taking values between $0$ and $\ell$, which is
\[\binom{\ell + n - 2}{n - 2}.\qedhere\]
\end{proof}

With this in hand, we can now explicitly determine the right-hand side of \eqref{eq:HomUn}.

\begin{lemma}
\label{lem:Cnewformoldformconductor}
Suppose that the restriction of $\tau_{\mu} \in \widehat{\Ugp(n)}$ to $\Ugp(n - 1)$ contains the trivial representation. Then the highest weight of $\tau_{\mu}$ is of the form $\mu = (\mu_1,0,\ldots,0,\mu_n) \in \Lambda_n$, the trivial representation occurs with multiplicity one, and
\begin{multline*}
\dim \Hom_{\Ugp(1)^n}\left(\tau_{\mu}|_{\Ugp(1)^n}, \bigboxtimes_{j = 1}^{n} \tau_{\kappa_j}\right)	\\
= \begin{dcases*}
\binom{\ell + n - 2}{n - 2} & \parbox{.6\textwidth}{if $\displaystyle \mu = \left(\sum_{j = 1}^{n} \max\{\kappa_j,0\} + \ell,0,\ldots,0,\sum_{j = 1}^{n} \min\{\kappa_j,0\} - \ell\right)$ for some $\ell \in \N_0$,}	\\
0 & otherwise.
\end{dcases*}
\end{multline*}
\end{lemma}

\begin{proof}
This is a direct consequence of \hyperref[lem:U_1^{n}branching]{Lemma \ref*{lem:U_1^{n}branching}} together with \hyperref[cor:Un-1triv]{Corollary \ref*{cor:Un-1triv}}.
\end{proof}

\begin{proof}[{Proofs of {\hyperref[thm:conductornewform]{Theorems \ref*{thm:conductornewform}}}, {\ref{thm:oldforms}}, {\ref{thm:epsilonconductor}}, and {\ref{thm:additiveconductor}} for $F = \C$}]
\hyperref[lem:HomUn]{Lemmata \ref*{lem:HomUn}} and \ref{lem:Cnewformoldformconductor} combine to complete the proofs of \hyperref[thm:conductornewform]{Theorems \ref*{thm:conductornewform}} and \ref{thm:oldforms} for $F = \C$, noting that for $\pi = \pi_1 \boxplus \cdots \boxplus \pi_n$ with $\pi_j = \chi^{\kappa_j} |\cdot|^{t_j}$, the newform $K$-type $\tau^{\circ} = \tau_{\mu^{\circ}}$ has highest weight
\[\mu^{\circ} = \left(\sum_{j = 1}^{n} \max\{\kappa_j,0\},0,\ldots,0,\sum_{j = 1}^{n} \min\{\kappa_j,0\}\right),\]
so that, recalling the definition \eqref{eq:HowedegreeC} of $\deg \tau^{\circ}$,
\[c(\pi) \coloneqq \deg \tau^{\circ} = \left\|\sum_{j = 1}^{n} \max\{\kappa_j,0\}\right\| + \left\|\sum_{j = 1}^{n} \min\{\kappa_j,0\}\right\| = \sum_{j = 1}^{n} \|\kappa_j\|.\]
\hyperref[thm:epsilonconductor]{Theorem \ref*{thm:epsilonconductor}} then holds for $F = \C$ via the fact that
\[\e(s,\pi,\psi) = \prod_{j = 1}^{n} \e(s,\pi_j,\psi) = \prod_{j = 1}^{n} i^{-\|\kappa_j\|} = i^{-\|\kappa_1\| - \cdots - \|\kappa_n\|} = i^{-c(\pi)},\]
recalling \eqref{eq:epsilonfactor}, while the case $n = 1$ implies \hyperref[thm:additiveconductor]{Theorem \ref*{thm:additiveconductor}} for $F = \C$.
\end{proof}

\begin{remark}
\label{rem:Unparity}
Note that $\deg \tau_{\mu(\ell)} = \deg \tau^{\circ} + 2\ell$ for
\[\mu(\ell) \coloneqq \left(\sum_{j = 1}^{n} \max\{\kappa_j,0\} + \ell,0,\ldots,0,\sum_{j = 1}^{n} \min\{\kappa_j,0\} - \ell\right),\]
and in particular that $\deg \tau_{\mu(\ell)} \equiv \deg \tau^{\circ} \pmod{2}$. This congruence holds not just for a $K$-type $\tau$ of $\pi$ for which $\Hom_{\Ugp(n - 1)}(1,\tau|_{\Ugp(n - 1)})$ is nontrivial, but for any $K$-type of $\pi$; see \cite[Theorem 2.1]{Fan18}.
\end{remark}

\begin{remark}
The minimal $K$-type of $\pi$ has highest weight
\[\mu = (\kappa_{\sigma(1)},\ldots,\kappa_{\sigma(n)}),\]
where $\sigma$ is a permutation for which $\kappa_{\sigma(1)} \geq \cdots \geq \kappa_{\sigma(n)}$. The corresponding Vogan norm is
\[\|\tau_{\mu}\|_{\V}^2 = \sum_{j = 1}^{n} \left(\kappa_{\sigma(j)} + n + 1 - 2j\right)^2.\]
This minimal $K$-type is the newform $K$-type $\tau^{\circ}$ if and only if $\kappa_{\sigma(2)} = \cdots = \kappa_{\sigma(n - 1)} = 0$.

In general, the Vogan norm of the minimal $K$-type is not equal to $c(\pi)$. On the other hand, the same cannot be said for the Howe degree: the Howe degree of the minimal $K$-type is
\[\sum_{j = 1}^{n} \left\|\kappa_{\sigma(j)}\right\| = \deg \tau^{\circ} = c(\pi).\]
\end{remark}

\subsection{Branching from \texorpdfstring{$\GL_n(\R)$}{GL\9040\231(R)} to \texorpdfstring{$\Ogp(n)$}{O(n)}}

Let $\pi = \pi_1 \boxplus \cdots \boxplus \pi_r$ be an induced representation of Whittaker type of $\GL_n(\R)$. This is induced from a standard parabolic subgroup $\Pgp(\R) = \Pgp_{(n_1,\ldots,n_r)}(\R)$ of $\GL_n(\R)$; when $n_j = 1$, $\pi_j$ is of the form $\chi^{\kappa_j} |\cdot|^{t_j}$ for some $\kappa_j \in \{0,1\}$ and $t_j \in \C$, while when $n_j = 2$, $\pi_j$ is of the form $D_{\kappa_j} \otimes \left|\det\right|^{t_j}$ for some $\kappa_j \geq 2$ and $t_j \in \C$.

\begin{lemma}
\label{lem:HomOn}
For $\tau \in \widehat{\Ogp(n)}$, the vector space $\Hom_{\Ogp(n)} \left(\tau, \pi|_{\Ogp(n)}\right)$ is isomorphic to
\[\bigoplus_{\substack{j = 1 \\ n_j = 2}}^{r} \bigoplus_{\substack{\ell_j = \kappa_j \\ \ell_j \equiv \kappa_j \hspace{-.25cm} \pmod{2}}}^{\infty} \Hom_{\Ogp(n_1) \times \cdots \times \Ogp(n_r)} \left(\tau|_{\Ogp(n_1) \times \cdots \times \Ogp(n_r)}, \bigboxtimes_{\substack{j = 1 \\ n_j = 1}}^{r} \tau_{\kappa_j} \boxtimes \bigboxtimes_{\substack{j = 1 \\ n_j = 2}}^{r} \tau_{(\ell_j,0)}\right).\]
\end{lemma}

Here we view $\Ogp(n_1) \times \cdots \times \Ogp(n_r)$ as a subgroup of block-diagonal matrices in $\Ogp(n)$; it is the maximal compact subgroup of the Levi subgroup $\Mgp_{\Pgp}(\R) \cong \GL_{n_1}(\R) \times \cdots \times \GL_{n_r}(\R)$ of the parabolic subgroup $\Pgp(\R) = \Pgp_{(n_1,\ldots,n_r)}(\R)$.

\begin{proof}
Mackey's restriction-induction formula implies that
\[\pi|_{\Ogp(n)} \cong \Ind_{\Ogp(n_1) \times \cdots \times \Ogp(n_r)}^{\Ogp(n)} \pi|_{\Ogp(n_1) \times \cdots \times \Ogp(n_r)},\]
and $\Hom_{\Ogp(n)} \left(\tau, \pi|_{\Ogp(n)}\right)$ is isomorphic to
\[\Hom_{\Ogp(n_1) \times \cdots \times \Ogp(n_r)}\left(\tau|_{\Ogp(n_1) \times \cdots \times \Ogp(n_r)}, \pi|_{\Ogp(n_1) \times \cdots \times \Ogp(n_r)}\right)\]
by the Frobenius reciprocity theorem. Note that if $n_j = 1$, so that $\pi_j = \chi^{\kappa_j} |\cdot|^{t_j}$, then $\pi_j |_{\Ogp(1)} \cong \tau_{\kappa_j}$, while if $n_j = 2$, so that $\pi_j = D_{\kappa_j} \otimes \left|\det\right|^{t_j}$, then
\[\pi_j |_{\Ogp(2)} \cong \bigoplus_{\substack{\ell_j = \kappa_j \\ \ell_j \equiv \kappa_j \hspace{-.25cm} \pmod{2}}}^{\infty} \tau_{(\ell_j,0)},\]
from which it follows that
\[\pi|_{\Ogp(n_1) \times \cdots \times \Ogp(n_r)} \cong \bigoplus_{\substack{j = 1 \\ n_j = 2}}^{r} \bigoplus_{\substack{\ell_j = \kappa_j \\ \ell_j \equiv \kappa_j \hspace{-.25cm} \pmod{2}}}^{\infty} \bigboxtimes_{\substack{j = 1 \\ n_j = 1}}^{r} \tau_{\kappa_j} \boxtimes \bigboxtimes_{\substack{j = 1 \\ n_j = 2}}^{r} \tau_{(\ell_j,0)}.\qedhere\]
\end{proof}

By restricting in stages and using the fact that $\#\{j : n_j = 1\} = 2r - n$ and $\#\{j : n_j = 2\} = n - r$, we deduce the following.

\begin{corollary}
\label{cor:dimHomOn}
For $\tau \in \widehat{\Ogp(n)}$, $\dim \Hom_{\Ogp(n)} \left(\tau, \pi|_{\Ogp(n)}\right)$ is equal to
\begin{multline}
\label{eq:HomOn}
\sum_{\substack{j = 1 \\ n_j = 2}}^{r} \sum_{\substack{\ell_j = \kappa_j \\ \ell_j \equiv \kappa_j \hspace{-.25cm} \pmod{2}}}^{\infty} \sum_{\nu_{2(n - r)} \in \Lambda_{2(n - r)}} \dim \Hom_{\Ogp(2)^{n - r}} \left(\tau_{\nu_{2(n - r)}} |_{\Ogp(2)^{n - r}}, \bigboxtimes_{\substack{j = 1 \\ n_j = 2}}^{r} \tau_{(\ell_j,0)} \right)	\\
\times \dim \Hom_{\Ogp(2(n - r)) \times \Ogp(1)^{2r - n}} \left(\tau |_{\Ogp(2(n - r)) \times \Ogp(1)^{2r - n}}, \tau_{\nu_{2(n - r)}} \boxtimes \bigboxtimes_{\substack{j = 1 \\ n_j = 1}}^{r} \tau_{\kappa_j}\right).
\end{multline}
\end{corollary}

To understand \eqref{eq:HomOn}, which involves branching from $\Ogp(n)$ to various subgroups, we make use of the following branching rule from $\Ogp(n)$ to $\Ogp(n - 1) \times \Ogp(1)$.

\begin{lemma}[{\cite[Proposition 10.1]{Pro94}}]
\label{lem:O_{n - 1}xO_1branching}
For $\tau_{\mu} \in \widehat{\Ogp(n)}$ of highest weight $\mu = (\mu_1,\ldots,\mu_n) \in \Lambda_n$,
\[\tau_{\mu}|_{\Ogp(n - 1) \times \Ogp(1)} \cong \bigoplus_{\lambda \in \Lambda_1} \bigoplus_{\substack{\nu = (\nu_1,\ldots,\nu_{n - 1}) \in \Lambda_{n - 1} \\ \mu_1 \geq \nu_1 \geq \mu_2 \geq \cdots \geq \nu_{n - 1} \geq \mu_n \\ \sum_{j = 1}^{n - 1} \nu_j \equiv \sum_{j = 1}^{n} \mu_j - \lambda \hspace{-.25cm} \pmod{2}}} \tau_{\nu} \boxtimes \tau_{\lambda}.\]
In particular,
\[\tau_{\mu}|_{\Ogp(n - 1)} \cong \bigoplus_{\substack{\nu = (\nu_1,\ldots,\nu_{n - 1}) \in \Lambda_{n - 1} \\ \mu_1 \geq \nu_1 \geq \mu_2 \geq \cdots \geq \nu_{n - 1} \geq \mu_n}} \tau_{\nu}.\]
\end{lemma}

\begin{corollary}
\label{cor:On-1triv}
The restriction to $\Ogp(n - 1)$ of the irreducible representation $\tau_{\mu} \in \widehat{\Ogp(n)}$ of highest weight $\mu = (\mu_1, \ldots, \mu_n) \in \Lambda_n$ contains the trivial representation if and only if $\mu_1 \geq 0$ and $\mu_2 = \cdots = \mu_n = 0$, in which case the trivial representation occurs with multiplicity one.
\end{corollary}

Now we iterate the branching rule in \hyperref[lem:O_{n - 1}xO_1branching]{Lemma \ref*{lem:O_{n - 1}xO_1branching}}.

\begin{lemma}
\label{lem:O_{n-n'}xO_1^{2r - n}branching}
For $\tau_{\mu} \in \widehat{\Ogp(n)}$ of highest weight $\mu = (\mu_1,0,\ldots,0) \in \Lambda_n$, and for $\lambda_1,\ldots,\lambda_{2r - n} \in \Lambda_1$ and $\nu_{2(n - r)} \in \Lambda_{2(n - r)}$,
\[\dim \Hom_{\Ogp(2(n - r)) \times \Ogp(1)^{2r - n}}\left(\tau_{\mu}|_{\Ogp(2(n - r)) \times \Ogp(1)^{2r - n}}, \tau_{\nu_{2(n - r)}} \boxtimes \bigboxtimes_{j = 1}^{2r - n} \tau_{\lambda_j}\right)\]
is equal to
\[\begin{dcases*}
1 & if $\nu_{2(n - r)} = (\mu_1 - \lambda_1 - 2\ell,0,\ldots,0)$ for some $\ell \in \N_0$,	\\
0 & otherwise
\end{dcases*}\]
if $\#\{j : n_j = 1\} = 1$, while if $\#\{j : n_j = 1\} \geq 2$, this is equal to
\[\begin{dcases*}
\binom{\ell + n' - 1}{n' - 1} & if $\displaystyle \nu_{2(n - r)} = \left(\mu_1 - \sum_{j = 1}^{2r - n} \lambda_j - 2\ell,0,\ldots,0\right)$ for some $\ell \in \N_0$,	\\
0 & otherwise.
\end{dcases*}\]
\end{lemma}

\begin{proof}
We take $\mu = (\mu_1,0,\ldots,0)$ in \hyperref[lem:O_{n - 1}xO_1branching]{Lemma \ref*{lem:O_{n - 1}xO_1branching}}; the case $\#\{j : n_j = 1\} = 1$ is then immediate, while if $\#\{j : n_j = 1\} \geq 2$, we iterate this branching rule in order to find that
\[\tau_{\mu}|_{\Ogp(2(n - r)) \times \Ogp(1)^{2r - n}} \cong \bigoplus_{j = 1}^{2r - n} \bigoplus_{\lambda_j \in \Lambda_1} \bigoplus_{\substack{(\nu_{j,1},0,\ldots,0) \in \Lambda_{n - j} \\ \nu_{j,1} \leq \nu_{j - 1,1} \\ \nu_{j,1} \equiv \nu_{j - 1,1} - \lambda_j \hspace{-.25cm} \pmod{2}}} \tau_{(\nu_{n',1},0,\ldots,0)} \boxtimes \bigboxtimes_{j = 1}^{2r - n} \tau_{\lambda_j},\]
where we set $\nu_{0,1} \coloneqq \mu_1$. It follows that for fixed $\mu = (\mu_1,0,\ldots,0) \in \Lambda_n$, $\lambda_1,\ldots,\lambda_{2r - n} \in \Lambda_1$, and $\nu_{2(n - r)} \in \Lambda_{2(n - r)}$, the multiplicity of $\tau_{\nu_{2(n - r)}} \boxtimes \tau_{\lambda_1} \boxtimes \cdots \boxtimes \tau_{\lambda_{2r - n}}$ in $\tau_{\mu}$ is zero unless $\nu_{2(n - r)}$ is of the form $(\nu_{n',1},0,\ldots,0)$ for some $\nu_{n',1} \in \N_0$, in which case it is equal to
\[\# \left\{(\nu_{1,1},\ldots,\nu_{n' - 1,1}) \in \N_0^{n' - 1} : \nu_{j - 1,1} \geq \nu_{j,1}, \ \nu_{j,1} \equiv \nu_{j - 1,1} - \lambda_j \hspace{-.25cm} \pmod{2} \text{ for all $j \in \{1, \ldots, n'\}$}\right\},\]
namely the number of $(n' - 1)$-tuples $(\nu_{1,1},\ldots,\nu_{n' - 1,1})$ for which the system of inequalities
\[\mu_1 \geq \nu_{1,1} + \lambda_1 \geq \cdots \geq \nu_{n',1} + \sum_{j = 1}^{2r - n} \lambda_j\]
holds with each quantity being of the same parity. This is zero unless
\[\mu_1 = \nu_{n',1} + \sum_{j = 1}^{2r - n} \lambda_j + 2\ell\]
for some $\ell \in \N_0$, in which case the multiplicity is precisely the number of ordered $(n' - 1)$-tuples taking values between $0$ and $\ell$, which is
\[\binom{\ell + n' - 2}{n' - 2}.\qedhere\]
\end{proof}

We also require a special case of the branching rule from $\Ogp(n)$ to $\Ogp(n - 2) \times \Ogp(2)$.

\begin{lemma}[{\cite[Proposition 10.3]{Pro94}}]
\label{lem:O_{n-2}xO_2branching}
For $n \geq 3$ and $\tau_{\mu} \in \widehat{\Ogp(n)}$ of highest weight $\mu = (\mu_1,0,\ldots,0) \in \Lambda_n$,
\[\tau_{\mu}|_{\Ogp(n - 2) \times \Ogp(2)} \cong \bigoplus_{\substack{\lambda = (\lambda_1,0) \in \Lambda_2 \\ \lambda_1 \leq \mu_1}} \bigoplus_{\substack{\nu = (\nu_1,0,\ldots,0) \in \Lambda_{n - 2} \\ \nu_1 \leq \mu_1 - \lambda_1 \\ \nu_1 \equiv \mu_1 - \lambda_1 \hspace{-.25cm} \pmod{2}}} \tau_{\nu} \boxtimes \tau_{\lambda}.\]
\end{lemma}

We again iterate this branching rule.

\begin{lemma}
\label{lem:O_2^{n - r}branching}
For $\tau_{\mu} \in \widehat{\Ogp(2(n - r))}$ of highest weight $\mu = (\mu_1,0,\ldots,0) \in \Lambda_{2(n - r)}$ and for $(\lambda_{j,1},\lambda_{j,2}) \in \Lambda_2$ with $j \in \{1,\ldots,n - r\}$,
\begin{multline*}
\dim \Hom_{\Ogp(2)^{n - r}}\left(\tau_{\mu}|_{\Ogp(2)^{n - r}}, \bigboxtimes_{j = 1}^{n - r} \tau_{(\lambda_{j,1},\lambda_{j,2})}\right)	\\
= \begin{dcases*}
1 & if $\#\{j : n_j = 2\} = 1$ and $(\lambda_{1,1},\lambda_{1,2}) = (\mu_1,0)$,	\\
\binom{\ell + n - r - 2}{n - r - 2} & \parbox{.6\textwidth}{if $\#\{j : n_j = 2\} \geq 2$, $\displaystyle \mu_1 = \sum_{j = 1}^{n - r} \lambda_{j,1} + 2\ell$ for some $\ell \in \N_0$, and $\lambda_{j,2} = 0$ for all $j \in \{1, \ldots, n - r\}$,}	\\
0 & otherwise.
\end{dcases*}
\end{multline*}
\end{lemma}

\begin{proof}
The result follows from Schur's lemma if $\#\{j : n_j = 2\} = 1$. If $\#\{j : n_j = 2\} \geq 2$, iterating \hyperref[lem:O_{n-2}xO_2branching]{Lemma \ref*{lem:O_{n-2}xO_2branching}} yields
\[\tau_{\mu}|_{\Ogp(2)^{n - r}} \cong \bigoplus_{j = 1}^{n - r - 1} \bigoplus_{\substack{(\lambda_{j,1},0) \in \Lambda_2 \\ \lambda_{j,1} \leq \nu_{j - 1,1}}} \bigoplus_{\substack{(\nu_{j,1},0,\ldots,0) \in \Lambda_{2(n - r - j)} \\ \nu_{j,1} \leq \nu_{j - 1,1} - \lambda_{j,1} \\ \nu_{j,1} \equiv \nu_{j - 1,1} - \lambda_{j,1} \hspace{-.25cm} \pmod{2}}} \bigboxtimes_{j = 1}^{n - r} \tau_{(\lambda_{j,1},\lambda_{j,2})},\]
where we again set $\nu_{0,1} \coloneqq \mu_1$. So the multiplicity of $\tau_{(\lambda_{1,1},\lambda_{1,2})} \boxtimes \cdots \boxtimes \tau_{(\lambda_{n - r,1},\lambda_{n - r,2})}$ in $\tau_{\mu}$ is equal to the number of $(n - r - 1)$-tuples $(\nu_{1,1}, \ldots, \nu_{n - r - 1})$ for which the system of inequalities
\[\mu_1 \geq \nu_{1,1} + \lambda_{1,1} \geq \nu_{2,1} + \lambda_{1,1} + \lambda_{2,1} \geq \cdots \geq \nu_{n - r - 1,1} + \sum_{j = 1}^{n - r - 1} \lambda_{j,1} \geq \sum_{j = 1}^{n - r} \lambda_{j,1}\]
holds with each quantity being of the same parity. This is zero unless
\[\mu_1 = \sum_{j = 1}^{n - r} \lambda_{j,1} + 2\ell\]
for some $\ell \in \N_0$, in which case the multiplicity is precisely the number of ordered $(n - r - 2)$-tuples taking values between $0$ and $\ell$, which is
\[\binom{\ell + n - r - 2}{n - r - 2}.\qedhere\]
\end{proof}

Shortly we shall require the following combinatorial identity involving binomial coefficients.

\begin{lemma}
\label{lem:binom}
For $k,m,n \in \N_0$, we have that
\[\sum_{j = 0}^{k} \binom{j + m}{m} \binom{k - j + n}{n} = \binom{k + m + n + 1}{m + n + 1}.\]
\end{lemma}

\begin{proof}
Such an identity begs for a proof via a generating series:
\[\begin{split}
\sum_{k = 0}^{\infty} \binom{k + m + n + 1}{m + n + 1} x^k & = \frac{1}{(1 - x)^{m + n + 2}}	\\
& = \frac{1}{(1 - x)^{m + 1}} \frac{1}{(1 - x)^{n + 1}}	\\
& = \sum_{k_1 = 0}^{\infty} \binom{k_1 + m}{m} x^{k_1} \sum_{k_2 = 0}^{\infty} \binom{k_2 + n}{n} x^{k_2}	\\
& = \sum_{k = 0}^{\infty} \sum_{j = 0}^{k} \binom{j + m}{m} \binom{k - j + n}{n} x^k. \qedhere
\end{split}\]
\end{proof}

With these results in hand, we can now explicitly determine \eqref{eq:HomOn}.

\begin{lemma}
\label{lem:Rnewformoldformconductor}
Suppose that the restriction of $\tau_{\mu} \in \widehat{\Ogp(n)}$ to $\Ogp(n - 1)$ contains the trivial representation. Then the highest weight of $\tau_{\mu}$ is of the form $\mu = (\mu_1,0,\ldots,0)$, the trivial representation occurs with multiplicity one, and
\begin{multline}
\label{eq:Rnewformoldformconductor}
\sum_{\substack{j = 1 \\ n_j = 2}}^{r} \sum_{\substack{\ell_j = \kappa_j \\ \ell_j \equiv \kappa_j \hspace{-.25cm} \pmod{2}}}^{\infty} \sum_{\nu_{2(n - r)} \in \Lambda_{2(n - r)}} \dim \Hom_{\Ogp(2)^{n - r}} \left(\tau_{\nu_{2(n - r)}} |_{\Ogp(2)^{n - r}}, \bigboxtimes_{\substack{j = 1 \\ n_j = 2}}^{r} \tau_{(\ell_j,0)} \right)	\\
\times \dim \Hom_{\Ogp(2(n - r)) \times \Ogp(1)^{2r - n}} \left(\tau_{\mu} |_{\Ogp(2(n - r)) \times \Ogp(1)^{2r - n}}, \tau_{\nu_{2(n - r)}} \boxtimes \bigboxtimes_{\substack{j = 1 \\ n_j = 1}}^{r} \tau_{\kappa_j}\right)	\\
= \begin{dcases*}
\binom{\ell + n - 2}{n - 2} & if $\displaystyle \mu = \left(\sum_{j = 1}^{r} \kappa_j + 2\ell,0,\ldots,0\right)$ for some $\ell \in \N_0$,	\\
0 & otherwise.
\end{dcases*}
\end{multline}
\end{lemma}

\begin{proof}
The first claim two claims are \hyperref[cor:On-1triv]{Corollary \ref*{cor:On-1triv}}. To prove the identity \eqref{eq:Rnewformoldformconductor}, we first note that if $\#\{j : n_j = 2\} = 0$, \hyperref[lem:O_{n-n'}xO_1^{2r - n}branching]{Lemma \ref*{lem:O_{n-n'}xO_1^{2r - n}branching}} implies the result upon replacing $2r - n$ with $2r - n - 1$. If $\#\{j : n_j = 2\} \geq 1$, we combine \hyperref[lem:O_{n-n'}xO_1^{2r - n}branching]{Lemmata \ref*{lem:O_{n-n'}xO_1^{2r - n}branching}} and \ref{lem:O_2^{n - r}branching} to see that the inner sum over $\nu_{2(n - r)} \in \Lambda_{2(n - r)}$ in the left-hand side of \eqref{eq:Rnewformoldformconductor} is equal to zero unless
\[\mu_1 = \sum_{\substack{j = 1 \\ n_j = 1}}^{r} \kappa_j + \sum_{\substack{j = 1 \\ n_j = 2}}^{r} \ell_j + 2\ell\]
for some $\ell \in \N_0$, so that
\[\nu_{2(n - r)} = \left(\sum_{\substack{j = 1 \\ n_j = 2}}^{r} \ell_j + 2(\ell - \ell'),0,\ldots,0\right)\]
for some $\ell' \in \{0, \ldots, \ell\}$, in which case this inner sum is equal to
\[\sum_{\ell' = 0}^{\ell} \binom{\ell' + n' - 2}{n' - 2} \binom{\ell - \ell' + n - r - 2}{n - r - 2}.\]
Coupled with \hyperref[cor:dimHomOn]{Corollary \ref*{cor:dimHomOn}}, we find that the left-hand side of \eqref{eq:Rnewformoldformconductor} is equal to zero unless
\[\mu_1 = \sum_{j = 1}^{r} \kappa_j + 2\ell\]
for some $\ell \in \N_0$, in which case it is equal to
\[\sum_{i = 1}^{n - r} \sum_{\alpha_i = 0}^{\infty} \sum_{\ell' = 0}^{\ell - \sum_{i = 1}^{n - r} \alpha_i} \binom{\ell' + 2r - n - 1}{2r - n - 1} \binom{\ell - \ell' - \sum_{i = 1}^{n - r} \alpha_i + n - r - 2}{n - r - 2}\]
upon defining $\alpha_i$ such that there is an equality of the sets $\{\alpha_i\}$ and $\{(\ell_j - \kappa_j)/2 : n_j = 2\}$. Interchanging the order of summation, this becomes
\begin{equation}
\label{eq:triplesum}
\sum_{\ell' = 0}^{\ell} \binom{\ell' + 2r - n - 1}{2r - n - 1} \sum_{i = 1}^{n - r} \sum_{\alpha_i = 0}^{\ell - \ell' - \sum_{m = 1}^{i - 1} \alpha_m} \binom{\ell - \ell' - \sum_{i = 1}^{n - r} \alpha_i + n - r - 2}{n - r - 2}.
\end{equation}
Using the fact that $\binom{n - 1}{k - 1} = \binom{n}{k} - \binom{n - 1}{k}$, the sum over $\alpha_{n - r}$ telescopes to
\[\binom{\ell - \ell' - \sum_{i = 1}^{n - r - 1} \alpha_i + n - r - 1}{n - r - 1},\]
and so upon iterating this process, we find that the inner double sum in \eqref{eq:triplesum} is equal to
\[\binom{\ell - \ell' + 2n - r - 2}{2n - r - 2},\]
at which point \hyperref[lem:binom]{Lemma \ref*{lem:binom}} completes the proof.
\end{proof}

\begin{proof}[{Proofs of {\hyperref[thm:conductornewform]{Theorems \ref*{thm:conductornewform}}}, {\ref{thm:oldforms}}, {\ref{thm:epsilonconductor}}, {\ref{thm:additiveconductor}}, and {\ref{thm:inductiveconductor}} for $F = \R$}]
\hyperref[lem:HomOn]{Lemmata \ref*{lem:HomOn}} and \ref{lem:Rnewformoldformconductor} combine to complete the proofs of \hyperref[thm:conductornewform]{Theorems \ref*{thm:conductornewform}} and \ref{thm:oldforms} for $F = \R$: for $\pi = \pi_1 \boxplus \cdots \boxplus \pi_r$ with $\pi_j = \chi^{\kappa_j} |\cdot|^{t_j}$ when $n_j = 1$ and $\pi_j = D_{\kappa_j} \otimes \left|\det\right|^{t_j}$ when $n_j = 2$, the newform $K$-type $\tau^{\circ} = \tau_{\mu^{\circ}}$ has highest weight
\[\mu^{\circ} = \left(\sum_{j = 1}^{r} \kappa_j,0,\ldots,0\right),\]
so that, recalling the definition \eqref{eq:HowedegreeR} of $\deg \tau^{\circ}$,
\[c(\pi) \coloneqq \deg \tau^{\circ} = \sum_{j = 1}^{r} \kappa_j.\]
\hyperref[thm:epsilonconductor]{Theorem \ref*{thm:epsilonconductor}} then holds for $F = \R$ via the fact that
\[\e(s,\pi,\psi) = \prod_{j = 1}^{r} \e(s,\pi_j,\psi) = \prod_{j = 1}^{r} i^{-\kappa_j} = i^{-\kappa_1 - \cdots - \kappa_r} = i^{-c(\pi)},\]
while the cases $n = \#\{j : n_j = 1\} = 1$, so that $\pi = \chi^{\kappa} |\cdot|^t$, and $n = 2\#\{j : n_j = 2\} = 2$, so that $\pi = D_{\kappa} \otimes \left|\det\right|^t$, imply \hyperref[thm:additiveconductor]{Theorem \ref*{thm:additiveconductor}} for $F = \R$. Finally, for an induced representation of Whittaker type of $\GL_n(\C)$ of the form $\pi = \chi^{\kappa_1} |\cdot|_{\C}^{t_1} \boxplus \cdots \boxplus \chi^{\kappa_n} |\cdot|_{\C}^{t_n}$, the induced representation $\AI_{\C/\R} \pi$ is isomorphic to the isobaric sum
\[\bigboxplus_{j = 1}^{n} \Ind_{\GL_1(\C)}^{\GL_2(\R)} \chi^{\kappa_j} |\cdot|_{\C}^{t_j},\]
and
\[\Ind_{\GL_1(\C)}^{\GL_2(\R)} \chi^{\kappa} |\cdot|_{\C}^t \cong \begin{dcases*}
|\cdot|_{\R}^t \boxplus \chi |\cdot|_{\R}^t & if $\kappa = 0$,	\\
D_{\|\kappa\| + 1} \otimes \left|\det\right|_{\R}^t & if $\kappa \neq 0$,
\end{dcases*}\]
so that
\[c\left(\AI_{\C/\R} \pi\right) = \sum_{j = 1}^{n} \left(\|\kappa_j\| + 1\right) = c(\pi) + n,\]
thereby proving \hyperref[thm:inductiveconductor]{Theorem \ref*{thm:inductiveconductor}}.
\end{proof}

\begin{remark}
Just as was observed in \hyperref[rem:Unparity]{Remark \ref*{rem:Unparity}} for $F = \C$, the Howe degree of a $K$-type of $\pi$ is always congruent to $\deg \tau^{\circ}$ modulo $2$.
\end{remark}

\begin{remark}
From \cite[Proposition 4.3]{Lin18}, the minimal $K$-type of $\pi = \pi_1 \boxplus \cdots \boxplus \pi_r$ has highest weight
\[\mu = \left(\kappa_{\sigma(1)},\ldots,\kappa_{\sigma(r)},\underbrace{0,\ldots,0}_{n - r \text{ times}}\right),\]
where $\sigma$ is a permutation for which $\kappa_{\sigma(1)} \geq \cdots \geq \kappa_{\sigma(r)}$. This minimal $K$-type is $\tau^{\circ}$ if and only if $r = n - 1$ and $\kappa_j = 0$ whenever $n_j = 1$ or $r = n$ and $\kappa_j \neq 0$ for at most one $j$.

Once again, the Howe degree of the minimal $K$-type is
\[\sum_{j = 1}^{n} \kappa_{\sigma(j)} = \deg \tau^{\circ} = c(\pi).\]
\end{remark}

\section{Homogeneous Harmonic Polynomials}
\label{sect:hhp}

Having identified the newform $K$-type, we now study a particular model, a space of homogeneous harmonic polynomials, of this representation of $K$. This allows us to explicitly describe the matrix coefficients of this representation, which are used to construct the explicit Schwartz functions $\Phi^{\circ} \in \Ss(\Mat_{1 \times n}(F))$ and $\Phi \in \Ss(\Mat_{n \times n}(F))$ given in \eqref{eq:Phicirc} and \eqref{eq:Phi}.

\subsection{Homogeneous Harmonic Polynomials and Representations of \texorpdfstring{$\Ugp(n)$}{U(n)}}
\label{sect:harmonicU(n)}

For nonnegative integers $p,q$, let $\HH_{p,q}(\C^n)$ denote the vector space consisting of harmonic polynomials that are homogeneous of bidegree $(p,q)$, namely the set of polynomials $P(z) = P(z_1,\ldots,z_n,\overline{z_1},\ldots,\overline{z_n})$ in $z \in \Mat_{1 \times n}(\C) = \C^n$ that are annihilated by the Laplacian
\[\Delta = 4 \sum_{j = 1}^{n} \frac{\dee^2}{\dee z_j \dee \overline{z_j}}\]
and satisfy
\[P(\lambda z) = P\left(\lambda z_1, \ldots, \lambda z_n,\overline{\lambda z_1},\ldots,\overline{\lambda z_n}\right) = \lambda^p \overline{\lambda}^q P\left(z_1,\ldots,z_n,\overline{z_1},\ldots,\overline{z_n}\right)\]
for all $\lambda \in \C$. The dimension of $\HH_{p,q}(\C^n)$ is $1$ for $n = 1$ and
\[\frac{(p + q + n - 1) (p + n - 2)! (q + n - 2)!}{p! q! (n - 2)! (n - 1)!} = \frac{p + q + n - 1}{n - 1} \binom{p + n - 2}{n - 2} \binom{q + n - 2}{n - 2}\]
for $n \geq 2$.

Let $\tau$ be an irreducible representation of $\Ugp(n)$ of highest weight $\mu = (p,0,\ldots,0,-q)$; note that for $n = 1$, either $p$ or $q$ must be zero. Then $\HH_{p,q}(\C^n)$ is a model of $\tau$, where the group $\Ugp(n) \ni k$ acts on $\HH_{p,q}(\C^n) \ni P$ via right translation, namely
\[(\tau(k) \cdot P)\left(z_1,\ldots,z_n,\overline{z_1},\ldots,\overline{z_n}\right) \coloneqq P\left(\begin{pmatrix} z_1 & \cdots & z_n \end{pmatrix} k, \begin{pmatrix} \overline{z_1} & \cdots & \overline{z_n} \end{pmatrix} \overline{k}\right);\]
that is, $(\tau(k) \cdot P)(z) \coloneqq P(zk)$. We define a $\Ugp(n)$-invariant inner product on $\HH_{p,q}(\C^n) \ni P,Q$ by
\[\langle P,Q \rangle \coloneqq \int_{\Ugp(n)} P(e_n k) \overline{Q(e_n k)} \, dk.\]

From the branching rule in \hyperref[lem:U_{n - 1}xU_1branching]{Lemma \ref*{lem:U_{n - 1}xU_1branching}}, there is a one-dimensional subspace of $\HH_{p,q}(\C^n)$ that is invariant under the action of $\Ugp(n - 1) \ni k'$ embedded in $\Ugp(n)$ via $k' \mapsto \begin{psmallmatrix} k' & 0 \\ 0 & 1 \end{psmallmatrix}$, which we can describe explicitly.

\begin{lemma}[{\cite[Proposition 12.2.6]{Rud08}}]
There exists a unique homogeneous harmonic polynomial $P^{\circ} \in \HH_{p,q}(\C^n)$ satisfying $P^{\circ}(e_n) = 1$ and $\tau\begin{psmallmatrix} k' & 0 \\ 0 & 1 \end{psmallmatrix} \cdot P^{\circ} = P^{\circ}$ for all $k' \in \Ugp(n - 1)$, namely
\begin{equation}
\label{eq:PcircUn}
P^{\circ}(z) \coloneqq \sum_{\nu = 0}^{\min\{p,q\}} \frac{(-1)^{\nu} \binom{p}{\nu} \binom{q}{\nu}}{\binom{\nu + n - 2}{n - 2}} \left(z_1 \overline{z_1} + \cdots + z_{n - 1} \overline{z_{n - 1}}\right)^{\nu} z_n^{p - \nu} \overline{z_n}^{q - \nu}.
\end{equation}
In particular, for $n = 1$, so that either $p$ or $q$ is equal to $0$,
\begin{equation}
\label{eq:PjcircC}
P^{\circ}(z_1) = \begin{dcases*}
z_1^p & for $p \in \N_0$ and $q = 0$,	\\
\overline{z_1}^q & for $p = 0$ and $q \in \N_0$.
\end{dcases*}
\end{equation}
\end{lemma}

\begin{remark}
When restricted to the unit sphere in $\C^n$, this polynomial is sometimes referred to as the zonal spherical harmonic of bidegree $(p,q)$ in dimension $n$. When $z_1 \overline{z_1} + \cdots + z_{n - 1} \overline{z_{n - 1}} = 1 - z_n \overline{z_n}$, $P^{\circ}(z_1,\ldots,z_n,\overline{z_1},\ldots,\overline{z_n})$ is a polynomial in $z_n,\overline{z_n}$ and can be expressed in terms of the generalised Zernike polynomial $P_{q,p}^{n - 2}$ (also called a generalised disc polynomial) or the Jacobi polynomial $P_q^{(n - 2,p - q)}$.
\end{remark}

We make crucial use of the fact that for all $P \in \HH_{p,q}(\C^n)$ and $k \in \Ugp(n)$, $P(e_n k)$ is equal to a matrix coefficient of $\tau$. This can be thought of as an explicit form of Schur orthogonality.

\begin{lemma}[{\cite[Theorem 12.2.5]{Rud08}}]
\label{lem:QenkC}
The reproducing kernel for $\HH_{p,q}(\C^n)$ is $(\dim \tau) P^{\circ}$, so that for all $P \in \HH_{p,q}(\C^n)$ and $k \in \Ugp(n)$,
\[P(e_n k) = (\dim \tau) \left\langle \tau(k) \cdot P, P^{\circ}\right\rangle.\]
In particular, $\langle P^{\circ}, P^{\circ} \rangle = (\dim \tau)^{-1}$. Moreover, for all $P \in \HH_{p,q}(\C^n)$ and $z \in \C^n$,
\[P(z) = \dim \tau \int_{\Ugp(n)} P(e_n k^{-1}) P^{\circ}(zk) \, dk.\]
\end{lemma}

Finally, we also require the following identity, which states that homogeneous harmonic polynomials $P \in \HH_{p,q}(\C^n)$ are eigenfunctions of the Fourier transform.

\begin{lemma}[{Hecke's Identity; cf.\ \cite[Chapter IV, Theorem 3.4]{SW71}}]
\label{lem:HeckeC}
For any homogeneous harmonic polynomial $P \in \HH_{p,q}(\C^n)$ and $w \in \C^n$, we have that
\[\int_{\C^n} P(z) \exp\left(-2\pi z \prescript{t}{}{\overline{z}}\right) \overline{\psi}\left(z \prescript{t}{}{\overline{w}}\right) \, dz = i^{-p - q} P(w) \exp\left(-2\pi w \prescript{t}{}{\overline{w}}\right).\]
\end{lemma}

\begin{proof}
First we prove this for $P(z) = z_1^p \overline{z_n}^q$, the highest weight vector of $\tau$. In this case,
\begin{align*}
i^{-p - q} P(w) \exp\left(-2\pi w \prescript{t}{}{\overline{w}}\right) & = (-2\pi i)^{-p - q} \frac{\dee^{p + q}}{\dee \overline{w_1}^p \dee w_n^q} \exp\left(-2\pi w \prescript{t}{}{\overline{w}}\right)	\\
& = (-2\pi i)^{-p - q} \frac{\dee^{p + q}}{\dee \overline{w_1}^p \dee w_n^q} \int_{\C^n} \exp\left(-2\pi z \prescript{t}{}{\overline{z}}\right) \overline{\psi}\left(z \prescript{t}{}{\overline{w}}\right) \, dz	\\
& = \int_{\C^n} P(z) \exp\left(-2\pi z \prescript{t}{}{\overline{z}}\right) \overline{\psi}\left(z \prescript{t}{}{\overline{w}}\right) \, dz.
\end{align*}
For any other $Q \in \HH_{p,q}(\C^n)$, we may write $Q(z)$ as a linear combination of elements of the form $P(zk)$ with $k \in \Ugp(n)$, and so using the above calculation with $w$ replaced by $wk$ and making the change of variables $z \mapsto zk$ yields the result upon recalling that $k \prescript{t}{}{\overline{k}} = 1_n$.
\end{proof}

\subsection{Homogeneous Harmonic Polynomials and Representations of \texorpdfstring{$\Ogp(n)$}{O(n)}}

Similarly, for a nonnegative integer $p$, let $\HH_p(\R^n)$ denote the vector space consisting of homogeneous harmonic polynomials of degree $p$, namely the set of polynomials $P(x) = P(x_1,\ldots,x_n)$ in $x \in \Mat_{1 \times n}(\R) = \R^n$ that are annihilated by the Laplacian
\[\Delta = \sum_{j = 1}^{n} \frac{\dee^2}{\dee x_j^2}\]
and satisfy $P(\lambda x) = \lambda^p P(x)$ for all $\lambda \in \R$. This space has dimension $1$ for $n = 1$ and $p \in \{0,1\}$ and has dimension
\[\binom{p + n - 2}{n - 2} + \binom{p + n - 3}{n - 3} = \frac{(2p + n - 2) (p + n - 3)!}{p!(n - 2)!} = \frac{2p + n - 2}{p + n - 2} \binom{p + n - 2}{n - 2}\]
for $n \geq 2$ and $p \in \N_0$.

Let $\tau$ be an irreducible representation of $\Ogp(n)$ of highest weight $\mu = (p,0,\ldots,0)$, where $p$ is a nonnegative integer; note that $p \in \{0,1\}$ for $n = 1$. Then $\HH_p(\R^n)$ is a model of $\tau$, where the group $\Ogp(n) \ni k$ acts on the space $\HH_p(\R^n) \ni P$ via right translation, namely $(\tau(k) \cdot P)(x) \coloneqq P(xk)$. We define an $\Ogp(n)$-invariant inner product on $\HH_p(\R^n) \ni P,Q$ by
\[\langle P,Q\rangle \coloneqq \int_{\Ogp(n)} P(e_n k) \overline{Q(e_n k)} \, dk.\]

We record the following results, all of which are analogous to those for $\HH_{p,q}(\C^n)$ in \hyperref[sect:harmonicU(n)]{Section \ref*{sect:harmonicU(n)}}.

\begin{lemma}[{\cite[Section 2.1.2]{AH12}}]
There exists a unique homogeneous harmonic polynomial $P^{\circ} \in \HH_p(\R^n)$ satisfying $P^{\circ}(e_n) = 1$ and $\tau\begin{psmallmatrix} k' & 0 \\ 0 & 1 \end{psmallmatrix} \cdot P^{\circ} = P^{\circ}$ for all $k' \in \Ogp(n - 1)$, namely
\begin{equation}
\label{eq:PcircOn}
P^{\circ}(x) \coloneqq \sum_{\substack{\nu = 0 \\ \nu \equiv 0 \hspace{-.25cm} \pmod{2}}}^{p} \frac{i^{\nu} p! \Gamma\left(\frac{n - 1}{2}\right)}{2^{\nu} \left(\frac{\nu}{2}\right)! (p - \nu)! \Gamma\left(\frac{\nu + n - 1}{2}\right)} \left(x_1^2 + \cdots + x_{n - 1}^2\right)^{\frac{\nu}{2}} x_n^{p - \nu}.
\end{equation}
In particular, for $n = 1$, so that $p \in \{0,1\}$,
\begin{equation}
\label{eq:PjcircR1}
P^{\circ}(x_1) = x_1^p,
\end{equation}
while for $n = 2$, so that $p \in \N_0$,
\begin{equation}
\label{eq:PjcircR2}
P^{\circ}(x_1,x_2) = \frac{1}{2} \left(x_2 - i x_1\right)^p + \frac{1}{2} \left(x_2 + i x_1\right)^p = \sum_{\substack{\nu = 0 \\ \nu \equiv 0 \hspace{-.25cm} \pmod{2}}}^{p} \binom{p}{\nu} (ix_1)^{\nu} x_2^{p - \nu}.
\end{equation}
\end{lemma}

\begin{remark}
Atkinson and Han name $P^{\circ}$ the Legendre polynomial of degree $p$ in $n$ dimensions \cite[Section 2.1.2]{AH12}; when restricted to the unit sphere in $\R^n$, this polynomial is also referred to as the zonal spherical harmonic. When $x_1^2 + \cdots + x_{n - 1}^2 = 1 - x_n^2$, $P^{\circ}(x_1,\ldots,x_n)$ is a polynomial in $x_n$ and can be expressed in terms of the Gegenbauer polynomial $C_p^{\frac{n - 2}{2}}$ (also called an ultraspherical polynomial) or the Jacobi polynomial $P_p^{\left(\frac{n - 3}{2},\frac{n - 3}{2}\right)}$; in particular, when $n = 2$, this is just the usual Legendre polynomial of degree $p$.
\end{remark}

\begin{lemma}[{\cite[Section 2.2]{AH12}}]
\label{lem:QenkR}
The reproducing kernel for $\HH_p(\R^n)$ is $(\dim \tau) P^{\circ}$, so that for all $P \in \HH_p(\R^n)$ and $k \in \Ogp(n)$,
\[P(e_n k) = (\dim \tau) \left\langle \tau(k) \cdot P, P^{\circ}\right\rangle.\]
In particular, $\langle P^{\circ}, P^{\circ} \rangle = (\dim \tau)^{-1}$. Moreover, for all $P \in \HH_p(\R^n)$ and $x \in \R^n$,
\[P(x) = \dim \tau \int_{\Ogp(n)} P(e_n k^{-1}) P^{\circ}(xk) \, dk.\]
\end{lemma}

\begin{lemma}[Hecke's Identity {\cite[Chapter IV, Theorem 3.4]{SW71}}]
\label{lem:HeckeR}
For any homogeneous harmonic polynomial $P \in \HH_p(\R^n)$ and $\xi \in \R^n$, we have that
\[\int_{\R^n} P(x) \exp\left(-\pi x \prescript{t}{}{x}\right) \overline{\psi}(x \prescript{t}{}{\xi}) \, dx = i^{-p} P(\xi) \exp\left(-\pi \xi \prescript{t}{}{\xi}\right).\]
\end{lemma}

\section{The Newform in the Induced Model}
\label{sect:newforminduced}

\subsection{The Induced Model}

Let $\pi = \pi_1 \boxplus \cdots \boxplus \pi_r$ be an induced representation of Whittaker type. Let $V_{\pi_j}$ be the space of $\pi_j$; the space $V_{\pi}$ of $\pi$ may then be viewed as the space of smooth functions $f : \GL_n(F) \to V_{\pi_1} \otimes \cdots \otimes V_{\pi_r}$ that satisfy
\[f(umg) = \delta_{\Pgp}^{1/2}(m) \pi_1(m_1) \otimes \cdots \otimes \pi_r(m_r) \cdot f(g),\]
for any $u \in \Ngp_{\Pgp}(F)$, $m = \blockdiag(m_1,\ldots,m_r) \in \Mgp_{\Pgp}(F)$, and $g \in \GL_n(F)$, where $\Pgp(F) = \Pgp_{(n_1,\ldots,n_r)}(F)$. The action of $\pi$ on $V_{\pi}$ is via right translation, namely $(\pi(h) \cdot f)(g) \coloneqq f(gh)$.

To make this more explicit, we first describe the space $V_{\pi_j}$ of the essentially square-integrable representation $\pi_j$ of $\GL_{n_j}(F)$. The following result is well-known; see, for example, \cite[Chapter 7]{GH11}.

\begin{lemma}
\label{lem:discreteseriessubrep}
\hspace{1em}
\begin{enumerate}[leftmargin=*]
\item[\textnormal{(1)}] Let $\pi = \chi^{\kappa} |\cdot|^{t}$ be a character of $\GL_1(F) = F^{\times}$. The space $V_{\pi}$ of $\pi$ is simply the one-dimensional vector space spanned by the function $\chi^{\kappa}(x) |x|^{t}$.
\item[\textnormal{(2)}] Let $\pi = D_{\kappa} \otimes \left|\det\right|^t$ be an essentially discrete series representation of $\GL_2(\R)$. Then $\pi$ is a subrepresentation of the reducible principal series representation $\pi^{\sharp} \coloneqq |\cdot|^{t + \frac{\kappa - 1}{2}} \boxplus \chi^{\kappa} |\cdot|^{t - \frac{\kappa - 1}{2}}$ of $\GL_2(\R)$. Moreover, if $V_{\pi^{\sharp}}$ denotes the induced model of $\pi^{\sharp}$ consisting of smooth functions $f : \GL_2(\R) \to \C$ that satisfy
\[f(uag) = |a_1|^{t + \frac{\kappa}{2}} \chi^{\kappa}(a_2) |a_2|^{t - \frac{\kappa}{2}} f(g)\]
for all $u \in \Ngp_2(\R)$, $a = \diag(a_1,a_2) \in \Agp_2(\R)$, and $g \in \GL_2(\R)$, where $\pi^{\sharp}$ acts on $V_{\pi^{\sharp}}$ via right translation, then the induced model of $\pi$ is precisely the subspace
\[V_{\pi} \coloneqq \bigcap_{\substack{\mu_1 = 0 \\ \mu_1 \equiv \kappa \hspace{-.25cm} \pmod{2}}}^{\kappa - 2} \ker \Pi^{\tau_{(\mu_1,0)}}\]
of $V_{\pi^{\sharp}}$, where $\Pi^{\tau}$ denotes the projection of $V_{\pi^{\sharp}}$ onto the $\tau$-isotypic component $V_{\pi^{\sharp}}^{\tau}$.
\end{enumerate}
\end{lemma}

\begin{corollary}
Let $\pi = \pi_1 \boxplus \cdots \boxplus \pi_r$ be an induced representation of Whittaker type. Then the induced model of $\pi$ may be taken to be the space $V_{\pi}$ of smooth functions $f : \GL_n(F) \times \Mgp_{\Pgp}(F) \to \C$ satisfying
\begin{equation}
\label{eq:inducedmodulus}
f(umg;m') = \delta_{\Pgp}^{1/2}(m) f(g; m' m)
\end{equation}
for all $u \in \Ngp_{\Pgp}(F)$, $m,m' \in \Mgp_{\Pgp}(F)$, and $g \in \GL_n(F)$ and such that for each $g \in \GL_n(F)$, $f(g;\cdot) : \Mgp_{\Pgp}(F) \to \C$ is an element of $V_{\pi_1} \otimes \cdots \otimes V_{\pi_r}$ with $V_{\pi_j}$ as in \hyperref[lem:discreteseriessubrep]{Lemma \ref*{lem:discreteseriessubrep}}.
\end{corollary}

For $f \in V_{\pi}$, we write $f(g)$ to denote $f(g;1_n)$.

\begin{example}
Suppose that $\pi = \chi^{\kappa_1} |\cdot|^{t_1} \boxplus \cdots \boxplus \chi^{\kappa_n} |\cdot|^{t_n}$ is a principal series representation, so that $\kappa_j \in \Z$ for $F = \C$ for $F = \R$ and $\kappa_j \in \{0,1\}$. The induced model of $\pi$ is the vector space $V_{\pi}$ of smooth functions $f : \GL_n(F) \to \C$ that satisfy
\[f(uag) = f(g) \delta_n^{1/2}(a) \prod_{j = 1}^{n} \chi^{\kappa_j}(a_j) |a_j|^{t_j}\]
for all $u \in \Ngp_n(F)$, $a = \diag(a_1,\ldots,a_n) \in \Agp_n(F)$, and $g \in \GL_n(F)$.
\end{example}

Our goal now is to explicitly describe the newform $f^{\circ}$ in the induced model $V_{\pi}$ of an induced representation of Whittaker type $\pi$. We give three different explicit constructions: via the Iwasawa decomposition, via convolution sections, and via Godement sections. Initially, we define the newform in the induced model only up to multiplication by a nonzero constant; eventually in \hyperref[def:cannormsq]{Definitions \ref*{def:cannormsq}} and \ref{def:cannorm} we specify a normalisation that is particularly useful when proceeding to study the newform in the Whittaker model.

\subsection{The Newform via the Iwasawa Decomposition}

\subsubsection{Essentially Square-Integrable Representations}

We first describe the newform in the induced model of essentially square-integrable representations.

\begin{lemma}
\hspace{1em}
\begin{enumerate}[leftmargin=*]
\item[\textnormal{(1)}] For $\pi = \chi^{\kappa} |\cdot|^t$, the newform in the induced model is simply
\begin{equation}
\label{eq:fcircchi}
f^{\circ}(x) = c^{\circ} \chi^{\kappa}(x) |x|^t
\end{equation}
for any $c^{\circ} \in \C^{\times}$.
\item[\textnormal{(2)}] For $F = \R$ and $\pi = D_{\kappa} \otimes \left|\det\right|^t$, the newform in the induced model of $\pi$ is
\begin{equation}
\label{eq:fcircD}
f^{\circ}(g) = c^{\circ} |a_1|^{t + \frac{\kappa}{2}} \chi^{\kappa}(a_2) |a_2|^{t - \frac{\kappa}{2}} \overline{P^{\circ}}(e_2 k^{-1})
\end{equation}
for any $c^{\circ} \in \C^{\times}$ and $g \in \GL_2(\R)$ having the Iwasawa decomposition $g = uak$ with $u \in \Ngp_2(\R)$, $a = \diag(a_1,a_2) \in \Agp_2(\R)$, and $k \in \Ogp(2)$, where $P^{\circ}$ is the homogeneous harmonic polynomial associated to the newform $K$-type $\tau^{\circ}$ given by \eqref{eq:PjcircR2}.
\end{enumerate}
\end{lemma}

\begin{remark}
Strictly speaking, there is no need to write $\overline{P^{\circ}}$ instead of $P^{\circ}$, since this is real-valued; we do this simply to ensure notational consistency when later treating the cases $F = \R$ and $F = \C$ simultaneously, for the distinction is no longer moot in the latter case.
\end{remark}

\begin{proof}
This is clear for $\pi = \chi^{\kappa} |\cdot|^t$. For $\pi = D_{\kappa} \otimes \left|\det\right|^t$, we must first check that $f^{\circ}$ is well-defined, for the Iwasawa decomposition is not unique as $\Agp_2(\R)$ and $\Ogp(2)$ intersect nontrivially. If $a' = \diag(a_1',a_2') \in \Agp_2(\R) \cap \Ogp(2)$, so that $a_1',a_2' \in \{1,-1\}$, then on the one hand,
\[f^{\circ}(uaa'k) = c^{\circ} |a_1|^{t + \frac{\kappa}{2}} \chi^{\kappa}(a_2') \chi^{\kappa}(a_2) |a_2|^{t - \frac{\kappa}{2}} \overline{P^{\circ}}(e_2 k^{-1})\]
since $a' \in \Agp_2(\R)$ with $|a_1'| = |a_2'| = 1$, while on the other hand,
\[f^{\circ}(uaa'k) = c^{\circ} |a_1|^{t + \frac{\kappa}{2}} \chi^{\kappa}(a_2) |a_2|^{t - \frac{\kappa}{2}} \overline{P^{\circ}}(e_2 a' k^{-1})\]
since $a' \in \Ogp(2)$, and these are equal since $e_2 a' = a_2' e_2$, $P^{\circ}$ is homogeneous of degree $\kappa$ as $\tau^{\circ} = \tau_{(\kappa,0)}$, so that $P^{\circ} \in \HH_{\kappa}(\R^2)$, and $a_2^{\prime \kappa} = \chi^{\kappa}(a_2')$ as $a_2' \in \{1,-1\}$.

Next, Schur orthogonality shows that $f^{\circ} \in \ker \Pi^{\tau_{(\mu_1,0)}}$ for $0 \leq \mu_1 \leq \kappa - 2$ with $\mu_1 \equiv \kappa \pmod{2}$ since $P^{\circ} \in \HH_{\kappa}(\R^2)$, and so $f^{\circ}$ is indeed an element of the induced model $V_{\pi}$ of $\pi$ as defined in \hyperref[lem:discreteseriessubrep]{Lemma \ref*{lem:discreteseriessubrep}}.

Finally, to prove that $f^{\circ}$ is the newform, we must show that
\[\int_{\Ogp(2)} \xi^{\tau^{\circ}|_{\Ogp(1)}}(k) (\pi(k) \cdot f^{\circ})(g) \, dk = f^{\circ}(g).\]
From the definition \eqref{eq:fcircD} of $f^{\circ}$, it suffices to show that
\begin{equation}
\label{eq:newformcheck0}
\int_{\Ogp(2)} \xi^{\tau^{\circ}|_{\Ogp(1)}}(k) \overline{P^{\circ}}\left(e_2 k^{-1} k^{\prime -1}\right) \, dk = \overline{P^{\circ}}\left(e_2 k^{\prime -1}\right)
\end{equation}
for any $k' \in \Ogp(2)$. We note that
\[\xi^{\tau^{\circ}|_{\Ogp(1)}}(k) = (\dim \tau^{\circ}) \frac{\left\langle \tau^{\circ}(k^{-1}) \cdot P^{\circ}, P^{\circ} \right\rangle}{\left\langle P^{\circ}, P^{\circ} \right\rangle} = (\dim \tau^{\circ}) P^{\circ}(e_2 k^{-1}),\]
where the first equality follows from the definitions \eqref{eq:xitauKn-1} of $\xi^{\tau^{\circ}|_{\Ogp(1)}}$ and \eqref{eq:PjcircR2} of $P^{\circ}$, while the second follows from \hyperref[lem:QenkR]{Lemma \ref*{lem:QenkR}}. The equality \eqref{eq:newformcheck0} then follows from one more application of \hyperref[lem:QenkR]{Lemma \ref*{lem:QenkR}}.
\end{proof}

For our applications, we require explicit choices of the constants $c^{\circ}$ appearing in \eqref{eq:fcircchi} and \eqref{eq:fcircD}.

\begin{definition}
\label{def:cannormsq}
Let $\pi$ be an essentially square-integrable representation of $\GL_n(F)$, and let $f^{\circ}$ denote the newform in the induced model, which is given by \eqref{eq:fcircchi} if $n = 1$, so that $\pi = \chi^{\kappa} |\cdot|^t$, and by \eqref{eq:fcircD} if $n = 2$, so that $F = \R$ and $\pi = D_{\kappa} \otimes \left|\det\right|^t$. We say that $f^{\circ}$ is canonically normalised if
\begin{equation}
\label{eq:cannormsq}
c^{\circ} = \begin{dcases*}
1 & if $\pi = \chi^{\kappa} |\cdot|^t$,	\\
i^{\kappa} \zeta_{\R}(\kappa) \zeta_{\R}(\kappa + 1) & if $\pi = D_{\kappa} \otimes \left|\det\right|^t$.
\end{dcases*}
\end{equation}
\end{definition}

\subsubsection{Induced Representations of Whittaker Type}

We now give an explicit construction of the newform $f^{\circ} : \GL_n(F) \times \Mgp_{\Pgp}(F) \to \C$ in the induced model of an induced representation of Whittaker type $\pi = \pi_1 \boxplus \cdots \boxplus \pi_r$ of $\GL_n(F)$ when $g \in \GL_n(F)$ is written in terms of its Iwasawa decomposition.

This description involve a distinguished homogeneous harmonic polynomial $P_{(n_1,\ldots,n_r)}^{\circ}$ that is defined in terms of polynomials $P_j^{\circ}$ in the following way. To each essentially square-integrable representation $\pi_j$ of $\GL_{n_j}(F)$ with newform $K$-type $\tau_j^{\circ}$, we associate a distinguished homogeneous harmonic polynomial $P_j^{\circ}$.
\begin{itemize}
\item For $F = \R$, $n_j = 1$, $\pi_j = \chi^{\kappa_j} |\cdot|_{\R}^{t_j}$, and $\tau_j^{\circ}$ the one-dimensional representation of $\Ogp(1)$ of highest weight $\kappa_j \in \{0,1\}$, $P_j^{\circ} \in \HH_p(\R)$ is the homogeneous harmonic polynomial associated to $\tau = \tau_j^{\circ}$ given by \eqref{eq:PjcircR1} with $p = \kappa_j$.
\item For $F = \R$, $n_j = 2$, $\pi_j = D_{\kappa_j} \otimes \left|\det\right|_{\R}^{t_j}$, and $\tau_j^{\circ}$ the two-dimensional representation of $\Ogp(2)$ of highest weight $(\kappa_j,0)$ with $\kappa_j \geq 2$ a positive integer, $P_j^{\circ} \in \HH_p(\R^2)$ is the homogeneous harmonic polynomial associated to $\tau = \tau_j^{\circ}$ given by \eqref{eq:PjcircR2} with $p = \kappa_j$.
\item For $F = \C$, so that $n_j = 1$, $\pi_j = \chi^{\kappa_j} |\cdot|_{\C}^{t_j}$, and $\tau_j^{\circ}$ the one-dimensional representation of $\Ugp(1)$ of highest weight $\kappa_j \in \Z$, $P_j^{\circ} \in \HH_{p,q}(\C)$ is the homogeneous harmonic polynomial associated to $\tau = \tau_j^{\circ}$ given by \eqref{eq:PjcircC} with $p = \max\{\kappa_j,0\}$ and $q = -\min\{\kappa_j,0\}$.
\end{itemize}
We then define
\begin{align}
\label{eq:PstarR}
P_{(n_1,\ldots,n_r)}^{\circ}(x) & \coloneqq \prod_{j = 1}^{r} \begin{dcases*}
P_j^{\circ}(x_{\ell}) & if $\ell = n_1 + \cdots + n_j$ and $\pi_j = \chi^{\kappa_j} |\cdot|_{\R}^{t_j}$,	\\
P_j^{\circ}(x_{\ell},x_{\ell + 1}) & if $\ell = n_1 + \cdots + n_j - 1$ and $\pi_j = D_{\kappa_j} \otimes \left|\det\right|_{\R}^{t_j}$
\end{dcases*}
\intertext{for $F = \R$, while for $F = \C$, we define}
\label{eq:PstarC}
P_{(n_1,\ldots,n_r)}^{\circ}(z) & \coloneqq \prod_{j = 1}^{r} P_j^{\circ}(z_j).
\end{align}
It is straightforward to see that the polynomials $P_{(n_1,\ldots,n_r)}^{\circ}(x)$ and $P_{(n_1,\ldots,n_r)}^{\circ}(z)$ are elements of $\HH_{\mu_1^{\circ}}(\R^n)$ and $\HH_{\mu_1^{\circ},-\mu_n^{\circ}}(\C^n)$ respectively, where $\mu^{\circ} = (\mu_1^{\circ},\ldots,\mu_n^{\circ})$ is the highest weight of the newform $K$-type $\tau^{\circ} = \tau_{\mu^{\circ}}$ of $\pi = \pi_1 \boxplus \cdots \boxplus \pi_r$.

The description of the newform $f^{\circ}$ in the induced model of $\pi = \pi_1 \boxplus \cdots \boxplus \pi_r$ also involves the canonically normalised newforms $f_1^{\circ},\ldots,f_r^{\circ}$ of the essentially square-integrable representations $\pi_1,\ldots,\pi_r$.

\begin{proposition}
\label{prop:newformidentity}
Let $\pi = \pi_1 \boxplus \cdots \boxplus \pi_r$ be an induced representation of Whittaker type with $r \geq 2$. For $g \in \GL_n(F)$ having the Iwasawa decomposition $g = umk$ with respect to the parabolic subgroup $\Pgp(F) = \Pgp_{(n_1,\ldots,n_r)}(F)$, so that $u \in \Ngp_{\Pgp}(F)$, $m = \blockdiag(m_1,\ldots,m_r) \in \Mgp_{\Pgp}(F)$, and $k \in K$, and for $m' = \blockdiag(m_1',\ldots,m_r') \in \Mgp_{\Pgp}(F)$, the newform $f^{\circ} : \GL_n(F) \times \Mgp_{\Pgp}(F) \to \C$ in the induced model $V_{\pi}$ of $\pi$ is of the form
\begin{multline}
\label{eq:fcirc}
f^{\circ}(g;m') \coloneqq c^{\circ} \dim \tau_1^{\circ} \cdots \dim \tau_r^{\circ} \delta_{\Pgp}^{1/2}(m) \int_{K_{n_1}} \hspace{-.3cm} \cdots \int_{K_{n_r}} f_1^{\circ}\left(m_1' m_1 k_1\right) \cdots f_r^{\circ}\left(m_r' m_r k_r\right)	\\
\times \overline{P_{(n_1,\ldots,n_r)}^{\circ}}\left(e_n k^{-1} \blockdiag\left(k_1,\ldots,k_r\right)\right) \, dk_r \cdots dk_1
\end{multline}
for some constant $c^{\circ} \in \C^{\times}$, where each $f_j^{\circ}$ is the canonically normalised newform of $\pi_j$ and $\tau_j^{\circ}$ is the newform $K_{n_j}$-type of $\pi_j$.
\end{proposition}

\begin{proof}
Since the Iwasawa decomposition is not unique as $\Mgp_{\Pgp}(F)$ and $K$ intersect nontrivially, our first task is to show that $f^{\circ}(umm''k;m')$ is well-defined for $m'' = \blockdiag(m_1'',\ldots,m_r'') \in \Mgp_{\Pgp}(F) \cap K$. On the one hand, this is
\begin{multline*}
c^{\circ} \dim \tau_1^{\circ} \cdots \dim \tau_r^{\circ} \delta_{\Pgp}^{1/2}(m) \int_{K_{n_1}} \hspace{-.3cm} \cdots \int_{K_{n_r}} f_1^{\circ}\left(m_1' m_1 m_1'' k_1\right) \cdots f_r^{\circ}\left(m_r' m_r m_r'' k_r\right)	\\
\times \overline{P_{(n_1,\ldots,n_r)}^{\circ}}\left(e_n k^{-1} \blockdiag\left(k_1,\ldots,k_r\right)\right) \, dk_r \cdots dk_1
\end{multline*}
since $m'' \in \Mgp_{\Pgp}(F)$, noting that $\delta_{\Pgp}(m'') = 1$ as $m'' \in K$. On the other hand, this is
\begin{multline*}
c^{\circ} \dim \tau_1^{\circ} \cdots \dim \tau_r^{\circ} \delta_{\Pgp}^{1/2}(m) \int_{K_{n_1}} \hspace{-.3cm} \cdots \int_{K_{n_r}} f_1^{\circ}\left(m_1' m_1 k_1\right) \cdots f_r^{\circ}\left(m_r' m_r k_r\right)	\\
\times \overline{P_{(n_1,\ldots,n_r)}^{\circ}}\left(e_n k^{-1} {m''}^{-1} \blockdiag\left(k_1, \ldots, k_r\right)\right) \, dk_r \cdots dk_1
\end{multline*}
since $m'' \in K$; as $m_j'' \in K_{n_j}$, this is seen to be equal to the first expression upon making the change of variables $k_j \mapsto m_j'' k_j$.

Next, we confirm that this is an element of the induced model $V_{\pi}$ of $\pi$. It is clear that $f^{\circ}$ is a smooth function from $\GL_n(F) \times \Mgp_{\Pgp}(F)$ to $\C$ that satisfies \eqref{eq:inducedmodulus}. Moreover, $f^{\circ}(g;\cdot)$ is indeed an element of $V_{\pi_1} \otimes \cdots \otimes V_{\pi_r}$ for each $g \in \GL_n(F)$, since upon writing $P_{(n_1,\ldots,n_r)}^{\circ} = \prod_{j = 1}^{r} P_j^{\circ}$, the integrals over $K_{n_j} \ni k_j$ are either trivial if $n_j = 1$, or lead to $f_j^{\circ} \overline{P_j^{\circ}}$ being replaced with the sum of two such products of elements of $V_{\pi_j}$ and homogeneous harmonic polynomials, for we may use Schur orthogonality for the two-dimensional representation $\tau_j^{\circ} = \tau_{(\kappa_j,0)}$ for $\pi_j = D_{\kappa_j} \otimes \left|\det\right|^{t_j}$.

Finally, we show that this is the newform, which requires confirming that
\[\int_{K_n} \xi^{\tau^{\circ},K_{n - 1}}(k) (\pi(k) \cdot f^{\circ})(g;m') \, dk = f^{\circ}(g;m').\]
From the definition \eqref{eq:fcirc} of $f^{\circ}$ together with the Iwasawa decomposition, it suffices to show that for each $k' \in K_n$,
\begin{equation}
\label{eq:newformcheck}
\int_{K_n} \xi^{\tau^{\circ},K_{n - 1}}(k) \overline{P_{(n_1,\ldots,n_r)}^{\circ}}\left(e_n k^{-1} k^{\prime -1}\right) \, dk = \overline{P_{(n_1,\ldots,n_r)}^{\circ}}\left(e_n k^{\prime -1}\right).
\end{equation}
We note that
\[\xi^{\tau^{\circ},K_{n - 1}}(k) = (\dim \tau^{\circ}) \frac{\left\langle \tau^{\circ}(k^{-1}) \cdot P^{\circ}, P^{\circ} \right\rangle}{\left\langle P^{\circ}, P^{\circ} \right\rangle} = (\dim \tau^{\circ}) P^{\circ}(e_n k^{-1}),\]
where the first equality follows from the definitions \eqref{eq:xitauKn-1} of $\xi^{\tau^{\circ},K_{n - 1}}$ and \eqref{eq:PcircUn} and \eqref{eq:PcircOn} of $P^{\circ}$, while the second follows from \hyperref[lem:QenkC]{Lemmata \ref*{lem:QenkC}} and \ref{lem:QenkR}. The equality \eqref{eq:newformcheck} then follows from one more application of \hyperref[lem:QenkC]{Lemmata \ref*{lem:QenkC}} and \ref{lem:QenkR}.
\end{proof}

\hyperref[prop:newformidentity]{Proposition \ref*{prop:newformidentity}} completely prescribes the behaviour of the newform $f^{\circ}(g) \coloneqq f^{\circ}(g;1_n)$ when $g = uak$ is given by the Iwasawa decomposition with respect to the standard Borel subgroup.

\begin{corollary}
\label{cor:newformIwasawa}
For $u \in \Ngp_n(F)$, $a = \diag(a_1,\ldots,a_n) \in \Agp_n(F)$, and $g \in \GL_n(F)$, the newform in the induced model satisfies
\begin{equation}
\label{eq:induced}
f^{\circ}(uag) = f^{\circ}(g) \delta_n^{1/2}(a) \prod_{j = 1}^{r} \begin{dcases*}
\chi^{\kappa_j}(a_{\ell}) |a_{\ell}|^{t_j} & \parbox{.26\textwidth}{if $\ell = n_1 + \cdots + n_j$ and $\pi_j = \chi^{\kappa_j} |\cdot|^{t_j}$,}	\\
\chi^{\kappa_j}(a_{\ell + 1}) |a_{\ell}|^{t_j + \frac{\kappa_j - 1}{2}} |a_{\ell + 1}|^{t_j - \frac{\kappa_j - 1}{2}} & 
\parbox{.26\textwidth}{if $\ell = n_1 + \cdots + n_j - 1$ and $\pi_j = D_{\kappa_j} \otimes \left|\det\right|^{t_j}$,}
\end{dcases*}
\end{equation}
and for $k \in K_n$,
\begin{equation}
\label{eq:inducedK}
\begin{split}
f^{\circ}(k) & = c^{\circ} c_1^{\circ} \cdots c_r^{\circ} \overline{P_{(n_1,\ldots,n_r)}^{\circ}}(e_n k^{-1})	\\
& = c^{\circ} \prod_{j = 1}^{r} \begin{dcases*}
\overline{P_j^{\circ}}\left(\overline{k_{\ell,n}}\right) & \parbox{.33\textwidth}{if $\ell = n_1 + \cdots + n_j$ and $\pi_j = \chi^{\kappa_j} |\cdot|^{t_j}$,}	\\
i^{\kappa_j} \zeta_{\R}(\kappa_j) \zeta_{\R}(\kappa_j + 1) \overline{P_j^{\circ}}\left(\overline{k_{\ell,n}}, \overline{k_{\ell + 1,n}}\right) & 
\parbox{.33\textwidth}{if $\ell = n_1 + \cdots + n_j - 1$ and $\pi_j = D_{\kappa_j} \otimes \left|\det\right|^{t_j}$.}
\end{dcases*}
\end{split}
\end{equation}
\end{corollary}

\begin{proof}
Via the Iwasawa decomposition, it suffices to prove \eqref{eq:induced} for $g = k \in K$. We write $u \in \Ngp_n(F)$ as $u' u''$ with $u' \in \Ngp_{\Pgp}(F)$ and $u'' = \blockdiag(u_1'',\ldots,u_r'') \in \Mgp_{\Pgp}(F)$ with $u_j'' \in \Ngp_{n_j}(F)$, so that $u'' a \in \Mgp_{\Pgp}(F)$ whenever $a \in \Agp_n(F)$. We then take $g = uak = u' u''a k$ and $m' = 1_n$ in \eqref{eq:fcirc} and apply \eqref{eq:fcircchi} and \eqref{eq:fcircD} to deduce \eqref{eq:induced}. The identity \eqref{eq:inducedK} then follows upon taking $u = a = 1_n$, writing $P_{(n_1,\ldots,n_r)}^{\circ} = \prod_{j = 1}^{r} P_j^{\circ}$, and invoking \hyperref[lem:QenkC]{Lemmata \ref*{lem:QenkC}} and \ref{lem:QenkR} to evaluate the integrals over $K_{n_j} \ni k_j$.
\end{proof}

\subsection{The Newform via Convolution Sections}

We now give a different description of the newform in the induced model. This description is a recursive formula for $f^{\circ}$ as an integral over $\GL_n(F)$ involving $f^{\circ}$ itself and a distinguished standard Schwartz function $\Phi \in \Ss_0(\Mat_{n \times n}(F))$, where the space of standard Schwartz functions $\Ss_0(\Mat_{n \times n}(F))$ consists of functions $\Phi : \Mat_{n \times n}(F) \to \C$ of the form
\[\Phi(x) = P(x) \exp\left(-d_F \pi \Tr\left(x \prescript{t}{}{\overline{x}}\right)\right)\]
with $P$ a polynomial in the entries of $x$ and $\overline{x}$ and $d_F \coloneqq [F:\R]$ as in \eqref{eq:Phicirc}. When $\pi$ is a spherical representation and $f$ is the spherical vector, such a formula is known by the work of Gerasimov, Lebedev, and Oblezin \cite[Theorem 5.1]{GLO08} and Ishii and Stade \cite[Proposition 2.6]{IsSt13} (with the latter expressed in terms of the Mellin transform of the Whittaker function); see also \cite[Section 5]{IsSt13}.

\begin{proposition}
\label{prop:Pieri}
Let $\pi = \pi_1 \boxplus \cdots \boxplus \pi_r$ be an induced representation of Whittaker type of $\GL_n(F)$ with newform $f^{\circ}$ in the induced model $V_{\pi}$. Then for all $h \in \GL_n(F)$ and for $\Re(s)$ sufficiently large,
\begin{equation}
\label{eq:Pieri}
\int_{\GL_n(F)} f^{\circ}(hg) \Phi(g) \left|\det g\right|^{s + \frac{n - 1}{2}} \, dg = L(s,\pi) f^{\circ}(h),
\end{equation}
where $\Phi \in \Ss_0(\Mat_{n \times n}(F))$ is the standard Schwartz function
\begin{equation}
\label{eq:Phistandard}
\Phi(x) \coloneqq (\dim \tau^{\circ}) \overline{P^{\circ}}(e_n x) \exp\left(-d_F \pi \Tr\left(x \prescript{t}{}{\overline{x}}\right)\right),
\end{equation}
with $P^{\circ}$ the homogeneous harmonic polynomial associated to the newform $K$-type $\tau^{\circ}$ of $\pi$ via \eqref{eq:PcircUn} and \eqref{eq:PcircOn}. 
\end{proposition}

In particular, the integral \eqref{eq:Pieri} converges absolutely if $\Re(s) > -\Re(t_j)$ for each $j \in \{1,\ldots,r\}$ for which $n_j = 1$, so that $\pi_j = \chi^{\kappa_j} |\cdot|^{t_j}$, and $\Re(s) > -\Re(t_j) + (\kappa_j - 1)/2$ for each $j \in \{1,\ldots,r\}$ for which $n_j = 2$, so that $F = \R$ and $\pi_j = D_{\kappa_j} \otimes \left|\det\right|^{t_j}$.

We may think of the integral \eqref{eq:Pieri} as defining a convolution section of $V_{\pi}$ in the sense of Jacquet \cite{Jac04}, where the convolution is with respect to the function $\phi(g) \coloneqq \Phi(g) \left|\det g\right|^{s + \frac{n - 1}{2}}$. (Note that Jacquet deals only with functions $\phi : \GL_n(F) \to \C$ that are smooth and compactly supported.) Alternatively, the identity \eqref{eq:Pieri} may be thought of as a Pieri-type formula, generalising \cite[Theorem 3.8]{Ish18}.

\begin{proof}
Via the Iwasawa decomposition with respect to the standard Borel subgroup and \eqref{eq:induced}, it suffices to show the identity \eqref{eq:Pieri} for $h = k \in K_n$. We make the change of variables $g \mapsto k^{-1} g$, then use the Iwasawa decomposition $g = umk'$ with respect to the parabolic subgroup $\Pgp(F) = \Pgp_{(n_1,\ldots,n_r)}(F)$; the Haar measure becomes $dg = \delta_{\Pgp}^{-1}(m) \, du \, d^{\times}m \, dk'$. We see that the left-hand side of \eqref{eq:Pieri} is
\begin{equation}
\label{eq:Pieriintabsconv}
\int_{\Mgp_{\Pgp}(F)} \left|\det m\right|^{s + \frac{n - 1}{2}} \delta_{\Pgp}^{-1}(m) \int_{\Ngp_{\Pgp}(F)} \int_{K_n} f^{\circ}(umk') \Phi(k^{-1} umk') \, dk' \, du \, d^{\times} m.
\end{equation}
The absolute convergence of this integral for $\Re(s)$ sufficiently large is not difficult; it follows directly from the definitions \eqref{eq:fcirc} of the newform in the induced model and \eqref{eq:Phistandard} of the standard Schwartz function $\Phi \in \Ss_0(\Mat_{n \times n}(F))$ together with the bounds from \cite[Lemma 3.3 (ii)]{Jac09}.

We may evaluate the integral over $K_n \ni k'$ in \eqref{eq:Pieriintabsconv} by inserting \eqref{eq:fcirc} and \eqref{eq:Phistandard} and using \hyperref[lem:QenkC]{Lemmata \ref*{lem:QenkC}} and \ref{lem:QenkR}. We subsequently make the change of variables $m_j \mapsto m_j k_j^{-1}$, where $m = \blockdiag(m_1,\ldots,m_r)$, to trivially evaluate the integrals over $K_{n_j} \ni k_j$, leading to
\begin{multline}
\label{eq:Pieriint}
c^{\circ} \dim \tau_1^{\circ} \cdots \dim \tau_r^{\circ} \int_{\Mgp_{\Pgp}(F)} \int_{\Ngp_{\Pgp}(F)} \left|\det m\right|^{s + \frac{n - 1}{2}} \delta_{\Pgp}^{-1/2}(m)	\\
\times f_1^{\circ}(m_1) \cdots f_r^{\circ}(m_r) \overline{P_{(n_1,\ldots,n_r)}^{\circ}}(e_n k^{-1} um) \exp\left(-d_F \pi \Tr\left(u m \prescript{t}{}{\overline{m}} \prescript{t}{}{\overline{u}}\right)\right) \, du \, d^{\times} m,
\end{multline}
We evaluate the integrals over $\Mgp_{\Pgp}(F) \ni m$ and $\Ngp_{\Pgp}(F) \ni u$ in \eqref{eq:Pieriint} by breaking these integrals up into parts, where this decomposition is dependent on the size of $n_j \in \{1,2\}$ for $j \in \{1,\ldots,r\}$. In doing so, we use the fact that $k^{-1} = \prescript{t}{}{\overline{k}}$ and recall the definitions \eqref{eq:PstarC} and \eqref{eq:PstarR} of the polynomial $P_{(n_1,\ldots,n_r)}^{\circ}$ in order to write
\[\dim \tau_1^{\circ} \cdots \dim \tau_r^{\circ} \overline{P_{(n_1,\ldots,n_r)}^{\circ}}(e_n g) = \prod_{j = 1}^{r} \dim \tau_j^{\circ} \begin{dcases*}
\overline{P_j^{\circ}}(g_{n,\ell}) & \parbox{.26\textwidth}{if $\ell = n_1 + \cdots + n_j$ and $\pi_j = \chi^{\kappa_j} |\cdot|^{t_j}$,}	\\
\overline{P_j^{\circ}}(g_{n,\ell},g_{n,\ell + 1}) & \parbox{.26\textwidth}{if $\ell = n_1 + \cdots + n_j - 1$ and $\pi_j = D_{\kappa_j} \otimes \left|\det\right|^{t_j}$.}
\end{dcases*}\]

We first deal with the case of $n_j = 1$, so that $\pi_j = \chi^{\kappa_j} |\cdot|^{t_j}$; in this case, we evaluate the integrals over $F^{\times} \ni m_j$ and $F \ni u_{i,\ell}$ with $i \in \{1,\ldots,\ell - 1\}$ for $\ell = n_1 + \cdots + n_j$. After making the change of variables $u_{i,\ell} \mapsto m_j^{-1} u_{i,\ell}$, recalling the definitions \eqref{eq:fcircchi} of the newform in the induced model $f_j^{\circ}$ and \eqref{eq:PjcircC} and \eqref{eq:PjcircR1} of the polynomial $P_j^{\circ}$, and expanding this polynomial via the multinomial theorem, we are left with evaluating
\begin{multline*}
c_j^{\circ} \sum_{\nu_1 + \cdots + \nu_{\ell} = \|\kappa_j\|} \binom{\|\kappa_j\|}{\nu_1, \ldots, \nu_{\ell}} \prod_{i = 1}^{\ell} k_{i,n}^{\max\{\sgn(\kappa_j) \nu_i,0\}} \overline{k_{i,n}}^{-\min\{\sgn(\kappa_j) \nu_i,0\}}	\\
\times \int_{F^{\times}} \overline{m_j}^{\max\{\sgn(\kappa_j) \nu_{\ell},0\}} m_j^{- \min\{\sgn(\kappa_j) \nu_{\ell},0\}} \chi^{\kappa_j}(m_j) |m_j|^{s + t_j} \exp\left(-d_F \pi m_j \overline{m_j}\right) \, dm_j	\\
\times \prod_{i = 1}^{\ell - 1} \int_F \overline{u_{i,\ell}}^{\max\{\sgn(\kappa_j) \nu_i,0\}} u_{i,\ell}^{-\min\{\sgn(\kappa_j) \nu_i,0\}} \exp\left(-d_F \pi u_{i,\ell} \overline{u_{i,\ell}}\right) \, du_{i,\ell}.
\end{multline*}
Here
\[\binom{\kappa}{\nu_1,\ldots,\nu_{\ell}} \coloneqq \frac{\kappa!}{\nu_1! \cdots \nu_{\ell}!}\]
denotes the multinomial coefficient for $\kappa,\nu_1,\ldots,\nu_{\ell} \in \N_0$ with $\nu_1 + \cdots + \nu_{\ell} = \kappa$. The integral over $F \ni u_{i,\ell}$ vanishes unless $\nu_i = 0$, in which case it is $1$, upon applying Hecke's identity, \hyperref[lem:HeckeC]{Lemmata \ref*{lem:HeckeC}} and \ref{lem:HeckeR}. All that remains is the integral over $F^{\times} \ni m_j$, which is equal to
\[c_j^{\circ} k_{\ell,n}^{\max\{\kappa_j,0\}} \overline{k_{\ell,n}}^{-\min\{\kappa_j,0\}} \int_{F^{\times}} |m_j|^{s + t_j + \frac{\|\kappa_j\|}{d_F}} \exp(-d_F \pi m_j \overline{m_j}) \, dm_j = c_j^{\circ} L(s,\pi_j) \overline{P_j^{\circ}}\left(\overline{k_{\ell,n}}\right)\]
having used the fact that
\begin{equation}
\label{eq:abstochar}
|x|^{\frac{\|\kappa\|}{d_F}} = \chi^{-\kappa}(x) x^{\max\{\kappa,0\}} \overline{x}^{-\min\{\kappa,0\}} = \chi^{\kappa}(x) \overline{x}^{\max\{\kappa,0\}} x^{-\min\{\kappa,0\}}
\end{equation}
and recalling the definitions \eqref{eq:PjcircC} and \eqref{eq:PjcircR1} of the polynomial $P_j^{\circ}$, \eqref{eq:L(s,pi)C} and \eqref{eq:L(s,pi)R1} of the $L$-function $L(s,\pi_j)$ in terms of $\zeta_F(s)$, and \eqref{eq:zetaCint} and \eqref{eq:zetaRint} of the zeta function $\zeta_F(s)$ as an integral over $F^{\times}$.

Next, we deal with the case of $n_j = 2$, so that $F = \R$ and $\pi_j = D_{\kappa_j} \otimes \left|\det\right|^{t_j}$; we evaluate the integrals over $\GL_2(\R) \ni m_j$ and $\R^2 \ni (u_{i,\ell}, u_{i,\ell + 1})$ with $i \in \{1,\ldots,\ell - 1\}$ for $\ell = n_1 + \cdots + n_j - 1$. We write $m_j = \begin{psmallmatrix} 1 & u_{\ell,\ell + 1} \\ 0 & 1 \end{psmallmatrix} \begin{psmallmatrix} a_{\ell} & 0 \\ 0 & a_{\ell + 1} \end{psmallmatrix} k'$ for $u_{\ell,\ell + 1} \in \R$, $a_{\ell},a_{\ell + 1} \in \R^{\times}$, and $k' \in \Ogp(2)$; the Haar measure becomes $dm_j = |a_{\ell}|^{-1} |a_{\ell + 1}| \, du_{\ell,\ell + 1} \, d^{\times}a_{\ell} \, d^{\times}a_{\ell + 1} \, dk'$. We use Schur orthogonality to evaluate the integral over $\Ogp(2) \ni k'$, then make the change of variables $u_{\ell,\ell + 1} \mapsto a_{\ell + 1}^{-1} u_{\ell,\ell + 1}$, $u_{i,\ell} \mapsto a_{\ell}^{-1} u_{i,\ell}$, and $u_{i,\ell + 1} \mapsto a_{\ell + 1}^{-1} u_{i,\ell + 1} - u_{\ell,\ell + 1} u_{i,\ell}$ for $i \in \{1,\ldots,\ell - 1\}$. Recalling the definitions \eqref{eq:fcircD} of $f_j^{\circ}$ and \eqref{eq:PjcircR2} of $P_j^{\circ}$, and expanding this polynomial via the multinomial theorem, we are led to
\begin{multline*}
\frac{c_j^{\circ}}{2} \sum_{\pm} \sum_{\nu_1 + \cdots + \nu_{\ell + 1} = \kappa_j} \binom{\kappa_j}{\nu_1, \ldots, \nu_{\ell + 1}} \prod_{i = 1}^{\ell + 1} k_{i,n}^{\nu_i}	\\
\times \int_{\R^{\times}} a_{\ell + 1}^{\nu_{\ell + 1}} \chi^{\kappa_j}(a_{\ell + 1}) |a_{\ell + 1}|^{s + t_j - \frac{\kappa_j - 1}{2}} \exp\left(-\pi a_{\ell + 1}^2\right) d^{\times} a_{\ell + 1}	\\
\times \int_{\R^{\times}} |a_{\ell}|^{s + t_j + \frac{\kappa_j - 1}{2}} \exp\left(-\pi a_{\ell}^2\right) \int_{\R} \left(u_{\ell,\ell + 1} \mp i a_{\ell}\right)^{\nu_{\ell}} \exp\left(-\pi u_{\ell,\ell + 1}^2\right) \, du_{\ell,\ell + 1} \, d^{\times} a_{\ell}	\\
\times \prod_{i = 1}^{\ell - 1} \int_{\R^2} \left(u_{i, \ell + 1} \mp i u_{i,\ell}\right)^{\nu_i} \exp\left(-\pi \left(u_{i,\ell}^2 + u_{i,\ell + 1}^2\right)\right) \, du_{i,\ell} \, du_{i,\ell + 1}.
\end{multline*}
We use Hecke's identity, \hyperref[lem:HeckeR]{Lemma \ref*{lem:HeckeR}}, to see that the integral over $\R^2 \ni (u_{i,\ell},u_{i,\ell + 1})$ vanishes unless $\nu_i = 0$, in which case it is $1$; consequently, the only nonzero summands are those for which $\nu_{\ell} = \kappa_j - \nu_{\ell + 1} \eqqcolon \nu$. For the integral over $\R \ni u_{\ell,\ell + 1}$, we make the change of variables $u_{\ell,\ell + 1} \mapsto u_{\ell,\ell + 1} \pm ia_{\ell}$ and shift the contour of integration back to the line $\Im(u_{\ell,\ell + 1}) = 0$ via Cauchy's integral theorem, for the integrand extends holomorphically to an entire function of the complex variable $u_{\ell,\ell + 1}$. Since
\[\int_{\R} u^{\nu} \exp\left(-\pi u^2\right) \overline{\psi}(ua) \, du = i^{\nu} (2\pi)^{-\nu} \frac{\dee^{\nu}}{\dee a^{\nu}} \exp\left(-\pi a^2\right),\]
we arrive at
\begin{multline*}
c_j^{\circ} \sum_{\substack{\nu = 0 \\ \nu \equiv 0 \hspace{-.25cm} \pmod{2}}}^{\kappa_j} \binom{\kappa_j}{\nu} \left(i k_{\ell,n}\right)^{\nu} k_{\ell + 1,n}^{\kappa_j - \nu}	\\
\times (2\pi)^{-\nu} \int_{\R^{\times}} |a_{\ell + 1}|^{s + t_j + \frac{\kappa_j + 1}{2} - \nu} \exp\left(-\pi a_{\ell + 1}^2\right) d^{\times} a_{\ell + 1} \int_{\R^{\times}} |a_{\ell}|^{s + t_j + \frac{\kappa_j - 1}{2}} \frac{\dee^{\nu}}{\dee a_{\ell}^{\nu}} \exp\left(-\pi a_{\ell}^2\right) \, d^{\times} a_{\ell},
\end{multline*}
having observed the vanishing of the integral over $\R^{\times} \ni a_{\ell + 1}$ for odd $\nu$. We integrate by parts $\nu$ times with respect to $a_{\ell}$, then integrate by parts $\nu/2$ times with respect to $a_{\ell}$, differentiating $\exp(-\pi a_{\ell}^2)$, and $\nu/2$ times with respect to $a_{\ell + 1}$, differentiating $\exp(-\pi a_{\ell + 1}^2)$. We end up at
\begin{multline*}
c_j^{\circ} \sum_{\substack{\nu = 0 \\ \nu \equiv 0 \hspace{-.25cm} \pmod{2}}}^{\kappa_j} \binom{\kappa_j}{\nu} \left(i k_{\ell,n}\right)^{\nu} k_{\ell + 1,n}^{\kappa_j - \nu}	\\
\times \int_{\R^{\times}} |a_{\ell + 1}|^{s + t_j + \frac{\kappa_j + 1}{2}} \exp\left(-\pi a_{\ell + 1}^2\right) d^{\times} a_{\ell + 1} \int_{\R^{\times}} |a_{\ell}|^{s + t_j + \frac{\kappa_j - 1}{2}} \exp\left(-\pi a_{\ell}^2\right) \, d^{\times} a_{\ell}	\\
= c_j^{\circ} L(s,\pi_j) \overline{P_j^{\circ}}\left(\overline{k_{\ell,n}}, \overline{k_{\ell + 1,n}}\right),
\end{multline*}
again recalling the definitions \eqref{eq:PjcircR2} of $P_j^{\circ}$, \eqref{eq:L(s,pi)R2} of $L(s,\pi_j)$ in terms of products of $\zeta_F(s)$, and \eqref{eq:zetaRint} of $\zeta_F(s)$ as an integral over $F^{\times}$.

Combining these calculations, we find that
\begin{multline*}
\int_{\GL_n(F)} f^{\circ}(kg) \Phi(g) \left|\det g\right|^{s + \frac{n - 1}{2}} \, dg	\\
= c^{\circ} \prod_{j = 1}^{r} c_j^{\circ} L(s,\pi_j) \begin{dcases*}
\overline{P_j^{\circ}}(\overline{k_{\ell,n}}) & if $\ell = n_1 + \cdots + n_j$ and $\pi_j = \chi^{\kappa_j} |\cdot|^{t_j}$,	\\
\overline{P_j^{\circ}}(\overline{k_{\ell,n}},\overline{k_{\ell + 1,n}}) & if $\ell = n_1 + \cdots + n_j - 1$ and $\pi_j = D_{\kappa_j} \otimes \left|\det\right|^{t_j}$,
\end{dcases*}
\end{multline*}
which is precisely $L(s,\pi) f^{\circ}(k)$ via the isobaric decomposition \eqref{eq:isobaricRS} of $L(s,\pi)$ and the identity \eqref{eq:inducedK} for $f^{\circ}(k)$.

Finally, an inspection of the proof above shows that the integral \eqref{eq:Pieriint} converges absolutely if $\Re(s) > -\Re(t_j)$ for each $j \in \{1,\ldots,r\}$ for which $n_j = 1$ and $\Re(s) > -\Re(t_j) + (\kappa_j - 1)/2$ for each $j \in \{1,\ldots,r\}$ for which $n_j = 2$.
\end{proof}

We may use the identity \eqref{eq:Pieri} to prove \hyperref[thm:GJram]{Theorem \ref*{thm:GJram}}, thereby resolving the test vector problem for archimedean Godement--Jacquet zeta integrals.

\begin{proof}[Proof of {\hyperref[thm:GJram]{Theorem \ref*{thm:GJram}}}]
From the definition \eqref{eq:GodementJacquet} of the Godement--Jacquet zeta integral and \hyperref[prop:Pieri]{Proposition \ref*{prop:Pieri}}, we have that
\begin{align*}
Z(s,\beta,\Phi) & = \int_{\GL_n(F)} \left\langle \pi(g) \cdot v^{\circ}, \widetilde{v}^{\circ} \right\rangle \Phi(g) \left|\det g\right|^{s + \frac{n - 1}{2}} \, dg	\\
& = \left\langle \int_{\GL_n(F)} \left(\pi(g) \cdot v^{\circ}\right) \Phi(g) \left|\det g\right|^{s + \frac{n - 1}{2}} \, dg, \widetilde{v}^{\circ} \right\rangle	\\
& = \left\langle L(s,\pi) v^{\circ}, \widetilde{v}^{\circ} \right\rangle	\\
& = L(s,\pi) \left\langle v^{\circ}, \widetilde{v}^{\circ} \right\rangle	\\
& = L(s,\pi).\qedhere
\end{align*}
\end{proof}

A similar calculation to that of the proof of \hyperref[prop:Pieri]{Proposition \ref*{prop:Pieri}} yields the following.

\begin{proposition}
Let $\pi$ be an induced representation of Whittaker type of $\GL_n(F)$ with newform $f^{\circ}$ in the induced model $V_{\pi}$. Define $\widetilde{f}^{\circ}(g) \coloneqq f^{\circ}(\prescript{t}{}{g}^{-1})$. Then for all $h \in \GL_n(F)$ and for $\Re(s)$ sufficiently large,
\begin{equation}
\label{eq:tildePieri}
\int_{\GL_n(F)} \widetilde{f}^{\circ}(hg) \widetilde{\Phi}(g) \left|\det g\right|^{s + \frac{n - 1}{2}} \, dg = L(s,\widetilde{\pi}) \widetilde{f}^{\circ}(h),
\end{equation}
where $\widetilde{\Phi} \in \Ss_0(\Mat_{n \times n}(F))$ is the standard Schwartz function
\begin{equation}
\label{eq:tildePhistandard}
\widetilde{\Phi}(x) \coloneqq \left(\dim \tau^{\circ}\right) \overline{P^{\circ}}(e_n \overline{x}) \exp\left(-d_F \pi \Tr\left(x \prescript{t}{}{\overline{x}}\right)\right)
\end{equation}
with $P^{\circ}$ the homogeneous harmonic polynomial associated to the newform $K$-type $\tau^{\circ}$ of $\pi$ via \eqref{eq:PcircUn} and \eqref{eq:PcircOn}.
\end{proposition}

In particular, the integral \eqref{eq:tildePieri} converges absolutely if $\Re(s) > \Re(t_j)$ for each $j \in \{1,\ldots,r\}$ for which $n_j = 1$, so that $\pi_j = \chi^{\kappa_j} |\cdot|^{t_j}$, and $\Re(s) > \Re(t_j) + (\kappa_j - 1)/2$ for each $j \in \{1,\ldots,r\}$ for which $n_j = 2$, so that $F = \R$ and $\pi_j = D_{\kappa_j} \otimes \left|\det\right|^{t_j}$.

\begin{remark}
From \eqref{eq:tildePieri}, we see that $Z(s,\widetilde{\beta},\widetilde{\Phi}) = L(s,\widetilde{\pi})$, where $\widetilde{\beta}(g) \coloneqq \beta(\prescript{t}{}{g}^{-1})$. This is in perfect accordance with the local functional equation \eqref{eq:GJfunceq} upon noting that $\widetilde{\Phi} = i^{c(\pi)} \widehat{\Phi}$ via Hecke's identity, \hyperref[lem:HeckeC]{Lemmata \ref*{lem:HeckeC}} and \ref{lem:HeckeR}.
\end{remark}

\subsection{The Newform via Godement Sections}
\label{sect:newformGodement}

Our third description of the newform in the induced model is via Godement sections. This is a recursive formula for the newform $f^{\circ}$ of $\pi = \pi_1 \boxplus \pi_2 \boxplus \cdots \boxplus \pi_r$ in terms of an integral involving the newform $f_0^{\circ}$ of $\pi_0 \coloneqq \pi_2 \boxplus \cdots \boxplus \pi_r$ and a distinguished standard Schwartz function. Unlike our earlier descriptions of $f^{\circ}$ via the Iwasawa decomposition and via convolution sections, this description via Godement sections is only valid for certain induced representations of Whittaker type; we require the parameter $t_1$ associated to $\pi_1$ to have sufficiently large real part. When we proceed to studying the Whittaker newform, we remove this condition via analytic continuation.

\subsubsection{The Case $\pi_1 = \chi^{\kappa_1} |\cdot|^{t_1}$}

We first consider the case for which $\pi = \pi_1 \boxplus \cdots \boxplus \pi_r$ with $n_1 = 1$, so that $\pi_1 = \chi^{\kappa_1} |\cdot|^{t_1}$. We begin with a simple modification of \hyperref[prop:newformidentity]{Proposition \ref*{prop:newformidentity}}.

\begin{lemma}
\label{lem:initialpropformula}
For $n \geq 2$, let $\pi = \pi_1 \boxplus \pi_2 \boxplus \cdots \boxplus \pi_r$ and $\pi_0 \coloneqq \pi_2 \boxplus \cdots \boxplus \pi_r$ be induced representations of Whittaker type of $\GL_n(F)$ and $\GL_{n - 1}(F)$ with $\pi_1 = \chi^{\kappa_1} |\cdot|^{t_1}$. Let $f_0^{\circ}$ be the newform of $\pi_0$ in the induced model $V_{\pi_0}$. Then for $v \in \Mat_{1 \times (n - 1)}(F)$, $x \in F^{\times}$, $h \in \GL_{n - 1}(F)$, and $k \in K_n$, the newform $f^{\circ}$ in the induced model $V_{\pi}$ satisfies
\begin{multline}
\label{eq:initialpropformula}
f^{\circ}\left(\begin{pmatrix} 1 & v \\ 0 & 1_{n - 1} \end{pmatrix} \begin{pmatrix} x & 0 \\ 0 & h \end{pmatrix} k\right) = \frac{c^{\circ}}{c_0^{\circ}} (\dim \tau_0^{\circ}) |x|^{\frac{n - 1}{2}} \left|\det h\right|^{-\frac{1}{2}}	\\ \times \int_{K_1} \int_{K_{n - 1}} \chi^{\kappa_1}(x k_1) |x k_1|^{t_1} f_0^{\circ}(h k_2) \overline{P_{(1,n - 1)}^{\circ}}\left(e_n k^{-1} \begin{pmatrix} k_1 & 0 \\ 0 & k_2 \end{pmatrix}\right) \, dk_2 \, dk_1.
\end{multline}
Here the constants $c^{\circ}$ and $c_0^{\circ}$ are those associated to $f^{\circ}$ and $f_0^{\circ}$ via \eqref{eq:fcirc}, while
\[P_{(1,n - 1)}^{\circ}(x) \coloneqq P_1^{\circ}(x_1) P_0^{\circ}(x_2,\ldots,x_n),\]
where $P_1^{\circ}$ and $P_0^{\circ}$ are the homogeneous harmonic polynomials associated to the newform $K$-types $\tau_1^{\circ}$ and $\tau_0^{\circ}$ of $\pi_1$ and $\pi_0$ respectively via \eqref{eq:PcircUn} and \eqref{eq:PcircOn}.
\end{lemma}

\begin{proof}
We show that \eqref{eq:initialpropformula} reproduces \hyperref[cor:newformIwasawa]{Corollary \ref*{cor:newformIwasawa}}, which determines $f^{\circ}$ completely. Writing $P_{(1,n - 1)}^{\circ} = P_1^{\circ} P_0^{\circ}$, we see that the integral over $K_1 \ni k_1$ is trivial by the homogeneity of $P_1^{\circ}$ and \eqref{eq:abstochar}. We then use the Iwasawa decomposition $h = u' a' k'$ with respect to the standard Borel subgroup, so that $u' \in \Ngp_{n - 1}(F)$, $a' = \diag(a_1',\ldots,a_{n - 1}') \in \Agp_{n - 1}(F)$, and $k' \in K_{n - 1}$, and apply \hyperref[cor:newformIwasawa]{Corollary \ref*{cor:newformIwasawa}} in order to rewrite $f_0^{\circ}(hk_2)$. The integral over $K_{n - 1} \ni k_2$ may then be evaluated via \hyperref[lem:QenkC]{Lemmata \ref*{lem:QenkC}} and \ref{lem:QenkR}. The resulting expression for $f^{\circ}$ is then precisely that given in \hyperref[cor:newformIwasawa]{Corollary \ref*{cor:newformIwasawa}} with $u = \begin{psmallmatrix} 1 & v \\ 0 & v u' \end{psmallmatrix} \in \Ngp_n(F)$, $a = \diag(x,a_1',\ldots,a_{n - 1}') \in \Agp_n(F)$, and $\begin{psmallmatrix}
 1 & 0 \\ 0 & k' \end{psmallmatrix} k \in K_n$ in place of $k$.
\end{proof}

We now use the identity \eqref{eq:initialpropformula} in conjunction with the convolution section \eqref{eq:tildePieri} in order to prove that $f^{\circ}$ may be written as a Godement section.

\begin{proposition}
For $n \geq 2$, let $\pi = \pi_1 \boxplus \pi_2 \boxplus \cdots \boxplus \pi_r$ and $\pi_0 \coloneqq \pi_2 \boxplus \cdots \boxplus \pi_r$ be induced representations of Whittaker type of $\GL_n(F)$ and $\GL_{n - 1}(F)$ with $\pi_1 = \chi^{\kappa_1} |\cdot|^{t_1}$. Let $f_0^{\circ}$ be the newform of $\pi_0$ in the induced model $V_{\pi_0}$. Let $\Phi \in \Ss_0(\Mat_{(n - 1) \times n}(F))$ be the standard Schwartz function of the form
\[\Phi(x) \coloneqq \overline{P_1^{\circ}}\left(\det \left(x \begin{pmatrix} 1_{n - 1} \\ 0 \end{pmatrix}\right)\right) \left(\dim \tau_0^{\circ}\right) \overline{P_0^{\circ}}\left(e_n \prescript{t}{}{\overline{x}}\right) \exp\left(-d_F \pi \Tr\left(x \prescript{t}{}{\overline{x}}\right)\right),\]
where $P_1^{\circ}$ and $P_0^{\circ}$ are the homogeneous harmonic polynomials associated to the newform $K$-types $\tau_1^{\circ}$ and $\tau_0^{\circ}$ of $\pi_1$ and $\pi_0$ respectively via \eqref{eq:PcircUn} and \eqref{eq:PcircOn}. Then if $\Re(t_1)$ is sufficiently large, the newform $f^{\circ}$ in the induced model $V_{\pi}$ satisfies the identity
\begin{multline}
\label{eq:fviaf0}
f^{\circ}(g) = \frac{c^{\circ} (-1)^{\kappa_1(n - 1)}}{c_0^{\circ} L\left(1 + t_1 + \frac{\|\kappa_1\|}{d_F}, \widetilde{\pi_0}\right)} \chi^{\kappa_1}(\det g) \left|\det g\right|^{t_1 + \frac{n - 1}{2}}	\\
\times \int_{\GL_{n - 1}(F)} f_0^{\circ}(h) \Phi\left(h^{-1} \begin{pmatrix} 0 & 1_{n - 1} \end{pmatrix} g\right) \chi^{-\kappa_1}(\det h) \left|\det h\right|^{-t_1 - \frac{n}{2}} \, dh,
\end{multline}
where the constants $c^{\circ}$ and $c_0^{\circ}$ are those associated to $f^{\circ}$ and $f_0^{\circ}$ via \eqref{eq:fcirc}.
\end{proposition}

In particular, the integral \eqref{eq:fviaf0} converges absolutely if $\Re(t_1) > \Re(t_j) - 1 - \|\kappa_1\|/d_F$ for each $j \in \{2,\ldots,r\}$ for which $n_j = 1$, so that $\pi_j = \chi^{\kappa_j} |\cdot|^{t_j}$, and $\Re(t_1) > \Re(t_j) + (\kappa_j - 1)/2 - 1 - \kappa_1$ for each $j \in \{2,\ldots,r\}$ for which $n_j = 2$, so that $F = \R$ and $\pi_j = D_{\kappa_j} \otimes \left|\det\right|^{t_j}$.

\begin{proof}
We take $s = 1 + t_1 + \|\kappa_1\|/d_F$ in the convolution section identity \eqref{eq:tildePieri}, so that
\[f^{\circ}(g) = \frac{1}{L\left(1 + t_1 + \frac{\|\kappa_1\|}{d_F}, \widetilde{\pi}\right)} \int_{\GL_n(F)} f^{\circ}\left(g \prescript{t}{}{g}^{\prime -1}\right) \widetilde{\Phi}(g') \left|\det g'\right|^{1 + t_1 + \frac{\|\kappa_1\|}{d_F} + \frac{n - 1}{2}} \, dg',\]
with $\widetilde{\Phi}$ as in \eqref{eq:tildePhistandard}. We make the change of variables $g' \mapsto \prescript{t}{}{g} \prescript{t}{}{g}^{\prime -1}$ and use the Iwasawa decomposition $g' = umk$ with respect to the parabolic subgroup $\Pgp(F) = \Pgp_{(1,n - 1)}(F)$, where $u = \begin{psmallmatrix} 1 & v \\ 0 & 1_{n - 1} \end{psmallmatrix} \in \Ngp_{\Pgp}(F)$ with $v \in \Mat_{1 \times (n - 1)}(F)$, $m = \begin{psmallmatrix} x & 0 \\ 0 & h \end{psmallmatrix} \in \Mgp_{\Pgp}(F)$ with $x \in F^{\times}$ and $h \in \GL_{n - 1}(F)$, and $k \in K_n$; the Haar measure is $dg' = |x|^{1 - n} \left|\det h\right| \, dv \, d^{\times}x \, dh \, dk$. We may now insert the identity \eqref{eq:initialpropformula} for $f^{\circ}$. Next, we make the change of variables $x \mapsto x k_1^{-1}$, $h \mapsto h k_2^{-1}$, and $k \mapsto \begin{psmallmatrix} k_1 & 0 \\ 0 & k_2 \end{psmallmatrix} k$, so that the integrals over $K_1 \ni k_1$ and $K_{n - 1} \ni k_2$ are trivial; subsequently, the integral over $K_n \ni k$ may be evaluated via \hyperref[lem:QenkC]{Lemmata \ref*{lem:QenkC}} and \ref{lem:QenkR} after inserting the definition \eqref{eq:tildePhistandard} of $\widetilde{\Phi}$. After making the change of variables $v \mapsto -xv$ and $x \mapsto x^{-1}$, we write $v' \coloneqq \begin{pmatrix} x & v \end{pmatrix} \in \Mat_{1 \times n}(F)$, so that $dv' = \zeta_F(1)^{-1} |x| \, d^{\times}x \, dv$, and make the change of variables $v' \mapsto v' g^{-1}$. Finally, we use \eqref{eq:abstochar} in conjunction with the homogeneity of $P_1^{\circ}$ and the fact that
\[L\left(1 + t_1 + \frac{\|\kappa_1\|}{d_F}, \widetilde{\pi}\right) = \zeta_F\left(1 + \frac{2\|\kappa_1\|}{d_F}\right) L\left(1 + t_1 + \frac{\|\kappa_1\|}{d_F}, \widetilde{\pi_0}\right)\]
via \eqref{eq:isobaricRS}, \eqref{eq:L(s,tildepi)C}, and \eqref{eq:L(s,tildepi)R1}. We arrive at the identity
\begin{multline}
\label{eq:lastline}
f^{\circ}(g) = \frac{c^{\circ}}{c_0^{\circ} L\left(1 + t_1 + \frac{\|\kappa_1\|}{d_F}, \widetilde{\pi_0}\right)} \chi^{\kappa_1}(\det g) \left|\det g\right|^{t_1 + \frac{n - 1}{2}} \int_{\GL_{n - 1}(F)} f_0^{\circ}(h) \chi^{-\kappa_1}(\det h) \left|\det h\right|^{-t_1 - \frac{n}{2}}	\\
\times \left(\dim \tau_0^{\circ}\right) \overline{P_0^{\circ}}\left(e_n \prescript{t}{}{\overline{g}} \begin{pmatrix} 0 \\ 1_{n - 1} \end{pmatrix} \prescript{t}{}{\overline{h}}^{-1}\right) \exp\left(-d_F \pi \Tr\left(h^{-1} \begin{pmatrix} 0 & 1_{n - 1} \end{pmatrix} g \prescript{t}{}{\overline{g}} \begin{pmatrix} 0 \\ 1_{n - 1} \end{pmatrix} \prescript{t}{}{\overline{h}}^{-1}\right)\right)	\\
\times \frac{\zeta_F(1)}{\zeta_F\left(1 + \frac{2\|\kappa_1\|}{d_F}\right)} \int\limits_{\Mat_{1 \times n}(F)} \overline{P_1^{\circ}}\left((\det h^{-1}) v' \adj(g) \prescript{t}{}{e}_1 e_n \prescript{t}{}{\overline{v'}}\right) \exp\left(-d_F \pi v' \prescript{t}{}{\overline{v'}}\right) \, dv' \, dh,
\end{multline}
having recalled that the adjugate of $g$ is $\adj(g) \coloneqq (\det g) g^{-1}$. To evaluate the last line of \eqref{eq:lastline}, we expand $\overline{P_1^{\circ}}$ as a polynomial in $v_1',\ldots,v_n'$ via the multinomial theorem, yielding
\begin{multline*}
\frac{\zeta_F(1)}{\zeta_F\left(1 + \frac{2\|\kappa_1\|}{d_F}\right)} (\overline{\det h^{-1}})^{\max\{\kappa_1,0\}} (\det h^{-1})^{-\min\{\kappa_1,0\}} \sum_{\nu_1 + \cdots + \nu_n = \|\kappa_1\|} \binom{\|\kappa_1\|}{\nu_1, \ldots, \nu_n}	\\
\times \prod_{i = 1}^{n} (\overline{\adj(g)_{i,1}})^{\max\{\sgn(\kappa_1) \nu_i,0\}} (\adj(g)_{i,1})^{-\min\{\sgn(\kappa_1) \nu_i,0\}}	\\
\times \prod_{i = 1}^{n - 1} \int_F \overline{v_i'}^{\max\{\sgn(\kappa_1) \nu_i,0\}} {v_i'}^{- \min\{\sgn(\kappa_1) \nu_i,0\}} \exp\left(-d_F \pi v_i' \overline{v_i'}\right) \, dv_i'	\\
\times \int_F \overline{v_n'}^{\max\{\sgn(\kappa_1)\nu_n,0\} - \min\{\kappa_1,0\}} {v_n'}^{-\min\{\sgn(\kappa_1)\nu_n,0\} + \max\{\kappa_1,0\}} \exp\left(-d_F \pi v_n' \overline{v_n'}\right) \, dv_n'.
\end{multline*}
Via Hecke's identity, \hyperref[lem:HeckeC]{Lemmata \ref*{lem:HeckeC}} and \ref{lem:HeckeR}, the integral over $F \ni v_n'$ vanishes unless $\nu_1 = \cdots = \nu_{n - 1} = 0$ and $\nu_n = \|\kappa_1\|$, in which case the integral over $F \ni v_i'$ for $i \in \{1,\ldots,n - 1\}$ is equal to $1$. Via \eqref{eq:abstochar} and the fact that $\zeta_F(1) |v_n'|^{-1} \, dv_n' = d^{\times}v_n'$, we are left with
\[\frac{\overline{P_1^{\circ}}\left((\det h^{-1}) \adj(g)_{n,1}\right)}{\zeta_F\left(1 + \frac{2\|\kappa_1\|}{d_F}\right)} \int_{F^{\times}} |v_n'|^{1 + \frac{2\|\kappa_1\|}{d_F}} \exp\left(-d_F \pi v_n' \overline{v_n'}\right) \, d^{\times}v_n' = \overline{P_1^{\circ}}\left((\det h^{-1}) \adj(g)_{n,1}\right)\]
by \eqref{eq:zetaCint} and \eqref{eq:zetaRint}. Since
\[\adj(g)_{n,1} = (-1)^{n - 1} \det \left(\begin{pmatrix} 0 & 1_{n - 1}\end{pmatrix} g \begin{pmatrix} 1_{n - 1} \\ 0 \end{pmatrix}\right),\]
the result then follows from \eqref{eq:lastline} and the homogeneity of $P_1^{\circ}$.
\end{proof}

\subsubsection{The Case $\pi_1 = D_{\kappa_1} \otimes \left|\det\right|^{t_1}$}

Next, we give a description of the newform in the induced model when $F = \R$, $n_1 = 2$, and $\pi_1 = D_{\kappa_1} \otimes \left|\det\right|^{t_1}$ is an essentially discrete series representation. We do this first when $n = 2$, so that $\pi = \pi_1$.

\begin{proposition}
Let $\pi = D_{\kappa} \otimes \left|\det\right|^{t}$ be an essentially discrete series representation of $\GL_2(\R)$. Let $\Phi \in \Ss_0(\Mat_{1 \times 2}(\R))$ be the standard Schwartz function of the form
\begin{equation}
\label{eq:PhiHermite}
\Phi(x) \coloneqq \int_{\R} \overline{P^{\circ}}\left(v, e_2 \prescript{t}{}{x}\right) \exp(-\pi v^2) \, dv \exp\left(-\pi x \prescript{t}{}{x}\right),
\end{equation}
where $P^{\circ}$ is the homogeneous harmonic polynomial associated to the newform $K$-type $\tau^{\circ} = \tau_{(\kappa,0)}$ of $\pi$ via \eqref{eq:PcircOn}. Then the canonically normalised newform $f^{\circ}$ in the induced model $V_{\pi}$ satisfies the identity
\begin{equation}
\label{eq:fviaf0ds}
f^{\circ}(g) = i^{\kappa} \left|\det g\right|^{t + \frac{\kappa}{2}} \int_{\R^{\times}} a_2^{-\kappa} \Phi\left(a_2^{-1} \begin{pmatrix} 0 & 1 \end{pmatrix} g\right) \, d^{\times}a_2.
\end{equation}
\end{proposition}

\begin{remark}
The integral over $\R \ni v$ in \eqref{eq:PhiHermite} may be expressed in terms of Hermite polynomials, though we do not make direct use of this fact.
\end{remark}

\begin{proof}
We take $s = 1 + t + (\kappa - 1)/2$ in the convolution section identity \eqref{eq:tildePieri}, so that
\[f^{\circ}(g) = \frac{1}{L\left(1 + t + \frac{\kappa - 1}{2}, \widetilde{\pi}\right)} \int_{\GL_2(\R)} f^{\circ}\left(g \prescript{t}{}{g}^{\prime -1}\right) \widetilde{\Phi}(g^{\prime}) \left|\det g'\right|^{1 + t + \frac{\kappa}{2}} \, dg',\]
with $\widetilde{\Phi}$ as in \eqref{eq:tildePhistandard}. We make the change of variables $g' \mapsto \prescript{t}{}{g} \prescript{t}{}{g}^{\prime -1}$ and use the Iwasawa decomposition $g' = uak$ with respect to the standard Borel subgroup, where $u = \begin{psmallmatrix} 1 & v \\ 0 & 1 \end{psmallmatrix} \in \Ngp_2(\R)$ with $v \in \R$, $a = \diag(a_1,a_2) \in \Agp_2(\R)$ with $a_1,a_2 \in \R^{\times}$, and $k \in \Ogp(2)$; the Haar measure is $dg' = |a_1|^{-1} |a_2| \, dv \, d^{\times}a_1 \, d^{\times}a_2 \, dk$. We may now insert the identity \eqref{eq:fcircD} for $f^{\circ}$, at which point \hyperref[lem:QenkR]{Lemma \ref*{lem:QenkR}} allows us to evaluate the integral over $\Ogp(2) \ni k$. Next, we make the change of variables $a_1 \mapsto a_1^{-1}$, $v \mapsto -a_1^{-1} v$, and $\begin{pmatrix} a_1 & v \end{pmatrix} \mapsto \begin{pmatrix} a_1 & v \end{pmatrix} g^{-1}$, noting that $d^{\times}a_1 = |a_1|^{-1} \, da_1$ as $\zeta_{\R}(1) = 1$. Finally, we use the fact that
\[L\left(1 + t + \frac{\kappa - 1}{2}, \widetilde{\pi}\right) = \zeta_{\R}(\kappa) \zeta_{\R}(\kappa + 1)\]
via \eqref{eq:L(s,tildepi)R2}. In this way, we find that
\begin{multline*}
f^{\circ}(g) = \frac{c^{\circ}}{\zeta_{\R}(\kappa) \zeta_{\R}(\kappa + 1)} \left|\det g\right|^{t + \frac{\kappa}{2}} \int_{\R^{\times}} a_2^{-\kappa} \int_{\R} \exp(-\pi a_1^2)	\\
\times \int_{\R} \overline{P^{\circ}}\left(v, e_2 \prescript{t}{}{g} \begin{pmatrix} 0 \\ 1\end{pmatrix} a_2^{-1}\right) \exp(-\pi v^2) \exp\left(-\pi a_2^{-1} \begin{pmatrix} 0 & 1 \end{pmatrix} g \prescript{t}{}{g} \begin{pmatrix} 0 \\ 1\end{pmatrix} a_2^{-1}\right) \, dv \, da_1 \, d^{\times}a_2.
\end{multline*}
It remains to note that the integral over $\R \ni a_1$ is equal to $1$ and to recall the definition \eqref{eq:cannormsq} of the normalising constant $c^{\circ}$.
\end{proof}

Finally, we consider the more general case for which $\pi = \pi_1 \boxplus \cdots \boxplus \pi_r$ with $n_1 = 2$ and $\pi_1 = D_{\kappa_1} \otimes \left|\det\right|^{t_1}$. We first require a simple modification of \hyperref[prop:newformidentity]{Proposition \ref*{prop:newformidentity}} akin to \hyperref[lem:initialpropformula]{Lemma \ref*{lem:initialpropformula}}.

\begin{lemma}
For $n \geq 3$, let $\pi = \pi_1 \boxplus \pi_2 \boxplus \cdots \boxplus \pi_r$ and $\pi_0 \coloneqq \pi_2 \boxplus \cdots \boxplus \pi_r$ be induced representations of Whittaker type of $\GL_n(\R)$ and $\GL_{n - 2}(\R)$ with $\pi_1 = D_{\kappa_1} \otimes \left|\det\right|^{t_1}$. Let $f_0^{\circ}$ be the newform of $\pi_0$ in the induced model $V_{\pi_0}$. Then for $v_1,v_2 \in \Mat_{1 \times (n - 2)}(\R)$, $v_3 \in \R$, $a_1,a_2 \in \R^{\times}$, $h \in \GL_{n - 2}(\R)$, and $k \in \Ogp(n)$, the newform $f^{\circ}$ in the induced model $V_{\pi}$ satisfies
\begin{multline}
\label{eq:initialpropformula2}
f^{\circ}\left(\begin{pmatrix} 1 & v_1 & v_2 \\ 0 & 1 & v_3 \\ 0 & 0 & 1_{n - 2} \end{pmatrix} \begin{pmatrix} a_1 & 0 & 0 \\ 0 & a_2 & 0 \\ 0 & 0 & h \end{pmatrix} k\right) = \frac{c^{\circ} c_1^{\circ}}{c_0^{\circ}} |a_1|^{t_1 + \frac{\kappa_1 - 1}{2} + \frac{n - 1}{2}} \chi^{\kappa_1}(a_2) |a_2|^{t_1 - \frac{\kappa_1 - 1}{2} + \frac{n - 3}{2}} \left|\det h\right|^{-1}	\\
\times \dim \tau_0^{\circ} \int_{\Ogp(n - 2)} f_0^{\circ}(hk_2) \overline{P_{(2,n - 2)}^{\circ}}\left(e_n k^{-1} \begin{pmatrix} 1_2 & 0 \\ 0 & k_2 \end{pmatrix}\right) \, dk_2.
\end{multline}
Here the constants $c^{\circ}$, $c_1^{\circ}$, and $c_0^{\circ}$ are those associated to $f^{\circ}$, $f_1^{\circ}$, and $f_0^{\circ}$ via \eqref{eq:fcirc} and \eqref{eq:fcircD}, while
\[P_{(2,n - 2)}^{\circ}(x) \coloneqq P_1^{\circ}(x_1,x_2) P_0^{\circ}(x_3,\ldots,x_n),\]
where $P_1^{\circ}$ and $P_0^{\circ}$ are the homogeneous harmonic polynomials associated to the newform $K$-types $\tau_1^{\circ}$ and $\tau_0^{\circ}$ of $\pi_1$ and $\pi_0$ respectively via \eqref{eq:PcircOn}.
\end{lemma}

\begin{proof}
The proof is essentially identical to that of \hyperref[lem:initialpropformula]{Lemma \ref*{lem:initialpropformula}}.
\end{proof}

We now use the identity \eqref{eq:initialpropformula2} in conjunction with the convolution section \eqref{eq:tildePieri} in order to prove that $f^{\circ}$ may be written as a Godement section.

\begin{proposition}
\label{prop:Godementdiscrete}
For $n \geq 3$, let $\pi = \pi_1 \boxplus \pi_2 \boxplus \cdots \boxplus \pi_r$ and $\pi_0 \coloneqq \pi_2 \boxplus \cdots \boxplus \pi_r$ be induced representations of Whittaker type of $\GL_n(\R)$ and $\GL_{n - 2}(\R)$ with $\pi_1 = D_{\kappa_1} \otimes \left|\det\right|^{t_1}$. Let $f_0^{\circ}$ be the newform of $\pi_0$ in the induced model $V_{\pi_0}$. Let $\Phi \in \Ss_0(\Mat_{(n - 1) \times n}(\R))$ be the standard Schwartz function of the form
\[\Phi(x) \coloneqq \int_{\R} \overline{P_1^{\circ}} \left(v, e_n \prescript{t}{}{x} \prescript{t}{}{e}_1\right) \exp(-\pi v^2) \, dv (\dim \tau_0^{\circ}) \overline{P_0^{\circ}}\left(e_n \prescript{t}{}{x} \begin{pmatrix} 0 \\ 1_{n - 2} \end{pmatrix}\right) \exp\left(-\pi \Tr\left(x \prescript{t}{}{x}\right)\right),\]
where $P_1^{\circ}$ and $P_0^{\circ}$ are the homogeneous harmonic polynomials associated to the newform $K$-types $\tau_1^{\circ}$ and $\tau_0^{\circ}$ of $\pi_1$ and $\pi_0$ respectively via \eqref{eq:PcircOn}. Then if $\Re(t_1)$ is sufficiently large, the newform $f^{\circ}$ in the induced model $V_{\pi}$ satisfies the identity
\begin{multline}
\label{eq:fviaf02}
f^{\circ}(g) = \frac{c^{\circ} i^{\kappa_1}}{c_0^{\circ} L\left(1 + t_1 + \frac{\kappa_1 - 1}{2},\widetilde{\pi_0}\right)} \left|\det g\right|^{t_1 + \frac{\kappa_1 - 1}{2} + \frac{n - 1}{2}} \int_{\GL_{n - 2}(\R)} f_0^{\circ}(h) \left|\det h\right|^{-t_1 - \frac{\kappa_1 - 1}{2} - \frac{n - 1}{2}}	\\
\times \int_{\R^{\times}} a_2^{-\kappa_1} \int\limits_{\Mat_{1 \times (n - 2)}(\R)} \Phi\left(\begin{pmatrix} a_2^{-1} & v_3 \\ 0 & h^{-1} \end{pmatrix} \begin{pmatrix} 0 & 1_{n - 1}\end{pmatrix} g\right) \, dv_3 \, d^{\times}a_2 \, dh,
\end{multline}
where the constants $c^{\circ}$ and $c_0^{\circ}$ are those associated to $f^{\circ}$ and $f_0^{\circ}$ via \eqref{eq:fcirc}.
\end{proposition}

In particular, the integral \eqref{eq:fviaf02} converges absolutely if $\Re(t_1) > \Re(t_j) - 1 - (\kappa_1 - 1)/2$ for each $j \in \{2,\ldots,r\}$ for which $n_j = 1$, so that $\pi_j = \chi^{\kappa_j} |\cdot|^{t_j}$, and $\Re(t_1) > \Re(t_j) + (\kappa_j - 1)/2 - 1 - (\kappa_1 - 1)/2$ for each $j \in \{2,\ldots,r\}$ for which $n_j = 2$, so that $\pi_j = D_{\kappa_j} \otimes \left|\det\right|^{t_j}$.

\begin{proof}
We take $s = 1 + t_1 + (\kappa_1 - 1)/2$ in the convolution section identity \eqref{eq:tildePieri}, so that
\[f^{\circ}(g) = \frac{1}{L\left(1 + t_1 + \frac{\kappa_1 - 1}{2},\widetilde{\pi}\right)} \int_{\GL_n(\R)} f^{\circ}\left(g \prescript{t}{}{g}^{\prime -1}\right) \widetilde{\Phi}(g') \left|\det g'\right|^{1 + t_1 + \frac{\kappa_1 - 1}{2} + \frac{n - 1}{2}} \, dg',\]
with $\widetilde{\Phi}$ as in \eqref{eq:tildePhistandard}. We make the change of variables $g' \mapsto \prescript{t}{}{g} \prescript{t}{}{g}^{\prime -1}$, then use the Iwasawa decomposition $g' = umk$ with respect to the parabolic subgroup $\Pgp(\R) = \Pgp_{(1,1,n - 2)}(\R)$, where $u = \begin{psmallmatrix} 1 & v_1 & v_2 \\ 0 & 1 & v_3 \\ 0 & 0 & 1_{n - 2} \end{psmallmatrix} \in \Ngp_{\Pgp}(\R)$ with $v_1 \in \R$, $v_2,v_3 \in \Mat_{1 \times (n - 2)}(\R)$, $m = \begin{psmallmatrix} a_1 & 0 & 0 \\ 0 & a_2 & 0 \\ 0 & 0 & h \end{psmallmatrix} \in \Mgp_{\Pgp}(\R)$ with $a_1,a_2 \in \R^{\times}$ and $h \in \GL_{n - 2}(\R)$, and $k \in \Ogp(n)$; the Haar measure is
\[dg' = |a_1|^{1 - n} |a_2|^{3 - n} \left|\det h\right|^2 \, dv_1 \, dv_2 \, dv_3 \, d^{\times}a_1 \, d^{\times}a_2 \, dh \, dk.\]
We may now insert the identity \eqref{eq:initialpropformula2} for $f^{\circ}$. Next, we make the change of variables $h \mapsto h k_2^{-1}$ and $k' \mapsto \begin{psmallmatrix} 1_2 & 0 \\ 0 & k_2\end{psmallmatrix} k'$, so that the integral over $\Ogp(n - 2) \ni k_2$ is trivial; subsequently, the integral over $\Ogp(n) \ni k$ may be evaluated via \hyperref[lem:QenkR]{Lemma \ref*{lem:QenkR}}. Next, we make the change of variables $v_1 \mapsto -a_1 v_1$, $v_2 \mapsto -a_1 v_2 - v_1 v_3$, $v_3 \mapsto -a_2 v_3$, $a_1 \mapsto a_1^{-1}$, and $\begin{pmatrix} a_1 & v_1 & v_2 \end{pmatrix} \mapsto \begin{pmatrix} a_1 & v_1 & v_2 \end{pmatrix} g^{-1}$, noting that $d^{\times}a_1 = |a_1|^{-1} \, da_1$ as $\zeta_{\R}(1) = 1$. Finally, we use the fact that
\[L\left(1 + t_1 + \frac{\kappa_1 - 1}{2},\widetilde{\pi}\right) = \zeta_{\R}(\kappa_1) \zeta_{\R}(\kappa_1 + 1) L\left(1 + t_1 + \frac{\kappa_1 - 1}{2},\widetilde{\pi_0}\right)\]
via \eqref{eq:isobaricRS} and \eqref{eq:L(s,tildepi)R2}. In this way, we find that
\begin{multline*}
f^{\circ}(g) = \frac{c^{\circ} c_1^{\circ}}{c_0^{\circ} \zeta_{\R}(\kappa_1) \zeta_{\R}(\kappa_1 + 1) L\left(1 + t_1 + \frac{\kappa_1 - 1}{2},\widetilde{\pi_0}\right)} \left|\det g\right|^{t_1 + \frac{\kappa_1 - 1}{2} + \frac{n - 1}{2}}	\\
\times \int_{\GL_{n - 2}(\R)} f_0^{\circ}(h) \left|\det h\right|^{-t_1 - \frac{\kappa_1 - 1}{2} - \frac{n - 1}{2}}	\\
\times (\dim \tau_0^{\circ}) \overline{P_0^{\circ}}\left(e_n \prescript{t}{}{g} \begin{pmatrix} 0 \\ 1_{n - 2} \end{pmatrix} \prescript{t}{}{h}^{-1} \right) \exp\left(-\pi \Tr\left(h^{-1} \begin{pmatrix} 0 & 1_{n - 2}\end{pmatrix} g \prescript{t}{}{g} \begin{pmatrix} 0 \\ 1_{n - 2} \end{pmatrix} \prescript{t}{}{h}^{-1}\right)\right)	\\
\times \int_{\R^{\times}} a_2^{-\kappa_1} \int_{\R} \exp(-\pi a_1^2) \int\limits_{\Mat_{1 \times (n - 2)}(\R)} \exp\left(-\pi \Tr\left(\begin{pmatrix} 0 & a_2^{-1} & v_3 \end{pmatrix} g \prescript{t}{}{g} \begin{pmatrix} 0 \\ a_2^{-1} \\ \prescript{t}{}{v_3} \end{pmatrix}\right)\right)	\\
\times \int\limits_{\Mat_{1 \times (n - 2)}(\R)} \overline{P_1^{\circ}}\left(e_{n - 1} v_2, e_n \prescript{t}{}{g} \begin{pmatrix} 0 \\ a_2^{-1} \\ \prescript{t}{}{v_3} \end{pmatrix}\right) \exp\left(-\pi v_2 \prescript{t}{}{v}_2\right) \int_{\R} \exp(-\pi v_1^2) \, dv_1 \, dv_2 \, dv_3 \, d^{\times}a_1 \, d^{\times}a_2 \, dh.
\end{multline*}
It remains to note that the integrals over $\R \ni a_1$, $\R \ni v_1$, and $\R \ni v_{2,i}$ for $i \in \{1,\ldots,n - 3\}$ are trivial and to recall the definition \eqref{eq:cannormsq} of the normalising constant $c_1^{\circ}$.
\end{proof}

In order to prove a recursive formula for the newform in the Whittaker model, we require a similar identity to \eqref{eq:fviaf02} for a slightly modified induced representation of Whittaker type.

\begin{lemma}
For $n \geq 3$, let $\pi^{\ast} = \pi_1^{\ast} \boxplus \pi_2^{\ast} \boxplus \pi_2 \boxplus \cdots \boxplus \pi_r$ and $\pi_0 \coloneqq \pi_2 \boxplus \cdots \boxplus \pi_r$ be induced representations of Whittaker type of $\GL_n(\R)$ and $\GL_{n - 2}(\R)$ with $\pi_1^{\ast} = |\cdot|^{t_1^{\ast}}$ and $\pi_2^{\ast} = |\cdot|^{t_2^{\ast}}$. Let $f_0^{\circ}$ be the newform of $\pi_0$ in the induced model $V_{\pi_0}$. Let $\Phi \in \Ss_0(\Mat_{(n - 1) \times n}(\R))$ be the standard Schwartz function of the form
\[\Phi^{\ast}(x) \coloneqq (\dim \tau_0^{\circ}) \overline{P_0^{\circ}}\left(e_n \prescript{t}{}{x} \begin{pmatrix} 0 \\ 1_{n - 2} \end{pmatrix}\right) \exp\left(-\pi \Tr\left(x \prescript{t}{}{x}\right)\right),\]
where $P_0^{\circ}$ is the homogeneous harmonic polynomial associated to the newform $K$-type $\tau_0^{\circ}$ of $\pi_0$ via \eqref{eq:PcircOn}. Then if $\Re(t_1^{\ast})$ is sufficiently large, the newform $f^{\ast\circ}$ in the induced model $V_{\pi^{\ast}}$ satisfies the identity
\begin{multline}
\label{eq:fastviaf02}
f^{\ast\circ}(g) = \frac{c^{\ast\circ}}{c_0^{\circ} \zeta_{\R}(1 + t_1^{\ast} - t_2^{\ast}) L\left(1 + t_1^{\ast},\widetilde{\pi_0}\right)} \left|\det g\right|^{t_1^{\ast} + \frac{n - 1}{2}} \int_{\GL_{n - 2}(\R)} f_0^{\circ}(h) \left|\det h\right|^{-t_1^{\ast} - \frac{n - 1}{2}}	\\
\times \int_{\R^{\times}} |a_2|^{-1 - t_1^{\ast} + t_2^{\ast}} \int\limits_{\Mat_{1 \times (n - 2)}(\R)} \Phi\left(\begin{pmatrix} a_2^{-1} & v_3 \\ 0 & h^{-1} \end{pmatrix} \begin{pmatrix} 0 & 1_{n - 1}\end{pmatrix} g\right) \, dv_3 \, d^{\times}a_2 \, dh,
\end{multline}
where the constants $c^{\ast\circ}$ and $c_0^{\circ}$ are those associated to $f^{\ast\circ}$ and $f_0^{\circ}$ via \eqref{eq:fcirc}.
\end{lemma}

In particular, the integral \eqref{eq:fastviaf02} converges absolutely if $\Re(t_1^{\ast}) > \Re(t_2^{\ast}) - 1$, $\Re(t_1^{\ast}) > \Re(t_j) - 1$ for each $j \in \{2,\ldots,r\}$ for which $n_j = 1$, so that $\pi_j = \chi^{\kappa_j} |\cdot|^{t_j}$, and $\Re(t_1^{\ast}) > \Re(t_j) + (\kappa_j - 1)/2 - 1$ for each $j \in \{2,\ldots,r\}$ for which $n_j = 2$, so that $\pi_j = D_{\kappa_j} \otimes \left|\det\right|^{t_j}$.

\begin{proof}
This follows via the same method as the proof of \hyperref[prop:Godementdiscrete]{Proposition \ref*{prop:Godementdiscrete}}.
\end{proof}

\section{The Newform in the Whittaker Model}
\label{sect:newformWhittaker}

\subsection{The Jacquet Integral}
\label{sect:Jacquetint}

Let $\pi = \pi_1 \boxplus \cdots \boxplus \pi_r$ be an induced representation of Whittaker type of $\GL_n(F)$, so that each $\pi_j$ is of the form $\chi^{\kappa_j} |\cdot|^{t_j}$ or $D_{\kappa_j} \otimes \left|\det\right|^{t_j}$. Given an element $f$ of the induced model $V_{\pi}$ of $\pi$, we define the Jacquet integral
\begin{equation}
\label{eq:Jacquetint}
W(g) \coloneqq \int_{\Ngp_n(F)} f(w_n ug) \overline{\psi_n}(u) \, du.
\end{equation}
This integral converges absolutely if $\Re(t_1) > \cdots > \Re(t_r)$ and defines a Whittaker function $W = W_f \in \WW(\pi,\psi)$; that is, as a function of $f \in V_{\pi}$, $\Lambda(f) \coloneqq W_f(1_n)$ defines a Whittaker functional, which is therefore unique up to scalar multiplication.

Wallach \cite{Wal92} has shown that the Jacquet integral gives a Whittaker functional for \emph{all} induced representations of Whittaker type, and not just those for which $\Re(t_1) > \cdots > \Re(t_r)$, via analytic continuation in the following way. Write $\pi = \pi_{t_1,\ldots,t_r}$ for such a representation, and let $V_{\pi} = V_{\pi_{t_1,\ldots,t_r}}$ denote its induced model. Fixing each $\chi^{\kappa_j}$ and $D_{\kappa_j}$ but regarding $t_j$ as a complex variable, we may view the space $V_{\pi_{t_1,\ldots,t_r}}$ as a holomorphic fibre bundle. A section $f_{t_1,\ldots,t_r}(g;m')$ is a map from $\GL_n(F) \times \Mgp_{\Pgp}(F) \times \C^r$ to $\C$ such that $f_{t_1,\ldots,t_r}(\cdot;\cdot)$ is an element of $V_{\pi_{t_1,\ldots,t_r}}$ for each fixed $(t_1,\ldots,t_r) \in \C^r$; a standard section (or flat section) is a section for which $f_{t_1,\ldots,t_r}(k;1_n)$ is independent of $(t_1,\ldots,t_r) \in \C^r$ for all $k \in K$. From \cite[Theorem 15.4.1]{Wal92}, the Jacquet integral \eqref{eq:Jacquetint} evaluated on a standard section extends holomorphically as a function of $(t_1,\ldots,t_r) \in \C^r$ with $\Re(t_1) > \cdots > \Re(t_r)$ to all of $\C^r$, and hence via analytic continuation defines an equivariant map from $V_{\pi_{t_1,\ldots,t_r}}$ to $\WW(\pi_{t_1,\ldots,t_r},\psi)$.

From this, we see that the newform $f^{\circ}$ in the induced model $V_{\pi}$ defined via the Iwasawa decomposition \eqref{eq:fcirc} gives a standard section of newforms $f_{t_1,\ldots,t_r}^{\circ}$ provided that we choose the normalising constant $c^{\circ}$ to be independent of $(t_1,\ldots,t_r) \in \C^r$. The corresponding Whittaker function is then given via the analytic continuation of the Jacquet integral \eqref{eq:Jacquetint}. Furthermore, we may choose the normalising constant $c^{\circ}$ to be dependent on $(t_1,\ldots,t_r) \in \C^r$ and still obtain the corresponding Whittaker function via the analytic continuation of the Jacquet integral so long as $c^{\circ}$ is holomorphic as a function of $(t_1,\ldots,t_r) \in \C^r$.

With this in mind, we may now define the canonically normalised newform in the induced and Whittaker models.

\begin{definition}
\label{def:cannorm}
Let $\pi = \pi_1 \boxplus \cdots \boxplus \pi_r$ be an induced representation of Langlands type of $\GL_n(F)$. The canonically normalised newform $f^{\circ}$ in the induced model $V_{\pi}$ is defined via \eqref{eq:fcircchi} and \eqref{eq:fcircD} with normalising constant \eqref{eq:cannormsq} if $r = 1$, while for $r \geq 2$, it is defined via \eqref{eq:fcirc} with normalising constant
\[c^{\circ} \coloneqq \prod_{j = 1}^{r - 1} \prod_{\ell = j + 1}^{r} i^{c(\pi_{\ell})} \begin{dcases*}
L\left(1 + t_j + \frac{\|\kappa_j\|}{d_F}, \widetilde{\pi_{\ell}}\right) & if $\pi_j = \chi^{\kappa_j} |\cdot|^{t_j}$,	\\
L\left(1 + t_j + \frac{\kappa_j - 1}{2}, \widetilde{\pi_{\ell}}\right) L\left(1 + t_j + \frac{\kappa_j + 1}{2}, \widetilde{\pi_{\ell}}\right) & if $\pi_j = D_{\kappa_j} \otimes \left|\det\right|^{t_j}$.
\end{dcases*}\]
The canonically normalised newform $W^{\circ}$ in the Whittaker model $\WW(\pi,\psi)$ is given by the analytic continuation of the Jacquet integral \eqref{eq:Jacquetint} of the canonically normalised newform $f^{\circ}$. We call $W^{\circ}$ the Whittaker newform.
\end{definition}

\begin{remark}
When $\pi$ is spherical, some authors refer to the canonically normalised Whittaker function as the completed Whittaker function; see, for example, \cite[Section 2.1]{BHM20}. We follow the nomenclature of \cite[Section 8]{GMW21}.
\end{remark}

Recalling the identities \eqref{eq:L(s,tildepi)C}, \eqref{eq:L(s,tildepi)R1}, and \eqref{eq:L(s,tildepi)R2} relating $L$-functions to zeta functions, we observe that the normalising constant $c^{\circ}$ is well-defined since $\zeta_F(s)$ is holomorphic for $\Re(s) > 0$ and $\pi$ being an induced representation of Langlands type means that $\Re(t_1) \geq \cdots \geq \Re(t_r)$. We also note that if $\pi = \pi_1 \boxplus \pi_2 \boxplus \cdots \boxplus \pi_r$ and $\pi_0 \coloneqq \pi_2 \boxplus \cdots \boxplus \pi_r$, then the associated normalising constants satisfy the relation
\begin{equation}
\label{eq:ccirctoc0circ}
c^{\circ} = \begin{dcases*}
i^{c(\pi_0)} L\left(1 + t_1 + \frac{\|\kappa_1\|}{d_F}, \widetilde{\pi_0}\right) c_0^{\circ} & if $\pi_1 = \chi^{\kappa_1} |\cdot|^{t_1}$,	\\
i^{c(\pi_0)} L\left(1 + t_1 + \frac{\kappa_1 - 1}{2}, \widetilde{\pi_0}\right) L\left(1 + t_1 + \frac{\kappa_1 + 1}{2}, \widetilde{\pi_0}\right) c_0^{\circ} & if $\pi_1 = D_{\kappa_1} \otimes \left|\det\right|^{t_1}$.
\end{dcases*}
\end{equation}
This is a consequence of \hyperref[def:cannorm]{Definition \ref*{def:cannorm}}, \hyperref[thm:additiveconductor]{Theorem \ref*{thm:additiveconductor}}, and \eqref{eq:isobaricRS}.

\begin{remark}
While we do not prove this, our methods below can be extended to show that not only is the Whittaker newform well-defined for induced representations of Langlands type, it is also well-defined for induced representations of Whittaker type, including those for which $c^{\circ}$ is \emph{not} well-defined (note that in these cases, the Whittaker model is a model for a quotient of $\pi$ rather than for $\pi$ itself). Furthermore, one can show that the Whittaker newform of $\pi = \pi_1 \boxplus \cdots \boxplus \pi_r$ remains unchanged when $\pi$ is replaced by $\pi_{\sigma(1)} \boxplus \cdots \boxplus \pi_{\sigma(r)}$ for any permutation $\sigma \in S_r$ (cf.\ \hyperref[rem:permutation]{Remark \ref*{rem:permutation}}). When $\pi$ is a spherical induced representation of Whittaker type, these claims follow from the work of Jacquet \cite[Th\'{e}or\`{e}me (8.6)]{Jac67}.
\end{remark}

\subsection{The Newform via Convolution Sections}

We now use the convolution section identity \eqref{eq:Pieri} for the newform in the induced model together with the Jacquet integral \eqref{eq:Jacquetint} in order to give a convolution section identity for the Whittaker newform. This is a recursive formula for $W^{\circ}$ as an integral over $\GL_n(F)$ involving $W^{\circ}$ and the distinguished standard Schwartz function $\Phi$ given by \eqref{eq:Phistandard}.

\begin{lemma}
Let $\pi$ be an induced representation of Langlands type of $\GL_n(F)$ with Whittaker newform $W^{\circ} \in \WW(\pi,\psi)$. Then for all $h \in \GL_n(F)$ and for $\Re(s)$ sufficiently large,
\begin{equation}
\label{eq:WhittakerPieri}
\int_{\GL_n(F)} W^{\circ}(hg) \Phi(g) \left|\det g\right|^{s + \frac{n - 1}{2}} \, dg = L(s,\pi) W^{\circ}(h),
\end{equation}
where $\Phi \in \Ss_0(\Mat_{n \times n}(F))$ is the standard Schwartz function given by \eqref{eq:Phistandard}.
\end{lemma}

\begin{proof}
We show this initially for $\Re(t_1) > \cdots > \Re(t_r)$ with $\Re(t_1)$ sufficiently large; this identity then extends via analytic continuation to all induced representations of Langlands type, for the left-hand side is absolutely convergent for $\Re(s)$ sufficiently large by \cite[Lemma 3.2 (ii) and Proposition 3.3]{Jac09}.

We replace $h$ with $w_n uh$ in the convolution section identity \eqref{eq:Pieri} for $f^{\circ}$ and insert this identity into the Jacquet integral \eqref{eq:Jacquetint}. The result then follows upon interchanging the order of integration, which is justified by the absolute convergence of the Jacquet integral together with the absolute convergence of the integral \eqref{eq:Pieri}.
\end{proof}

\subsection{The Newform via Godement Sections}
\label{sect:WhittakernewformGodement}

Next, we use Godement section identities for the newform in the induced model together with the Jacquet integral \eqref{eq:Jacquetint} in order to give a Godement section identity for the Whittaker newform. This is a propagation formula: a recursive formula for $W^{\circ}$ in terms of an integral over $\GL_{n - 1}(F)$ involving a $\GL_{n - 1}$ Whittaker function and a distinguished standard Schwartz function.

\subsubsection{The Case $\pi_1 = \chi^{\kappa_1} |\cdot|^{t_1}$}

As in \hyperref[sect:newformGodement]{Section \ref*{sect:newformGodement}}, we first treat the case for which $\pi = \pi_1 \boxplus \cdots \boxplus \pi_r$ with $n_1 = 1$, so that $\pi_1 = \chi^{\kappa_1} |\cdot|^{t_1}$.

\begin{lemma}
\label{lem:propagationformula}
For $n \geq 2$, let $\pi = \pi_1 \boxplus \pi_2 \boxplus \cdots \boxplus \pi_r$ and $\pi_0 \coloneqq \pi_2 \boxplus \cdots \boxplus \pi_r$ be induced representations of Langlands type of $\GL_n(F)$ and $\GL_{n - 1}(F)$ with $\pi_1 = \chi^{\kappa_1} |\cdot|^{t_1}$. Let $W^{\circ} \in \WW(\pi,\psi)$ and $W_0^{\circ} \in \WW(\pi_0,\psi)$ be the Whittaker newforms of $\pi$ and $\pi_0$. Then for $g \in \GL_{n - 1}(F)$,
\begin{equation}
\label{eq:WcircviaW0circ}
W^{\circ} \begin{pmatrix} g & 0 \\ 0 & 1 \end{pmatrix} = \left|\det g\right|^{t_1 + \frac{\|\kappa_1\|}{d_F} + \frac{n - 1}{2}} \int_{\GL_{n - 1}(F)} W_0^{\circ}(h) \Phi_1(h^{-1} g) \Phi_0^{\circ}(e_{n - 1} h) \left|\det h\right|^{-t_1 - \frac{\|\kappa_1\|}{d_F} - \frac{n}{2} + 1} \, dh,
\end{equation}
with the standard Schwartz functions $\Phi_1 \in \Ss_0(\Mat_{(n - 1) \times (n - 1)}(F))$ and $\Phi_0^{\circ} \in \Ss_0(\Mat_{1 \times (n - 1)}(F))$ given by
\begin{align}
\label{eq:Phi1}
\Phi_1(x_1) & \coloneqq \exp\left(-d_F \pi \Tr\left(x_1 \prescript{t}{}{\overline{x_1}}\right)\right),	\\
\label{eq:Phi0circ}
\Phi_0^{\circ}(x_2) & \coloneqq \left(\dim \tau_0^{\circ}\right) \overline{P_0^{\circ}}(x_2) \exp\left(-d_F \pi x_2 \prescript{t}{}{\overline{x_2}}\right),
\end{align}
where $P_0^{\circ}$ is the homogeneous harmonic polynomial associated to the newform $K_{n - 1}$-type $\tau_0^{\circ}$ of $\pi_0$ via \eqref{eq:PcircUn} and \eqref{eq:PcircOn}.
\end{lemma}

\begin{remark}
When $n = 2$, so that $\pi = \chi^{\kappa_1} |\cdot|^{t_1} \boxplus \chi^{\kappa_2} |\cdot|^{t_2}$, the integral over $\GL_1(F) = F^{\times} \ni h$ in \eqref{eq:WcircviaW0circ} may be explicitly evaluated in order to show that
\[W^{\circ}\begin{pmatrix} g & 0 \\ 0 & 1 \end{pmatrix} = \begin{dcases*}
2 |g|^{\frac{t_1 + t_2 + 1 + \kappa_1 + \kappa_2}{2}} K_{\frac{t_1 - t_2 + \kappa_1 - \kappa_2}{2}}(2\pi|g|) & if $F = \R$,	\\
4 |g|^{\frac{t_1 + t_2 + 1}{2} + \frac{\|\kappa_1\| + \|\kappa_2\|}{4}} K_{t_1 - t_2 + \frac{\|\kappa_1\| - \|\kappa_2\|}{2}}(4\pi|g|^{1/2}) & if $F = \C$,
\end{dcases*}\]
where $K_{\nu}(z)$ denotes the modified Bessel function of the second kind.
\end{remark}

\begin{proof}
We show this initially for $\Re(t_1) > \cdots > \Re(t_r)$ with $\Re(t_1)$ sufficiently large; from \cite[Proposition 7.2]{Jac09}, this identity then extends via analytic continuation to all induced representations of Langlands type (note that Jacquet instead works with representations that are induced from a \emph{lower} parabolic subgroup rather than an upper parabolic subgroup). We may write
\[W^{\circ} \begin{pmatrix} g & 0 \\ 0 & 1 \end{pmatrix} = \int\limits_{\Mat_{(n - 1) \times 1}(F)} \int_{\Ngp_{n - 1}(F)} f^{\circ}\left(\begin{pmatrix} 0 & 1 \\ w_{n - 1} & 0 \end{pmatrix} \begin{pmatrix} u & v \\ 0 & 1 \end{pmatrix} \begin{pmatrix} g & 0 \\ 0 & 1 \end{pmatrix}\right) \overline{\psi_{n - 1}}(u) \overline{\psi}(e_{n - 1} v) \, du \, dv\]
from the definition \eqref{eq:Jacquetint} of the Jacquet integral. We insert the Godement section identity \eqref{eq:fviaf0} for $f^{\circ}$ with $g$ replaced by $\begin{psmallmatrix} 0 & 1 \\ w_{n - 1} & 0 \end{psmallmatrix} \begin{psmallmatrix} u & v \\ 0 & 1 \end{psmallmatrix} \begin{psmallmatrix} g & 0 \\ 0 & 1 \end{psmallmatrix}$ into this expression, additionally inserting the identity \eqref{eq:ccirctoc0circ} for the normalising constant $c^{\circ}$. As explicated in \cite[Section 7.2]{Jac09}, the ensuing double integral is absolutely convergent, so that we can make the change of variables $h \mapsto w_{n - 1} u h$ and $v \mapsto uhv$; we find that $W^{\circ}\begin{psmallmatrix} g & 0 \\ 0 & 1 \end{psmallmatrix}$ is equal to
\begin{multline*}
i^{c(\pi_0)} (-1)^{\kappa_1(n - 1)} \chi^{\kappa_1}(\det w_n) \chi^{-\kappa_1}(\det w_{n - 1}) \chi^{\kappa_1}(\det g) \left|\det g\right|^{t_1 + \frac{n - 1}{2}}	\\
\times \int_{\GL_{n - 1}(F)} \overline{P_1^{\circ}}\left(\det\left(h^{-1} g\right)\right) \exp\left(-d_F \pi \Tr\left(h^{-1} g \prescript{t}{}{\overline{g}} \prescript{t}{}{\overline{h}}^{-1}\right)\right) \chi^{-\kappa_1}(\det h) \left|\det h\right|^{-t_1 - \frac{n}{2} + 1}	\\
\times (\dim \tau_0^{\circ}) \int\limits_{\Mat_{(n - 1) \times 1}(F)} \overline{P_0^{\circ}}\left(\prescript{t}{}{\overline{v}}\right) \exp\left(-d_F \pi \prescript{t}{}{\overline{v}} v\right) \overline{\psi}(e_{n - 1} hv) \, dv \int_{\Ngp_{n - 1}(F)} f_0^{\circ}(w_{n - 1} uh) \overline{\psi_{n - 1}}(u) \, du \, dh.
\end{multline*}
The integral over $\Mat_{(n - 1) \times 1}(F) \ni v$ is equal to
\[i^{-c(\pi_0)} \overline{P_0^{\circ}}(e_{n - 1} h) \exp(-\pi e_{n - 1} h \prescript{t}{}{h} \prescript{t}{}{e}_{n - 1})\]
via Hecke's identity, \hyperref[lem:HeckeC]{Lemmata \ref*{lem:HeckeC}} and \ref{lem:HeckeR}, while the integral over $\Ngp_{n - 1}(F) \ni u$ is equal to $W_0^{\circ}(h)$ via \eqref{eq:Jacquetint}. It remains to use \eqref{eq:abstochar} in conjunction with the definition of $P_1^{\circ}$ as well as to note that $\det w_n \det w_{n - 1} = (-1)^{n - 1}$, so that $\chi^{\kappa_1}(\det w_n) \chi^{-\kappa_1}(\det w_{n - 1}) = (-1)^{\kappa_1(n - 1)}$.
\end{proof}

\begin{remark}
For spherical Whittaker functions, such a propagation formula (in a slightly modified form) is due to Gerasimov, Lebedev, and Oblezin \cite[Proposition 4.1]{GLO08} and Ishii and Stade \cite[Proposition 2.1]{IsSt13} (cf.\ \cite[Appendix A]{IM22}); iterating this propagation formula gives a recursive formula for $\GL_n(F)$ Whittaker functions in terms of $\GL_2(F)$ and $\GL_{n - 2}(F)$ Whittaker functions known earlier by the work of Stade \cite[Theorem 2.1]{Sta90}.
\end{remark}

We also require the following propagation formula for $W^{\prime\circ} \in \WW(\pi',\overline{\psi})$ when $\pi'$ is spherical, which follows analogously to \hyperref[lem:propagationformula]{Lemma \ref*{lem:propagationformula}}.

\begin{lemma}
For $n \geq 2$, let $\pi' = |\cdot|^{t_1'} \boxplus |\cdot|^{t_2'} \boxplus \cdots \boxplus |\cdot|^{t_n'}$ and $\pi_0' \coloneqq |\cdot|^{t_2'} \boxplus \cdots \boxplus |\cdot|^{t_n'}$ be spherical representations of Langlands type of $\GL_n(F)$ and $\GL_{n - 1}(F)$. Let $W^{\prime\circ} \in \WW(\pi',\overline{\psi})$ and $W_0^{\prime\circ} \in \WW(\pi_0',\overline{\psi})$ be the spherical Whittaker functions of $\pi'$ and $\pi_0'$. Then for $g \in \GL_n(F)$,
\begin{multline}
\label{eq:W'circviaW0'circ}
W^{\prime\circ}(g) = \left|\det g\right|^{t_1' + \frac{n - 1}{2}} \int_{\GL_{n - 1}(F)} W_0^{\prime\circ}(h) \left|\det h\right|^{-t_1' - \frac{n}{2}}	\\
\times \int\limits_{\Mat_{(n - 1) \times 1}(F)} \Phi'\left(h^{-1} \begin{pmatrix} 1_{n - 1} & v \end{pmatrix} g\right) \psi(e_{n - 1} v) \, dv \, dh,
\end{multline}
with the standard Schwartz function $\Phi' \in \Ss_0(\Mat_{(n - 1) \times n}(F))$ given by
\begin{equation}
\label{eq:Phi'}
\Phi'(x) \coloneqq \exp\left(-d_F \pi \Tr\left(x \prescript{t}{}{\overline{x}}\right)\right).
\end{equation}
\end{lemma}

\subsubsection{The Case $\pi_1 = D_{\kappa_1} \otimes \left|\det\right|^{t_1}$}

We next treat the case for which $\pi = \pi_1 \boxplus \cdots \boxplus \pi_r$ with $n_1 = 2$, so that $F = \R$ and $\pi_1 = D_{\kappa_1} \otimes \left|\det\right|^{t_1}$.

\begin{lemma}
For $n \geq 2$, let $\pi = \pi_1 \boxplus \pi_2 \boxplus \cdots \boxplus \pi_r$ and $\pi_0^{\ast} \coloneqq \pi_1^{\ast} \boxplus \pi_2 \boxplus \cdots \boxplus \pi_r$ be induced representations of Langlands type of $\GL_n(\R)$ and $\GL_{n - 1}(\R)$ with $\pi_1 = D_{\kappa_1} \otimes \left|\det\right|^{t_1}$ and $\pi_1^{\ast} \coloneqq |\cdot|^{t_1 + \frac{\kappa_1 + 1}{2}}$. Let $W^{\circ} \in \WW(\pi,\psi)$ and $W_0^{\ast\circ} \in \WW(\pi_0^{\ast},\psi)$ be the Whittaker newforms of $\pi$ and $\pi_0^{\ast}$. Then for $g \in \GL_{n - 1}(\R)$,
\begin{equation}
\label{eq:WcircviaW0astcirc}
W^{\circ} \begin{pmatrix} g & 0 \\ 0 & 1 \end{pmatrix} = \left|\det g\right|^{t_1 + \frac{\kappa_1 - 1}{2} + \frac{n - 1}{2}} \int_{\GL_{n - 1}(\R)} W_0^{\ast\circ}(h) \Phi_1(h^{-1} g) \Phi_0^{\circ}(e_{n - 1} h) \left|\det h\right|^{-t_1 - \frac{\kappa_1 - 1}{2} - \frac{n}{2} + 1} \, dh,
\end{equation}
with the standard Schwartz functions $\Phi_1 \in \Ss_0(\Mat_{(n - 1) \times (n - 1)}(\R))$ as in \eqref{eq:Phi1} and $\Phi_0^{\ast\circ} \in \Ss_0(\Mat_{1 \times (n - 1)}(\R))$ given by
\begin{equation}
\label{eq:Phi0astcirc}
\Phi_0^{\ast\circ}(x_2) \coloneqq \left(\dim \tau_0^{\ast\circ}\right) \overline{P_0^{\ast\circ}}(x_2) \exp\left(-\pi x_2 \prescript{t}{}{x}_2\right),
\end{equation}
where $P_0^{\ast\circ}$ is the homogeneous harmonic polynomial associated to the newform $\Ogp(n - 1)$-type $\tau_0^{\ast\circ}$ of $\pi_0^{\ast}$ via \eqref{eq:PcircOn}.
\end{lemma}

\begin{remark}
When $n = 2$, so that $\pi = D_{\kappa} \otimes \left|\det\right|^t$, the integral over $\GL_1(\R) = \R^{\times} \ni h$ in \eqref{eq:WcircviaW0astcirc} may be explicitly evaluated in order to show that
\[W^{\circ}\begin{pmatrix} g & 0 \\ 0 & 1 \end{pmatrix} = |g|^{t + \frac{\kappa}{2}} \exp(-2\pi|g|).\]
\end{remark}

\begin{proof}
First we consider the case $n = 2$, so that $\pi = D_{\kappa} \otimes \left|\det\right|^t$ and $\pi_0^{\ast} = |\cdot|^{t + \frac{\kappa + 1}{2}}$. We insert the Godement section identity \eqref{eq:fviaf0ds} for $f^{\circ}$ with $g$ replaced by $\begin{psmallmatrix} 0 & 1 \\ 1 & 0 \end{psmallmatrix} \begin{psmallmatrix} 1 & v' \\0 & 1 \end{psmallmatrix} \begin{psmallmatrix} g & 0 \\ 0 & 1 \end{psmallmatrix}$ into the Jacquet integral \eqref{eq:Jacquetint}, then interchange the order of integration and make the change of variables $v' \mapsto a_2 v'$, yielding
\begin{multline*}
W^{\circ}\begin{pmatrix} g & 0 \\ 0 & 1 \end{pmatrix} = i^{\kappa} |g|^{t + \frac{\kappa}{2}} \int_{\R^{\times}} |a_2| a_2^{-\kappa} \exp\left(-\pi (a_2^{-1} g)^2\right)	\\
\times \int_{\R} \int_{\R} \overline{P^{\circ}}(v,v') \exp\left(-\pi\left(v^2 + v^{\prime 2}\right)\right) \overline{\psi}(a_2 v') \, dv \, dv' \, d^{\times}a_2.
\end{multline*}
Via Hecke's identity, \hyperref[lem:HeckeR]{Lemma \ref*{lem:HeckeR}}, the integral over $\R^2 \ni (v,v')$ is
\[i^{-\kappa} \overline{P^{\circ}}(0,a_2) \exp(-\pi a_2^2) = i^{-\kappa} a_2^{\kappa} \exp(-\pi a_2^2),\]
and so we obtain the identity \eqref{eq:WcircviaW0astcirc} upon relabelling $a_2$ as $h$.

Now we consider the case $n \geq 3$. We again show this initially for $\Re(t_1) > \cdots > \Re(t_r)$ with $\Re(t_1)$ sufficiently large; from \cite[Proposition 7.2]{Jac09}, this identity then extends via analytic continuation to all induced representations of Langlands type. The derivation of the identity \eqref{eq:WcircviaW0astcirc} is somewhat indirect: we first determine an alternate expression for the right-hand side of \eqref{eq:WcircviaW0astcirc}, and then show that the left-hand side is equal to this expression.

To begin, let $\pi^{\ast} \coloneqq \pi_1^{\ast} \boxplus \pi_2^{\ast} \boxplus \pi_2 \boxplus \cdots \boxplus \pi_r$ and $\pi_0^{\ast} \coloneqq \pi_2^{\ast} \boxplus \pi_2 \boxplus \cdots \boxplus \pi_r$ be induced representations of Langlands type of $\GL_n(\R)$ and $\GL_{n - 1}(\R)$ with $\pi_1^{\ast} = |\cdot|^{t_1^{\ast}}$ and $\pi_2^{\ast} = |\cdot|^{t_2^{\ast}}$, and let $W^{\ast\circ} \in \WW(\pi^{\ast},\psi)$ and $W_0^{\ast\circ} \in \WW(\pi_0^{\ast},\psi)$ be the Whittaker newforms of $\pi^{\ast}$ and $\pi_0^{\ast}$.

On the one hand, we have from \eqref{eq:WcircviaW0circ} that $W^{\ast\circ} \begin{psmallmatrix} g & 0 \\ 0 & 1 \end{psmallmatrix}$ is equal to
\begin{equation}
\label{eq:Wastcirc1}
\left|\det g\right|^{t_1^{\ast} + \frac{n - 1}{2}} \int_{\GL_{n - 1}(\R)} W_0^{\ast\circ}(h) \Phi_1(h^{-1} g) \Phi_0^{\ast\circ}(e_{n - 1} h) \left|\det h\right|^{-t_1^{\ast} - \frac{n}{2} + 1} \, dh,
\end{equation}
with $\Phi_1$ as in \eqref{eq:Phi1} and $\Phi_0^{\ast\circ}$ as in \eqref{eq:Phi0astcirc}.

On the other hand, for $\Re(t_1^{\ast})$ sufficiently large, $W^{\ast\circ} \begin{psmallmatrix} g & 0 \\ 0 & 1 \end{psmallmatrix}$ is equal to
\begin{multline*}
\int\limits_{\Mat_{(n - 2) \times 1}(\R)} \int\limits_{\Mat_{(n - 2) \times 1}(\R)} \int_{\R} \int_{\Ngp_{n - 2}(\R)} f^{\ast\circ}\left(\begin{pmatrix} 0 & 0 & 1 \\ 0 & 1 & 0 \\ w_{n - 2} & 0 & 0 \end{pmatrix} \begin{pmatrix} u' & v_1' & v_2' \\ 0 & 1 & v_3' \\ 0 & 0 & 1 \end{pmatrix} \begin{pmatrix} g & 0 \\ 0 & 1 \end{pmatrix}\right)	\\
\times \overline{\psi_{n - 2}}(u) \overline{\psi}(e_{n - 2} v_1') \overline{\psi}(v_3') \, du \, dv_3' \, dv_2' \, dv_1'.
\end{multline*}
We insert the Godement section identity \eqref{eq:fastviaf02} for $f^{\ast\circ}$ into this expression, additionally inserting the identity \eqref{eq:ccirctoc0circ} for the normalising constant $c^{\ast\circ}$. By a straightforward extension of \cite[Proposition 7.2]{Jac09}, the ensuing multiple integral is absolutely convergent, so that we can make the change of variables $h \mapsto w_{n - 2} u' h$, $v_1' \mapsto u' v_1'$, $v_2' \mapsto u' hv_2'$, $v_3 \mapsto u^{\prime -1} w_{n - 2}$, and $v_3' \mapsto a_2 (v_3' - v_3 hv_2')$. Using the definition of the Jacquet integral, \eqref{eq:Jacquetint}, to evaluate the ensuing integral over $\Ngp_{n - 2}(\R) \ni u'$ and Hecke's identity, \hyperref[lem:HeckeR]{Lemma \ref*{lem:HeckeR}}, to evaluate the ensuing integrals over $\Mat_{(n - 2) \times 1}(\R) \ni v_2'$ and $\R \ni v_3'$, we find that
\begin{multline}
\label{eq:Wastcirc2}
W^{\ast\circ} \begin{pmatrix} g & 0 \\ 0 & 1 \end{pmatrix} = (-1)^{c(\pi_0)} L\left(1 + t_2^{\ast}, \widetilde{\pi_0}\right) \left|\det g\right|^{t_1^{\ast} + \frac{n - 1}{2}} \int_{\GL_{n - 2}(\R)} W_0^{\circ}(h) \left|\det h\right|^{1 - t_1^{\ast} - \frac{n - 1}{2}}	\\
\times \int\limits_{\Mat_{(n - 2) \times 1}(\R)} \exp\left(-\pi \Tr\left(h^{-1} \begin{pmatrix} 1_{n - 2} & v_1' \end{pmatrix} g \prescript{t}{}{g} \begin{pmatrix} 1_{n - 2} \\ \prescript{t}{}{v}_1' \end{pmatrix} \prescript{t}{}{h}^{-1}\right)\right) \overline{\psi}(e_{n - 2} v_1')	\\
\times \int_{\R^{\times}} |a_2|^{-t_1^{\ast} + t_2^{\ast}} \exp(-\pi a_2^2) \int\limits_{\Mat_{1 \times (n - 2)}(\R)} (\dim \tau_0^{\circ}) \overline{P_0^{\circ}}(a_2 v_3 h) \exp(-\pi a_2^2 v_3 h \prescript{t}{}{h} \prescript{t}{}{v}_3)	\\
\times \exp\left(-\pi \Tr\left(\begin{pmatrix} v_3 & v_3 v_1' + a_2^{-1} \end{pmatrix} g \prescript{t}{}{g} \begin{pmatrix} \prescript{t}{}{v}_3 \\ v_1' v_3 + a_2^{-1} \end{pmatrix}\right)\right) \, dv_3 \, d^{\times}a_2 \, dv_1' \, dh.
\end{multline}
Here $W_0^{\circ}$ is the Whittaker newform for $\pi_0^{\circ} \coloneqq \pi_2 \boxplus \cdots \boxplus \pi_r$ and $P_0^{\circ}$ is the homogeneous harmonic polynomial associated to the newform $\Ogp(n - 2)$-type $\tau_0^{\circ}$ of $\pi^{\circ}$ via \eqref{eq:PcircOn}, and we have used \hyperref[thm:additiveconductor]{Theorem \ref*{thm:additiveconductor}} to write $c(\pi_0^{\ast}) = c(\pi_2^{\ast}) + c(\pi_0) = c(\pi_0)$.

Next, we note that the identities \eqref{eq:Wastcirc1} and \eqref{eq:Wastcirc2} for $W^{\ast\circ} \begin{psmallmatrix} g & 0 \\ 0 & 1 \end{psmallmatrix}$ both extend holomorphically to $t_1^{\ast} = t_1 + (\kappa_1 - 1)/2$ and $t_2^{\ast} = t_1 + (\kappa_1 + 1)/2$. From this, we see that the right-hand side of \eqref{eq:WcircviaW0astcirc} is equal to
\begin{multline}
\label{eq:Wastcirc3}
(-1)^{c(\pi_0)} L\left(1 + t_1 + \frac{\kappa_1 + 1}{2}, \widetilde{\pi_0}\right) \left|\det g\right|^{t_1 + \frac{\kappa_1 - 1}{2} + \frac{n - 1}{2}} \int_{\GL_{n - 2}(\R)} W_0^{\circ}(h) \left|\det h\right|^{1 - t_1 - \frac{\kappa_1 - 1}{2} - \frac{n - 1}{2}}	\\
\times \int\limits_{\Mat_{(n - 2) \times 1}(\R)} \exp\left(-\pi \Tr\left(h^{-1} \begin{pmatrix} 1_{n - 2} & v_1' \end{pmatrix} g \prescript{t}{}{g} \begin{pmatrix} 1_{n - 2} \\ \prescript{t}{}{v}_1' \end{pmatrix} \prescript{t}{}{h}^{-1}\right)\right) \overline{\psi}(e_{n - 2} v_1')	\\
\times \int_{\R^{\times}} |a_2| \exp(-\pi a_2^2) \int\limits_{\Mat_{1 \times (n - 2)}(\R)} (\dim \tau_0^{\circ}) \overline{P_0^{\circ}}(a_2 v_3 h) \exp(-\pi a_2^2 v_3 h \prescript{t}{}{h} \prescript{t}{}{v}_3)	\\
\times \exp\left(-\pi \Tr\left(\begin{pmatrix} v_3 & v_3 v_1' + a_2^{-1} \end{pmatrix} g \prescript{t}{}{g} \begin{pmatrix} \prescript{t}{}{v}_3 \\ v_1' v_3 + a_2^{-1} \end{pmatrix}\right)\right) \, dv_3 \, d^{\times}a_2 \, dv_1' \, dh.
\end{multline}

Now we show that $W^{\circ} \begin{psmallmatrix} g & 0 \\ 0 & 1 \end{psmallmatrix}$ is equal to \eqref{eq:Wastcirc3} when $\Re(t_1)$ is sufficiently large, from which the result shall follow via analytic continuation. We begin by noting that it is equal to
\begin{multline*}
\int\limits_{\Mat_{(n - 2) \times 1}(\R)} \int\limits_{\Mat_{(n - 2) \times 1}(\R)} \int_{\R} \int_{\Ngp_{n - 2}(\R)} f^{\circ}\left(\begin{pmatrix} 0 & 0 & 1 \\ 0 & 1 & 0 \\ w_{n - 2} & 0 & 0 \end{pmatrix} \begin{pmatrix} u' & v_1' & v_2' \\ 0 & 1 & v_3' \\ 0 & 0 & 1 \end{pmatrix} \begin{pmatrix} g & 0 \\ 0 & 1 \end{pmatrix}\right)	\\
\times \overline{\psi_{n - 2}}(u) \overline{\psi}(e_{n - 2} v_1') \overline{\psi}(v_3') \, du \, dv_3' \, dv_2' \, dv_1'.
\end{multline*}
We insert the Godement section identity \eqref{eq:fviaf02} for $f^{\circ}$ into this expression, additionally inserting the identity \eqref{eq:ccirctoc0circ} for the normalising constant $c^{\circ}$. The ensuing multiple integral is again absolutely convergent, so that we can make the change of variables $h \mapsto w_{n - 2} u' h$, $v_1' \mapsto u' v_1'$, $v_2' \mapsto u' hv_2'$, $v_3 \mapsto u^{\prime -1} w_{n - 2}$, and $v_3' \mapsto a_2 (v_3' - v_3 hv_2')$. We again evaluate the ensuing integral over $\Ngp_{n - 2}(\R) \ni u'$ via the definition of the Jacquet integral and use Hecke's identity, \hyperref[lem:HeckeR]{Lemma \ref*{lem:HeckeR}}, to evaluate the integrals over $\Mat_{(n - 2) \times 1}(\R) \ni v_2'$ and $\R^2 \ni (v_3,v_3')$. The latter integral is equal to
\[i^{-\kappa_1} \overline{P_1^{\circ}}(0,a_2) \exp(-\pi a_2^2) = i^{-\kappa_1} a_2^{\kappa_1} \exp(-\pi a_2^2).\]
The resulting expression is precisely \eqref{eq:Wastcirc3}.
\end{proof}

\section{Rankin--Selberg Integrals}
\label{sect:RankinSelberg}

It is time to put the propagation formul\ae{} \eqref{eq:WcircviaW0circ}, \eqref{eq:W'circviaW0'circ}, and \eqref{eq:WcircviaW0astcirc} for $W^{\circ}$ and $W^{\prime\circ}$ to good use. Following the method of Jacquet \cite[Section 8]{Jac09}, we use these formul\ae{} to express the $\GL_n \times \GL_n$ Rankin--Selberg integral as the product of a $\GL_n \times \GL_{n - 1}$ Rankin--Selberg integral and a $\GL_n \times \GL_1$ Rankin--Selberg $L$-function, and similarly express the $\GL_n \times \GL_{n - 1}$ Rankin--Selberg integral as a product of a $\GL_{n - 1} \times \GL_{n - 1}$ Rankin--Selberg integral and a $\GL_1 \times \GL_{n - 1}$ Rankin--Selberg $L$-function.

\subsection{\texorpdfstring{$\GL_n \times \GL_n$}{GL\9040\231\80\327GL\9040\231} Rankin--Selberg Integrals}

We first consider the $\GL_1 \times \GL_1$ Rankin--Selberg integral defined by \eqref{eq:RankinSelbergnn}; this is simply the Tate zeta integral.

\begin{proposition}
\label{prop:basecase}
Let $\pi = \chi^{\kappa} |\cdot|^t$ be a character of $F^{\times}$ and let $\pi' = |\cdot|^{t'}$ be a spherical character of $F^{\times}$. Let $W^{\circ} \in \WW(\pi,\psi)$ be the Whittaker newform of $\pi$ and let $W^{\prime\circ} \in \WW(\pi',\overline{\psi})$ be the spherical Whittaker function of $\pi'$. Then for $\Re(s)$ sufficiently large, the $\GL_1 \times \GL_1$ Rankin--Selberg integral $\Psi(s,W^{\circ},W^{\prime\circ},\Phi^{\circ})$ is equal to $L(s,\pi \times \pi')$ with the standard Schwartz function $\Phi^{\circ} \in \Ss_0(\Mat_{1 \times 1}(F))$ given by
\[\Phi^{\circ}(x) \coloneqq \overline{P^{\circ}}(x) \exp\left(-d_F \pi x \overline{x}\right),\]
where $P^{\circ}$ is the homogeneous harmonic polynomial associated to the newform $K$-type $\tau^{\circ}$ of $\pi$ via \eqref{eq:PcircUn} and \eqref{eq:PcircOn}.
\end{proposition}

\begin{proof}
By definition, $W^{\circ}(g) = \chi^{\kappa}(g) |g|^t$ and $W^{\prime\circ}(g) = |g|^{t'}$. We then use \eqref{eq:abstochar} in conjunction with the definition of $P^{\circ}$ in order to see that
\[\Psi(s,W^{\circ},W^{\prime\circ},\Phi^{\circ}) = \int_{F^{\times}} |x|^{s + t + \frac{\|\kappa\|}{d_F} + t'} \exp(-d_F \pi x \overline{x}) \, d^{\times}x.\]
By the identities \eqref{eq:zetaCint} and \eqref{eq:zetaRint} relating the integral to zeta functions, the identities \eqref{eq:L(s,pi)C} and \eqref{eq:L(s,pi)R1} relating zeta functions to $L$-functions, and the identity \eqref{eq:twistedLfunction} relating Rankin--Selberg $L$-functions involving twists by a character to standard $L$-functions, this is precisely $L(s, \pi \times \pi')$.
\end{proof}

Next, we prove a recursive formula for the $\GL_n \times \GL_n$ Rankin--Selberg integral for $n \geq 2$.

\begin{proposition}
\label{prop:(n,n)reduction}
For $n \geq 2$, let $\pi = \pi_1 \boxplus \cdots \boxplus \pi_r$ be an induced representations of Langlands type of $\GL_n(F)$, and let $\pi' = |\cdot|^{t_1'} \boxplus |\cdot|^{t_2'} \boxplus \cdots \boxplus |\cdot|^{t_n'}$ and $\pi_0' \coloneqq |\cdot|^{t_2'} \boxplus \cdots \boxplus |\cdot|^{t_n'}$ be spherical representations of Langlands type of $\GL_n(F)$ and $\GL_{n - 1}(F)$. Let $W^{\circ} \in \WW(\pi,\psi)$ be the Whittaker newform of $\pi$ and let $W^{\prime\circ} \in \WW(\pi',\overline{\psi})$ and $W_0^{\prime\circ} \in \WW(\pi_0',\overline{\psi})$ be the spherical Whittaker functions of $\pi'$ and $\pi_0'$. Then for $\Re(s)$ sufficiently large, the $\GL_n \times \GL_n$ Rankin--Selberg integral $\Psi(s,W^{\circ},W^{\prime\circ},\Phi^{\circ})$ is equal to
\[\Psi(s,W^{\circ},W_0^{\prime\circ}) L\left(s, \pi \times |\cdot|^{t_1'}\right),\]
with the standard Schwartz function $\Phi^{\circ} \in \Ss_0(\Mat_{1 \times n}(F))$ given by
\[\Phi^{\circ}(x) \coloneqq \left(\dim \tau^{\circ}\right) \overline{P^{\circ}}(x) \exp\left(-d_F \pi x \prescript{t}{}{\overline{x}}\right),\]
where $P^{\circ}$ is the homogeneous harmonic polynomial associated to the newform $K$-type $\tau^{\circ}$ of $\pi$ via \eqref{eq:PcircUn} and \eqref{eq:PcircOn}.
\end{proposition}

\begin{proof}
Just as in \cite[Equation (8.1)]{Jac09}, we insert the propagation formula \eqref{eq:W'circviaW0'circ} for $W^{\prime\circ}(g)$ into the definition \eqref{eq:RankinSelbergnn} of $\Psi(s,W^{\circ},W^{\prime\circ},\Phi^{\circ})$; the absolute convergence of the triple integral is shown in \cite[Section 8.2]{Jac09}. We replace $h$ with $uh$, where now $u \in \Ngp_{n - 1}(F)$ and $h \in \Ngp_{n - 1}(F) \backslash \GL_{n - 1}(F)$, make the change of variables $u \mapsto u^{-1}$ and $v \mapsto u^{-1} v$, then replace $\begin{psmallmatrix} u & v \\ 0 & 1\end{psmallmatrix} g$ with $g$, where now $g \in \GL_n(F)$; in doing so, we use the fact that $W^{\prime\circ}(uh) = \overline{\psi_{n - 1}}(u) W^{\prime\circ}(h)$ and that $\psi_{n - 1}(u) \psi(e_{n - 1} v) W^{\circ}(g) = W^{\circ}\left(\begin{psmallmatrix} u & v \\ 0 & 1\end{psmallmatrix} g\right)$. We then make the change of variables $g \mapsto \begin{psmallmatrix} h & 0 \\ 0 & 1 \end{psmallmatrix} g$. In this way, we find that $\Psi(s,W^{\circ},W^{\prime\circ},\Phi^{\circ})$ is equal to
\[\int\limits_{\Ngp_{n - 1}(F) \backslash \GL_{n - 1}(F)} W_0^{\prime\circ}(h) \left|\det h\right|^{s - \frac{1}{2}} \int_{\GL_n(F)} W^{\circ}\left(\begin{pmatrix} h & 0 \\ 0 & 1 \end{pmatrix} g\right) \Phi(g) \left|\det g\right|^{s + t_1' + \frac{n - 1}{2}} \, dg \, dh,\]
where we have defined $\Phi(g) \coloneqq \Phi^{\circ}(e_n g) \Phi'\left(\begin{pmatrix} 1_{n - 1} & 0 \end{pmatrix} g\right)$, where $\Phi' \in \Ss_0(\Mat_{(n - 1) \times n}(F))$ is as in \eqref{eq:Phi'}. Since the standard Schwartz function $\Phi \in \Ss_0(\Mat_{n \times n}(F))$ is as in \eqref{eq:Phistandard}, the integral over $\GL_n(F) \ni g$ is equal to $L(s + t_1',\pi) W^{\circ}\begin{psmallmatrix} h & 0 \\ 0 & 1 \end{psmallmatrix}$ from \eqref{eq:WhittakerPieri}. This yields the desired identity upon recalling the definition \eqref{eq:RankinSelbergnm} of $\Psi(s,W^{\circ},W_0^{\prime\circ})$ and the identity \eqref{eq:twistedLfunction} relating Rankin--Selberg $L$-functions involving twists by a character to standard $L$-functions.
\end{proof}

\subsection{\texorpdfstring{$\GL_n \times \GL_{n - 1}$}{GL\9040\231\80\327GL\9040\231\9040\213\9040\201} Rankin--Selberg Integrals}

\subsubsection{The Case $\pi_1 = \chi^{\kappa_1} |\cdot|^{t_1}$}

We now prove a recursive formula for the $\GL_n \times \GL_{n - 1}$ Rankin--Selberg integral. As in \hyperref[sect:newformGodement]{Sections \ref*{sect:newformGodement}} and \ref{sect:WhittakernewformGodement}, we first treat the case for which $\pi = \pi_1 \boxplus \cdots \boxplus \pi_r$ with $n_1 = 1$, so that $\pi_1 = \chi^{\kappa_1} |\cdot|^{t_1}$.

\begin{proposition}
\label{prop:(n,n-1)reduction}
For $n \geq 2$, let $\pi = \pi_1 \boxplus \pi_2 \boxplus \cdots \boxplus \pi_r$ and $\pi_0 \coloneqq \pi_2 \boxplus \cdots \boxplus \pi_r$ be induced representations of Langlands type of $\GL_n(F)$ and $\GL_{n - 1}(F)$ with $\pi_1 = \chi^{\kappa_1} |\cdot|^{t_1}$ a character of $F^{\times}$, and let $\pi' = |\cdot|^{t_1'} \boxplus \cdots \boxplus |\cdot|^{t_{n - 1}'}$ be a spherical representation of Langlands type of $\GL_{n - 1}(F)$. Let $W^{\circ} \in \WW(\pi,\psi)$ and $W_0^{\circ} \in \WW(\pi_0,\psi)$ be the Whittaker newforms of $\pi$ and $\pi_0$ and let $W^{\prime\circ} \in \WW(\pi',\overline{\psi})$ be the spherical Whittaker function of $\pi'$. Then for $\Re(s)$ sufficiently large, the $\GL_n \times \GL_{n - 1}$ Rankin--Selberg integral $\Psi(s,W^{\circ},W^{\prime\circ})$ is equal to
\[\Psi(s,W_0^{\circ},W^{\prime\circ},\Phi_0^{\circ}) L\left(s, \pi_1 \times \pi'\right),\]
with the standard Schwartz function $\Phi_0^{\circ} \in \Ss_0(\Mat_{1 \times (n - 1)}(F))$ given by
\[\Phi_0^{\circ}(x) \coloneqq \left(\dim \tau_0^{\circ}\right) \overline{P_0^{\circ}}(x) \exp\left(-d_F \pi x \prescript{t}{}{\overline{x}}\right),\]
where $P_0^{\circ}$ is the homogeneous harmonic polynomial associated to the newform $K_{n - 1}$-type $\tau_0^{\circ}$ of $\pi_0$ via \eqref{eq:PcircUn} and \eqref{eq:PcircOn}.
\end{proposition}

\begin{proof}
Just as in \cite[Equation (8.3)]{Jac09}, we insert the propagation formula \eqref{eq:WcircviaW0circ} for $W^{\circ}\begin{psmallmatrix} g & 0 \\ 0 & 1 \end{psmallmatrix}$ into the definition \eqref{eq:RankinSelbergnm} of $\Psi(s,W^{\circ},W^{\prime\circ})$; the absolute convergence of the ensuing double integral is justified in \cite[Section 8.3]{Jac09}. We replace $h$ with $uh$, where now $u \in \Ngp_{n - 1}(F)$ and $h \in \Ngp_{n - 1}(F) \backslash \GL_{n - 1}(F)$, make the change of variables $u \mapsto u^{-1}$, then replace $ug$ with $g$, where now $g \in \GL_{n - 1}(F)$; in doing so, we use the fact that $W_0^{\circ}(uh) = \psi_{n - 1}(u) W_0^{\circ}(h)$ and that $\overline{\psi_{n - 1}}(u) W^{\prime\circ}(g) = W^{\prime\circ}(ug)$. We then make the change of variables $g \mapsto hg$, leading to the identity
\begin{multline*}
\Psi(s,W^{\circ},W^{\prime\circ}) = \int\limits_{\Ngp_{n - 1}(F) \backslash \GL_{n - 1}(F)} W_0^{\circ}(h) \Phi_0^{\circ}(e_{n - 1} h) \left|\det h\right|^s	\\
\times \int_{\GL_{n - 1}(F)} W^{\prime\circ}(hg) \Phi_1(g) \left|\det g\right|^{s + t_1 + \frac{\|\kappa_1\|}{d_F} + \frac{n}{2} - 1} \, dg \, dh,
\end{multline*}
where the standard Schwartz functions $\Phi_1 \in \Ss_0(\Mat_{(n - 1) \times (n - 1)}(F))$ and $\Phi_0^{\circ} \in \Ss_0(\Mat_{1 \times (n - 1)}(F))$ are as in \eqref{eq:Phi1} and \eqref{eq:Phi0circ}. From \eqref{eq:WhittakerPieri}, the integral over $\GL_{n - 1}(F) \ni g$ is equal to
\[L\left(s + t_1 + \frac{\|\kappa_1\|}{d_F},\pi'\right) W^{\prime\circ}(h).\]
From \eqref{eq:isobaricRS}, \eqref{eq:twistedLfunction}, \eqref{eq:L(s,pi)C}, and \eqref{eq:L(s,pi)R1}, we have that
\[L\left(s + t_1 + \frac{\|\kappa_1\|}{d_F},\pi'\right) = L(s,\pi_1 \times \pi').\]
This yields the desired identity upon recalling the definition \eqref{eq:RankinSelbergnn} of $\Psi(s,W_0^{\circ},W^{\prime\circ},\Phi_0^{\circ})$.
\end{proof}

\subsubsection{The Case $\pi_1 = D_{\kappa_1} \otimes \left|\det\right|^{t_1}$}

We next treat the case for which $\pi = \pi_1 \boxplus \cdots \boxplus \pi_r$ with $n_1 = 2$, so that $F = \R$ and $\pi_1 = D_{\kappa_1} \otimes \left|\det\right|^{t_1}$.

\begin{proposition}
\label{prop:(n,n-1)reduction2}
For $n \geq 2$, let $\pi = \pi_1 \boxplus \pi_2 \boxplus \cdots \boxplus \pi_r$ and $\pi_0^{\ast} \coloneqq \pi_1^{\ast} \boxplus \pi_2 \boxplus \cdots \boxplus \pi_r$ be induced representations of Langlands type of $\GL_n(\R)$ and $\GL_{n - 1}(\R)$ with $\pi_1 = D_{\kappa_1} \otimes \left|\det\right|^{t_1}$ and $\pi_1^{\ast} = |\cdot|^{t_1 + \frac{\kappa_1 + 1}{2}}$, and let $\pi' = |\cdot|^{t_1'} \boxplus \cdots \boxplus |\cdot|^{t_{n - 1}'}$ be a spherical representation of Langlands type of $\GL_{n - 1}(\R)$. Let $W^{\circ} \in \WW(\pi,\psi)$ and $W_0^{\ast\circ} \in \WW(\pi_0^{\ast},\psi)$ be the Whittaker newforms of $\pi$ and $\pi_0^{\ast}$ and let $W^{\prime\circ} \in \WW(\pi',\overline{\psi})$ be the spherical Whittaker function of $\pi'$. Then for $\Re(s)$ sufficiently large, the $\GL_n \times \GL_{n - 1}$ Rankin--Selberg integral $\Psi(s,W^{\circ},W^{\prime\circ})$ is equal to
\[\Psi(s,W_0^{\ast\circ},W^{\prime\circ},\Phi_0^{\ast\circ}) L\left(s + t_1 + \frac{\kappa_1 - 1}{2}, \pi'\right),\]
with the standard Schwartz function $\Phi_0^{\circ} \in \Ss_0(\Mat_{1 \times (n - 1)}(\R))$ given by
\[\Phi_0^{\ast\circ}(x) \coloneqq \left(\dim \tau_0^{\ast\circ}\right) \overline{P_0^{\ast\circ}}(x) \exp\left(-\pi x \prescript{t}{}{x}\right),\]
where $P_0^{\ast\circ}$ is the homogeneous harmonic polynomial associated to the newform $\Ogp(n - 1)$-type $\tau_0^{\ast\circ}$ of $\pi_0^{\ast}$ via \eqref{eq:PcircOn}.
\end{proposition}

\begin{proof}
The proof is identical to that of \hyperref[prop:(n,n-1)reduction]{Proposition \ref*{prop:(n,n-1)reduction}} except that we use the Godement section identity \eqref{eq:WcircviaW0astcirc} for $W^{\circ}\begin{psmallmatrix} g & 0 \\ 0 & 1 \end{psmallmatrix}$ in place of \eqref{eq:WcircviaW0circ}.
\end{proof}

\subsection{Proofs of \texorpdfstring{\hyperref[thm:testvector]{Theorems \ref*{thm:testvector}}}{Theorems \ref{thm:testvector}} and {\ref{thm:GLnxGLn}}}

We first record the following uniqueness principle.

\begin{lemma}
\label{lem:GLn-1uniqueness}
Suppose that $W$ is a smooth function on $\GL_{n - 1}(F)$ of moderate growth that satisfies $W(ugk) = \psi_{n - 1}(u) W(g)$ for all $u \in \Ngp_{n - 1}(F)$, $g \in \GL_{n - 1}(F)$, and $k \in K_{n - 1}$. Then if
\[\int\limits_{\Ngp_{n - 1}(F) \backslash \GL_{n - 1}(F)} W(g) W^{\prime\circ}(g) \left|\det g\right|^{s - \frac{1}{2}} \, dg = 0\]
for all $s \in \C$ and spherical representations $\pi'$ of $\GL_{n - 1}(F)$, we must have that $W(g) = 0$ for all $g \in \GL_{n - 1}(F)$.
\end{lemma}

\begin{proof}
This is proved by Jacquet, Piatetski-Shapiro, and Shalika \cite[Lemme (3.5)]{JP-SS81} when $F$ is nonarchimedean; the same proof holds for archimedean $F$ with minimal modifications. Alternatively, one can show this via the Whittaker--Plancherel theorem \cite[Chapter 15]{Wal92}.
\end{proof}

With these results in hand, we may complete the proofs of \hyperref[thm:testvector]{Theorems \ref*{thm:testvector}} and \ref{thm:GLnxGLn}.

\begin{proof}[{Proofs of {\hyperref[thm:testvector]{Theorems \ref*{thm:testvector}}} and {\ref{thm:GLnxGLn}}}]
We prove these theorems by double induction. The base case is the case $n = 1$ of \hyperref[thm:GLnxGLn]{Theorem \ref*{thm:GLnxGLn}}, which is precisely \hyperref[prop:basecase]{Proposition \ref*{prop:basecase}}.

Suppose by induction that \hyperref[thm:GLnxGLn]{Theorem \ref*{thm:GLnxGLn}} holds with $n - 1$ in place in $n$. If $\pi = \pi_1 \boxplus \pi_2 \boxplus \cdots \boxplus \pi_r$ and $\pi_0 \coloneqq \pi_2 \boxplus \cdots \boxplus \pi_r$ with $\pi_1 = \chi^{\kappa_1} |\cdot|^{t_1}$, then \hyperref[prop:(n,n-1)reduction]{Proposition \ref*{prop:(n,n-1)reduction}} and the induction hypothesis imply that
\[\Psi(s,W^{\circ},W^{\prime\circ}) = L(s,\pi_0 \times \pi') L(s,\pi_1 \times \pi').\]
By \eqref{eq:isobaricRS}, this is precisely $L(s,\pi \times \pi')$. Similarly, if $\pi = \pi_1 \boxplus \pi_2 \boxplus \cdots \boxplus \pi_r$ and $\pi_0^{\ast} \coloneqq \pi_1^{\ast} \boxplus \pi_2 \boxplus \cdots \boxplus \pi_r$ with $\pi_1 = D_{\kappa_1} \otimes \left|\det\right|^{t_1}$ and $\pi_1^{\ast} \coloneqq |\cdot|^{t_1 + \frac{\kappa_1 + 1}{2}}$, then \hyperref[prop:(n,n-1)reduction2]{Proposition \ref*{prop:(n,n-1)reduction2}} and the induction hypothesis imply that
\[\Psi(s,W^{\circ},W^{\prime\circ}) = L(s,\pi_0^{\ast} \times \pi') L\left(s + t_1 + \frac{\kappa_1 - 1}{2},\pi'\right).\]
By \eqref{eq:isobaricRS} and \eqref{eq:twistedLfunction}, this is
\[L\left(s + t_1 + \frac{\kappa_1 - 1}{2},\pi'\right) L\left(s + t_1 + \frac{\kappa_1 + 1}{2},\pi'\right) \prod_{j = 2}^{r} L(s,\pi_j \times \pi'),\]
which by \eqref{eq:L(s,pi)R1}, \eqref{eq:L(s,pi)R2}, and \eqref{eq:isobaricRS} is again $L(s,\pi \times \pi')$.

Next, suppose by induction that \hyperref[thm:testvector]{Theorem \ref*{thm:testvector}} holds. Then \hyperref[prop:(n,n)reduction]{Proposition \ref*{prop:(n,n)reduction}} and the induction hypothesis imply that for $\pi' = |\cdot|^{t_1'} \boxplus |\cdot|^{t_2'} \boxplus \cdots \boxplus |\cdot|^{t_n'}$ and $\pi_0' \coloneqq |\cdot|^{t_2'} \boxplus \cdots \boxplus |\cdot|^{t_n'}$,
\[\Psi(s,W^{\circ},W^{\prime\circ},\Psi^{\circ}) = L(s,\pi \times \pi_0') L(s,\pi \times |\cdot|^{t_1'}).\]
By \eqref{eq:isobaricRS}, this is precisely $L(s,\pi \times \pi')$.

Finally, \hyperref[lem:GLn-1uniqueness]{Lemma \ref*{lem:GLn-1uniqueness}} implies the uniqueness of $W^{\circ}$ as a right $K_{n - 1}$-invariant test function for the $\GL_n \times \GL_{n - 1}$ Rankin--Selberg integral.
\end{proof}

\phantomsection
\addcontentsline{toc}{section}{Acknowledgements}
\hypersetup{bookmarksdepth=-1}

\subsection*{Acknowledgements}

I would like to thank Herv\'{e} Jacquet and Akshay Venkatesh for their encouragement as well as useful discussions. Thanks are also owed to Subhajit Jana for helpful clarifications regarding the theory of analytic newvectors.

\hypersetup{bookmarksdepth}

\end{document}